\documentclass{article}

\usepackage[marginparwidth=3cm, marginparsep=.5cm]{geometry}

\usepackage{graphicx} 
\usepackage[utf8]{inputenc}
\usepackage{amsmath,amsthm,amsfonts,amssymb,bm, mathdots}
\usepackage{authblk}
\usepackage{tikz}
\usepackage{ytableau}

\usepackage{mathtools}
\mathtoolsset{showonlyrefs}

\usepackage{caption}
\usepackage{subcaption}
\usepackage{enumerate}
\usepackage{hyperref}

\usepackage{soul}

\theoremstyle{plain}
\newtheorem{thm}{Theorem}[section]
\newtheorem{theorem}{Theorem}[section]
\newtheorem*{thm*}{Theorem}
\newtheorem{prop}[thm]{Proposition}
\newtheorem*{prop*}{Proposition}
\newtheorem{conj}[thm]{Conjecture}
\newtheorem*{conj*}{Conjecture}
\newtheorem{cor}[thm]{Corollary}
\newtheorem{lem}[thm]{Lemma}
\newtheorem*{lem*}{Lemma}

\newtheorem{definition}[thm]{Definition}
\newtheorem{remark}[thm]{Remark}
\newtheorem{ass}[thm]{Assumption}

\newcommand{\ZZ}{\mathbb{Z}}
\newcommand{\RR}{\mathbb{R}}

\newcommand{\onee}{\mathbf{1}}

\renewcommand{\d}{\,\mathrm{d}}
\renewcommand{\i}{\mathrm{i}}

\DeclareMathOperator{\Tr}{Tr}
\DeclareMathOperator{\Log}{Log}

\DeclareMathOperator{\Sign}{sgn}

\DeclareMathOperator{\var}{Var}

\DeclareMathOperator{\adj}{adj}

\DeclareMathOperator{\Cov}{\mathrm{Cov}}

\title{Gaussian Free Field and Discrete Gaussians in Periodic Dimer Models}
\author{Tomas Berggren and Matthew Nicoletti}
\date{}

\begin{document}

\maketitle

\begin{abstract}
    We analyze height fluctuations in Aztec diamond dimer models with nearly arbitrary periodic edge weights. We show that the centered height function approximates the sum of two independent components: a Gaussian free field on the multiply connected liquid region and a harmonic function with random liquid-gas boundary values. The boundary values are jointly distributed as a discrete Gaussian random vector. This discrete Gaussian distribution maintains a quasi-periodic dependence on~$N$, a phenomenon also observed in multi-cut random matrix models.
\end{abstract}

\tableofcontents

\section{Introduction}
\subsection{Motivation and informal description of results}
The dimer model is the study of random perfect matchings, or dimer covers, on finite (or infinite) subgraphs of a bipartite lattice equipped with nonnegative edge weights. In the finite case, the probability of a matching is proportional to the product of the weights of its edges. The dimer model on lattices equipped with a weight-preserving~$\mathbb{Z}^2$ action by translations (i.e. periodic lattices), is an exactly solvable lattice model in statistical mechanics, with origins in chemistry and physics~\cite{Kas61,TF61}. Physically, it models the surface of a crystal in equilibrium~\cite{NHB84}. For a detailed historical overview of periodically weighted dimer models, we refer the reader to Section 1.1 of~\cite{BB23}; our work builds directly on top of the developments of that work (which in turn is a culmination of efforts of many works cited there, including~\cite{CJ16,BD19,Ber21,BD22,DK21}).

We will be analyzing the square lattice dimer model on the \emph{Aztec diamond} (see Figure~\ref{fig:Aztec_ht} for an example) with spatially varying periodic edge weights with nontrivial period in both directions (as in Figure~\ref{fig:Aztec_example}). Each matching can be viewed as domino tiling, and we will characterize the global asymptotic behavior of a random perfect matching via its associated \emph{height function} (shown in Figure~\ref{fig:Aztec_ht}), first introduced for domino tilings by Thurston~\cite{Thu90}. The foundational work~\cite{CKP00} (see also~\cite{KOS06} and~\cite{Kuc17}) establishes the convergence at the large scale of random dimer height functions to a limiting deterministic height function via a variational principle. Results of~\cite{BB23} give explicit formulas describing this so-called \emph{limit shape} in our setting via the computation of new exact formulas for correlation functions and of local dimer statistics asymptotically. The works~\cite{BBS24},~\cite{BB24} generalize both the variational principle and the explicit computation of limit shapes for the Aztec diamond (and the hexagon) to the setting of~\emph{quasi-periodic} weights using algebro-geometric techniques, and these two works also employ a computable uniformization scheme to numerically compute the predicted limit shapes and match these with simulations. Results of~\cite{BdT24} include general exact formulas for correlation functions on the Aztec diamond with quasi-periodic weights.

As illustrated in Figure~\ref{fig:Aztec_sims}, with doubly periodic weights one observes the emergence of three distinct types of macroscopic regions in the domain (\emph{spatial phase separation}). The three types of regions are known as \emph{frozen regions} (near the boundary of the Aztec diamond), the \emph{liquid region} (also called the \emph{rough region}, this is where the tiling looks ``more random''), and \emph{gaseous regions} (also called \emph{smooth regions}, these are the islands in the bulk); these correspond to the three phases of ergodic, translation-invariant Gibbs measures (which describe local dimer statistics away from phase boundaries, or~\emph{arctic curves}) described in~\cite{KOS06}.

Crucially, due to the gaseous facets appearing in our setup, the liquid region is \emph{not simply connected}. Roughly speaking, the liquid region is where the height function fluctuates around its mean more wildly (see Figure~\ref{fig:doubledimer}), and thus one generally restricts attention to this region to extract a nontrivial (and conformally invariant~\cite{Ken99}) scaling limit. For context, when the liquid region \emph{is} simply connected, it is expected (and proven in many cases, e.g. \cite{Ken01,Ken08,BF14,Dui13,Pet15,BK17,BG18,BL18,Rus18,Rus20,Hua20}) that the height fluctuations are described by the pullback of the \emph{Gaussian free field} by a certain diffeomorphism mapping the liquid region to the upper half plane or the disc. This diffeomorphism is sometimes referred to as the \emph{uniformizing map}, since it endows the liquid region the complex structure (known as the \emph{Kenyon-Okounkov complex structure} due to a general prediction in~\cite{KO07}) with which the conformal invariance of the model is to be understood. In settings where the liquid region is instead multiply connected due to the emergence of gaseous facets, height fluctuations have not yet been characterized in any given setup; our goal here is to provide such a characterization in a many-parameter family of examples.

\begin{figure}
    \centering 
\includegraphics[scale=1.25]{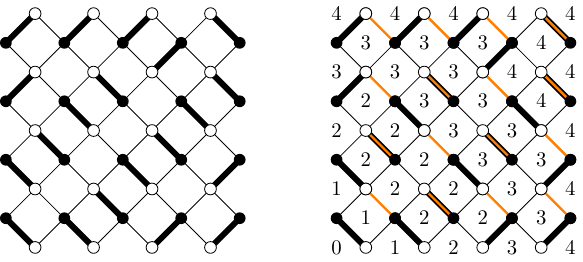}
\caption{A perfect matching of a size~$N =4$ Aztec diamond. On the right we also show the reference matching and the associated height function.}
\label{fig:Aztec_ht}
\end{figure}

In this work, we compute the asymptotic behavior of the height fluctuation field for doubly periodic Aztec diamond dimer models on the \emph{multiply connected} liquid region. We prove that for large size~$N$, the height fluctuations approximate~\textbf{an independent sum of a Gaussian free field and a random harmonic function, whose boundary values on each of the~$g$ liquid-gas boundaries are~$g$ random constants distributed according to an~$N$-dependent discrete Gaussian distribution}. Our results concerning the asymptotic distribution of the height field are all in the sense of moments; in particular, we do not prove convergence at the process level, though we expect that this should be doable. 

In more detail, we define an observable we call the~\emph{discrete component}~$(Z_1,\dots,Z_g)$ as a tuple of~$g$ real numbers, one for each gaseous facet; each entry $Z_i$ is an appropriate spatial average of the height function inside the corresponding facet. We show (Theorem~\ref{thm:main_intro}) that \emph{after subtracting out the harmonic function with boundary values given by the discrete components}, the height function fluctuations converge (in the sense of moments) to a Gaussian free field on a multiply connected domain. The underlying complex structure of the Gaussian free field is described through the \emph{critical point map} introduced in \cite{BB23}, which by results of the same work 
 coincides with the map from the liquid region to the spectral curve defined via the slopes, as described in~\cite[Theorem 1]{KO07}. Moreover, we show that the discrete component~$(Z_1,\dots,Z_g)$ is asymptotically independent (in the sense of moments) from this limiting Gaussian field. Finally, we also identify the 
 joint distribution of the discrete component with that of an~$N$-dependent multivariate discrete Gaussian random variable (Theorem \ref{thm:discrete_gauss_intro}). We ultimately use \emph{Fay's identity} for theta functions~\cite{Fay73} to prove that the joint cumulants of~$(Z_1,\dots,Z_g)$ are expressed in terms of logarithmic derivatives of the theta function associated to the spectral curve. There is a shift in the argument of these theta functions which does not converge; it evolves quasi-periodically in~$\mathbb{R}^g/\mathbb{Z}^g$.

The discrete component should be interpreted as approximating the difference of the global height of each gas region and its expected height, and the fact that the discrete component is random means that the height is not deterministic in the limit. This phenomenon can be compared to dimer models on multiply connected domains, where the heights of the holes are typically random rather than deterministic. Consequently, it is natural to expect an additional component in such settings as well; see~\cite{Gor21} for a heuristic discussion. From the perspective of height functions, however, it is also natural to fix the boundary conditions on the inner boundaries as well, and in that case, a Gaussian free field without an additional component has been observed~\cite{BG19}.

For additional context, we briefly compare our setup to other models where discrete Gaussians appear, or are expected to appear. In the well-known nonintersecting path picture for tilings (the paths are level curves of the height function), paths still move through the gaseous regions, and moreover since there are local fluctuations inside of each gaseous facet, the paths are not completely rigid there. Thus, in our setting, at the combinatorial level there is not a clear separation of the state space into sectors; this is why we must take a spatial average to define~$(Z_1,\dots, Z_g)$. This may be contrasted to other ``higher genus'' dimer models including random tilings of cylinders studied in special setups in~\cite{Ken14,ARV21}, and dimer models on surface graphs studied in~\cite{BT06,DG15,KSW16,BLR24}; in these settings, the topological sectors are clearly defined at the discrete level. In view of these other setups, the results of the present paper say that the gaseous facets effectively introduce distinct topological sectors as the mesh size goes to~$0$, though at the discrete level there are no obvious topological obstructions. In addition, discrete Gaussians (or least theta functions) arise in the early work~\cite{DIZ97} analyzing certain large deviation probabilities in the sine process, in~$\beta$ ensembles in the \emph{multi-cut} setting studied in math papers \cite{Shc13, BG24} and physics papers~\cite{BDE00,Eyn09}, and in several recent developments for two-dimensional random point processes in multi-component regimes including~\cite{ACC22,ACCL24,Cha24,AC24}.

We end this subsection with a brief remark about universality. In each of the works cited in the previous paragraph, the discrete Gaussian distribution is either proven or conjectured to describe the ``topologically nontrivial'' component of the fluctuation field (for a field-theoretic interpretation, see the discussion about the \emph{compactified free field} in~\cite{Dub15}). In hindsight, we believe that our result provides another class of examples supporting the universality of the discrete Gaussian distribution in random (possibly multivalued) height function models taking place on topologically nontrivial domains. Discrete Gaussian random variables are uniquely entropy-maximizing (in the sense of Shannon entropy) among the restricted class of random variables with support in~$\mathbb{Z}^g$ which have a specified mean and covariance matrix~\cite{AA19}. Since entropy considerations can be used to prove the classical central limit theorem (\cite{Lin59}, \cite{B86}, \cite{GK24}), it is maybe not surprising from this perspective that the discrete Gaussian distribution seems to universally describe the ``topological part'' of fluctuations for higher genus models with conformally invariant scaling limits.

\begin{figure}
    \centering 
\includegraphics[scale=.9 ]{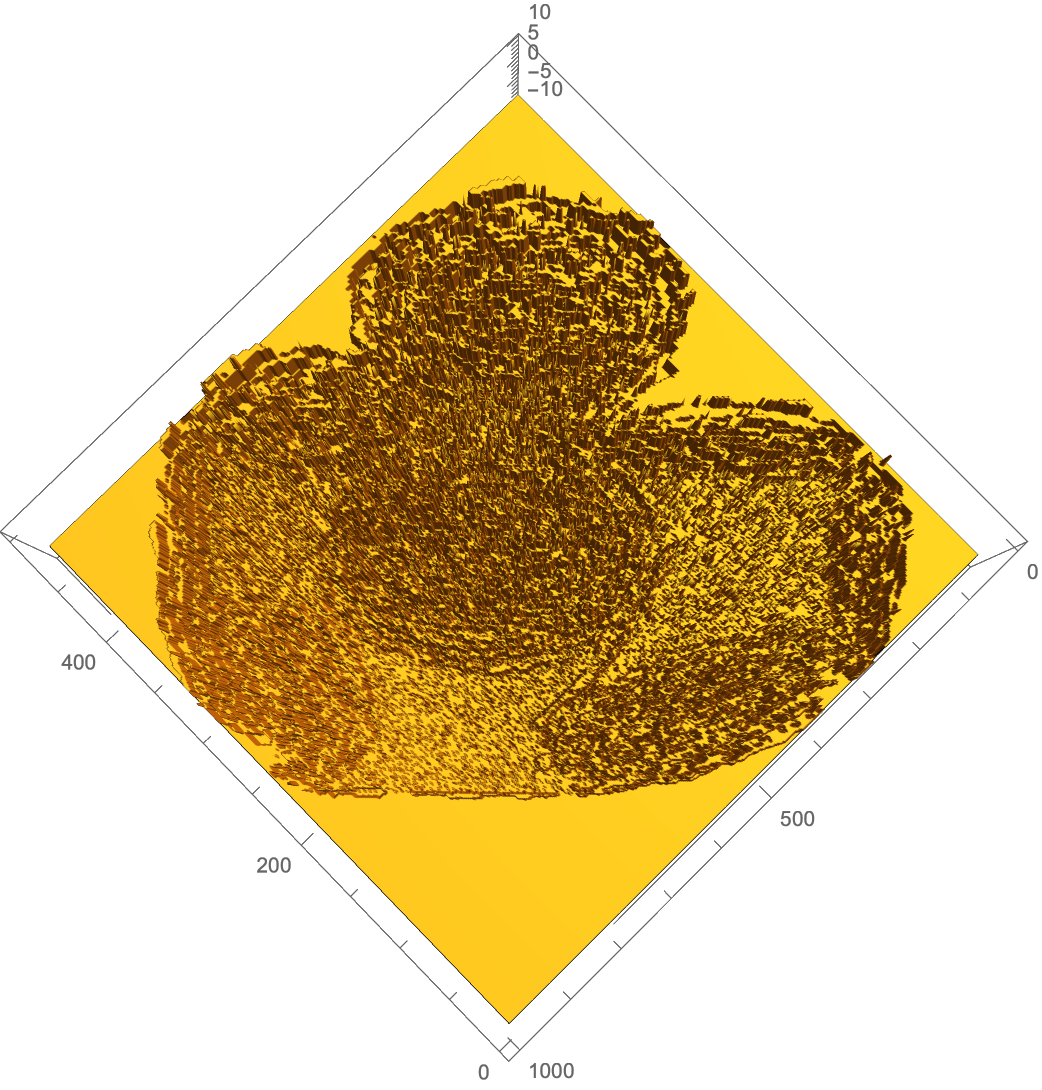}
\caption{The difference of two independent height functions sampled from a dimer model with~$3 \times 3$ periodic weights, with four gaseous facets. (One of the four facets is difficult to see in the picture.) One observes qualitatively more wild behavior in the liquid region, while gaseous facets are approximately flat, though they exhibit (relatively) sparse defects.}
\label{fig:doubledimer}
\end{figure}

\begin{figure}
    \centering 
\includegraphics[scale=.42 ]{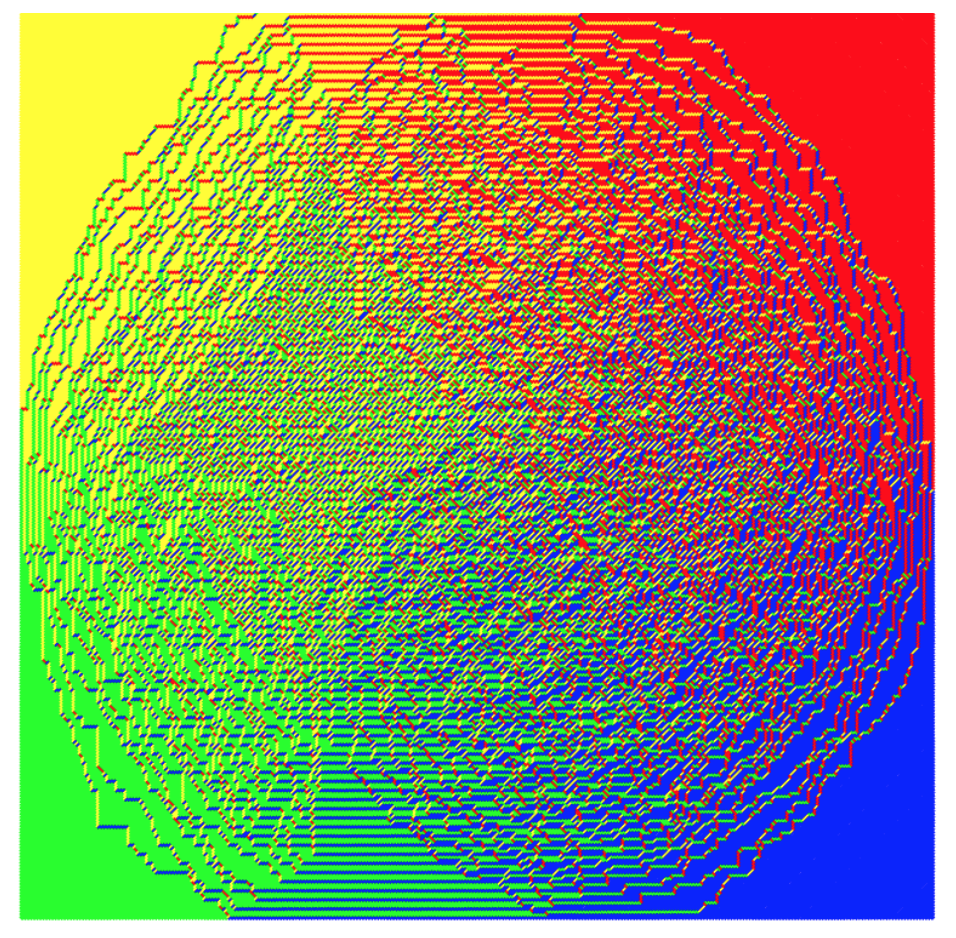}
\includegraphics[scale=.42 ]{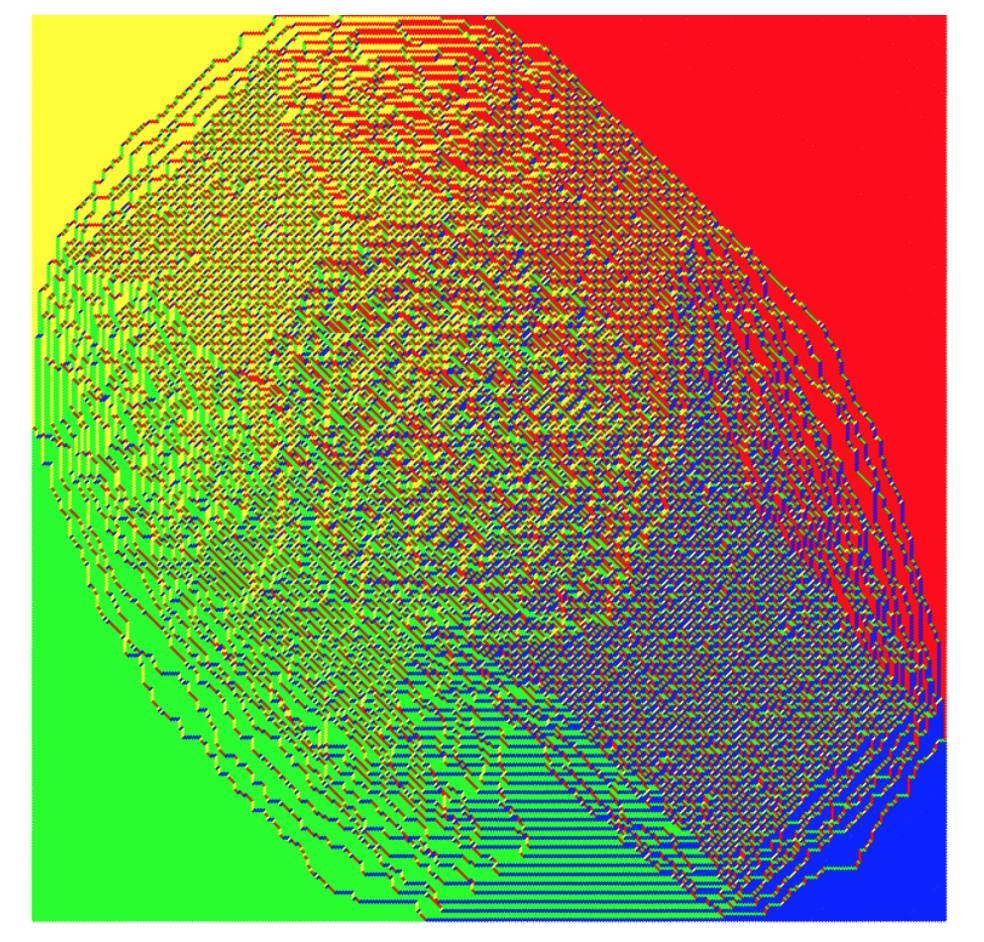}
\caption{Two random samples of domino tilings with doubly periodic weights. The liquid region is connected, but not simply connected.}
\label{fig:Aztec_sims}
\end{figure}


\subsection{Statement of results}
\label{subsec:res}

To state our results, let us give some setup. We study random perfect matchings, or dimer covers, of the size~$k \ell N$ Aztec diamond weighted proportionally to the product of their edge weights, which we take to be \emph{doubly periodic}. The positive integers $k$ and~$\ell$ are the vertical and horizontal periods of the weights, respectively; see Figure~\ref{fig:Aztec_example} for an illustration of the edge weights. Throughout this work we assume only minor technical conditions on the edge weights (essentially, that the spectral curve has maximal genus), given in Assumption~\ref{ass:1}; generic periodic edge weights satisfy our assumptions. In particular, we omit Assumption 4.1 (b) and the ``distinct angles'' assumption in 4.1 (c) in~\cite{BB23}.

Denote by~$h_N$ the dimer model height function defined on the faces of the size~$ k \ell N$ Aztec diamond; while crossing an edge with the white vertex on the right, the height changes by~$+1$ if an edge of the matching is crossed, and by~$-1$ if an edge of the reference matching is crossed, and by~$0$ if both or neither are present. See Figure~\ref{fig:Aztec_ht} for an example, and see Section~\ref{subsec:comb} for a detailed definition. We choose \emph{rescaled}, or~\emph{macroscopic},  coordinates~$(\xi, \eta)$ on fundamental domains of the square lattice such that for large~$N$ the Aztec diamond converges to the smooth domain~$[-1,1]^2$. By~\cite{CKP00, Kuc17}, the rescaled height function converges in probability to a deterministic \emph{limit shape}~$\overline{h}$,
$$\frac{1}{k \ell N}h_N(\lfloor k \ell N \xi \rfloor, \lfloor k \ell N \eta \rfloor) \rightarrow \overline{h}(\xi,\eta).$$
 A full description of the conventions for microscopic and macroscopic coordinates can be found in Section~\ref{subsec:comb}. 

We analyze the next order term in the large~$N$ expansion of $h_N(\lfloor k \ell N \xi \rfloor, \lfloor k \ell N \eta \rfloor)$. In order to extract a limit, we study the random function~$h_N - \mathbb{E}[h_N]$ (with no rescaling, as in previously studied dimer models) in the liquid region~$\mathcal{F}_R$ (the subscript~$R$ stands for ``rough''; we use this notation to stay consistent with~\cite{BB23}); this will converge to a random object living in an appropriate space of distributions (or, generalized functions; it will not be defined pointwise). It follows from~\cite{BB23} that~$\mathcal{F}_R$ is diffeomorphic, via the \emph{critical point map}~$q(\xi, \eta)$ (which is denoted in that work by~$\Omega(\xi,\eta)$, and is a distinguished critical point of a certain \emph{action function}), to the interior of the ``top half''~$\mathcal{R}_0$ of a compact Riemann surface~$\mathcal{R}$. We briefly elaborate on this in the following paragraph, and we review the definition and properties of~$\mathcal{R}$ and~$\mathcal{R}_0$ in more detail in Section~\ref{subsec:spectral}.

The surface~$\mathcal{R}$ is nothing other than the compactification of the Harnack curve associated to a periodic dimer model as described in~\cite{KOS06}; it is a certain closure of the 
 zero set of the determinant of the magnetically altered Kasteleyn matrix~$P(z,w) = \det K(z,w)$,
 $$\mathcal{R} = \overline{\{P(z,w) = 0\}}.$$
 The set of real points~$(z,w)\in \mathbb{R}^2$ on the surface splits the surface into two halves, and~$\mathcal{R}_0$ is a distinguished half. Moreover, for~$(\xi, \eta) \in \mathcal{F}_R$ the critical point map can be described as follows. Given~$(\xi, \eta) \in \mathcal{F}_R$, the slopes of the limit shape~$(s, t) = \nabla \overline{h}|_{(\xi,\eta)}$ are liquid phase slopes, meaning~$(s,t)$ is in the interior of the Newton polygon of~$\mathcal{R}$, away from interior lattice points; there is a unique point~$q = q(\xi, \eta) \in \mathcal{R}_0$ corresponding to this pair of slopes via the natural map which identifies the interior of~$\mathcal{R}_0$ with the set of all liquid phase slopes, and this~$q(\xi, \eta)$ is equal to the critical point (this is Equation (4.16) in \cite{BB23}; as explained there, the gradient of~$\overline{h}$ must be taken with respect to slightly different coordinates, but it is not important for us here).

For conceptual clarity only, we remark that topologically, and (by the Koebe uniformization theorem) as a complex manifold,~$\mathcal{R}_0$ can be viewed as the unit disc with~$g$ round (circular) holes cut out of it, where~$g$ is the genus of~$\mathcal{R}$. By known theory and our genericity assumptions on edge weights,~$g = (\ell-1)(k-1)$. The~$g$ inner holes in~$\mathcal{R}_0$ are called \emph{compact ovals} and we denote them by~$A_1,\dots, A_g$; the outer boundary component (corresponding to the unit circle) is the \emph{outer oval} and we denote it by~$A_0$.

It is conjectured that for more general boundary conditions (and other periodic lattices), an appropriate analog of the critical point mapping gives the correct coordinates to understand the conformal invariance of the fluctuations~\cite{KO07}. In our case, $q : \mathcal{F}_R \rightarrow \mathcal{R}_0$ is the uniformizing map, while in more general settings, this map is a finite degree cover of~$\mathcal{R}_0$, as explained in~\cite{KO07} and illustrated in~\cite{BB24} in the setting of quasi-periodic weights on the hexagonal lattice. In any case, in the simply connected setup (as in, e.g.,~\cite{BF14} or \cite{Pet15}) the image of this map is usually taken to be the upper half plane or the unit disc, so from the perspective of the uniformizing map,~$\mathcal{R}_0$ (a unit disc with~$g$ holes) in our setting plays the role that the unit disc plays in the simply connected setup. As previously mentioned, in many genus zero dimer models, where~$\mathcal{F}_R$ is simply connected, the convergence of the height fluctuations to a pullback of the Gaussian free field via the uniformizing map has been obtained. By universality considerations, one expects to observe a Gaussian free field on~$\mathcal{R}_0$ in our setting. However, essentially because in our setting~$\mathcal{F}_R$ is multiply connected, the field of fluctuations turns out to have an additional non-Gaussian (but still conformally invariant) component that we describe next, which asymptotically is described by a \emph{discrete Gaussian distribution}.

\subsubsection{Convergence to the Gaussian free field}
Define ~$\mathfrak{g}_{\mathcal{R}_0}$ as the Gaussian free field (GFF) with Dirichlet boundary conditions on~$\mathcal{R}_0$. This can be defined as the pullback of a Dirichlet GFF on the unit disc with~$g$ circular holes by a conformal isomorphism. Alternatively, it can be defined directly from the conformal structure on~$\mathcal{R}_0$. Below we informally state some of its properties from this latter perspective, and refer the reader to Section~\ref{subsec:Gff} for a slightly more detailed review of the GFF on~$\mathcal{R}_0$. 

The GFF is not well-defined pointwise as a function, but it should be thought of as a Gaussian process, and its joint moments at tuples of pairwise distinct points can still be defined via the Green's function and the Wick rule. Namely, if~$\mathcal{G}_{\mathcal{R}_0}$ denotes the Green's function of the Laplacian with Dirichlet boundary conditions on~$\mathcal{R}_0$ (the Green's function is well-defined from the conformal structure alone, e.g., by choosing any Riemannian metric compatible with the conformal structure), we have, with~$\{q_j\}_{j=1}^r$ denoting~$r \geq 2$ distinct points in~$\mathcal{R}_0$,
\begin{equation}\label{eqn:intro_gff}
\mathbb{E}\left[ \prod_{j=1}^r \mathfrak{g}_{\mathcal{R}_0} (q_j) \right] =
\begin{cases}
   \sum_{pairings \; \sigma} \prod_{i=1}^{r/2} \mathcal{G}_{\mathcal{R}_0}(q_{\sigma(2i-1)},q_{\sigma(2 i)}) & r \text{ even} \\
   0 & \text{otherwise}
\end{cases}.
\end{equation}

As mentioned earlier, in the large size limit, we observe, in addition to the GFF, a discrete Gaussian random variable. To capture this, we define the~\emph{discrete components}, which are defined precisely in Section~\ref{subsec:DCdist}. 
In the scaling limit of the model, there are~$g = (k-1)(\ell-1)$ gaseous facets in the rescaled Aztec diamond; these are the regions where local fluctuations are in the gaseous phase~\cite{BB23,KOS06}. Define~$Z_i$,~$i=1,\dots,g$, as the spatial average of the height fluctuation~$h_N - \mathbb{E}[h_N]$ over a subset of faces sampled from 
a large (growing to infinity) set of faces in the interior of the~$i$th gaseous facet; see Figure~\ref{fig:dc_grids} for an illustration. We sample the faces such that adjacent samples are at a mesoscopic (with respect to the lattice mesh) distance away from each other. The exact mesoscopic scale is not important, though we must be mindful of it for technical reasons. We point out that~$Z_i$ depends on~$N$, i.e. on the size of the Aztec diamond being sampled, though we hide the dependence in the notation. This is crucial, as the~$N$-dependence will in fact~\emph{not} completely wash away in the large~$N$ limit, as we explain in Theorem~\ref{thm:discrete_gauss_intro} below.

 For the purposes of stating the first main result, we define the function
\begin{equation}\label{eqn:intro_fidef}
    f_i : \mathcal{R}_0 \rightarrow \mathbb{R} \qquad i=1,\dots, g
\end{equation}
as the unique function satisfying
\begin{align}\label{eqn:intro_fidef2}
   & \partial \overline{\partial} f_i = 0 \qquad \text{in } \mathcal{R}_0  \\ 
   & f_i|_{\partial \mathcal{R}_0}(q) = \begin{cases}
       1 & q \in A_i\\
 0 & \text{ otherwise}       
       \end{cases}.
\label{eqn:intro_fidef3}
\end{align}
Note~$\sum_{i=1}^g Z_i f_i(q)$ is the unique harmonic function which takes value~$Z_i$ on the compact oval~$A_i$ and~$0$ on the outer oval~$A_0$ and is what we need to subtract from the height function to see the GFF. See Remark~\ref{rem:holomorphic_form} for a brief explanation of why these harmonic functions naturally appear.

In the theorem below, we consider a tuple of pairwise distinct  positions~$(\xi_j, \eta_j) \in \mathcal{F}_R$,~$j=1,\dots, r$, and for each~$j$ we suppose we have a sequence of faces~$\mathsf{f}_{j, N}$ in the Aztec diamond such that macroscopic coordinates~$(\xi_{j, N}, \eta_{j, N})$ of the face~$\mathsf{f}_{j, N}$ converge,~$(\xi_{j, N}, \eta_{j, N}) \rightarrow (\xi_j, \eta_j)$. Recall once more our notation~$q : \mathcal{F}_R \rightarrow \mathcal{R}_0 $ for the diffeomorphism given by the critical point map. Let us define, for any face~$\mathsf{f}$ with macroscopic coordinates~$(\xi_{\mathsf{f}}, \eta_{\mathsf{f}})$, with~$h_N$ denoting the height function and~$(Z_1,\dots, Z_g)$ denoting the discrete component defined above, 
   \begin{equation}
\label{eqn:tildeh_def}
\tilde{h}_N(\mathsf{f}) \coloneqq      h_N(\mathsf{f}) - \mathbb{E}[h_N(\mathsf{f})] - \sum_{i=1}^g Z_i f_i (q(\xi_{\mathsf{f}}, \eta_{\mathsf{f}})).
\end{equation}
Our main theorem states that~$\tilde{h}_N$ converges in distribution, in the sense of moments, to the pullback by~$q : \mathcal{F}_R \rightarrow \mathcal{R}_0$ of the Gaussian free field~$\mathfrak{g}_{\mathcal{R}_0}$, and that~$\tilde{h}_N$ and~$(Z_1,\dots,Z_g)$ are asymptotically independent as~$N\rightarrow \infty$.

\begin{theorem}\label{thm:main_intro}
    Let~$h_N$ denote the height function of a random dimer configuration of a size~$k \ell N$ Aztec diamond with~$k \times \ell$ doubly periodic edge weights, and let~$\tilde h_N$ be defined as in~\eqref{eqn:tildeh_def} above. 
    
    For any positive integer~$r \geq 2$, consider faces~$\{\mathsf{f}_{j,N}\}_{j=1}^r$ as described above. Then, we have the convergence of moments (Theorem~\ref{thm:multipt_conv} in the text)
   \begin{equation}\label{eqn:intro_heightconv}
   \mathbb{E}\left[ \prod_{j=1}^r \tilde{h}_N(\mathsf{f}_{j, N})\right]   \rightarrow 
   \frac{1}{\pi^{\frac{r}{2}}} \mathbb{E}\left[ \prod_{j=1}^r \mathfrak{g}_{\mathcal{R}_0} (q(\xi_j, \eta_j)) \right] .
   \end{equation}
   
Moreover,~$\tilde{h}_N$ and~$(Z_1,\dots, Z_g)$ are asymptotically independent in the sense of moments: For~$\{\mathsf{f}_{j,N}\}_{j=1}^r$ as above and any nonnegative integers~$n_1,\dots, n_g$, we have (Proposition~\ref{prop:ind_moments} in the text)
\begin{equation}\label{eqn:intro_0moments}
       \left| \mathbb{E}\left[\prod_{m=1}^r\tilde{h}_N(\mathsf{f}_{j,N})  \prod_{j=1}^g Z_j^{n_j}\right] - \mathbb{E}\left[\prod_{m=1}^r\tilde{h}_N(\mathsf{f}_{j,N}) \right] \mathbb{E}\left[\prod_{j=1}^g Z_j^{n_j}\right] \right|  \rightarrow 0 .
    \end{equation}
\end{theorem}

Note that at this stage, the discrete components~$(Z_1,\dots,Z_g)$ could in principle be identically~$0$ in the large~$N$ limit. However, we show that this is not the case. In fact, we compute asymptotics of arbitrary joint moments of~$(Z_1,\dots, Z_g)$, and we identify the limiting moments with those of a~\emph{discrete Gaussian} distribution. This is a Gaussian random variable conditioned to take values in a lattice.

\subsubsection{Convergence to the discrete Gaussian distribution}

We briefly define the discrete Gaussian random variable here; see Section~\ref{subsec:discGauss} for a more detailed discussion of discrete Gaussians. The parameters of a discrete Gaussian are the symmetric \emph{scale matrix}~$\tau \in \i \mathbb{R}^{g \times g}$, which must be pure imaginary with positive definite imaginary part, and the \emph{shift}~$e \in \mathbb{R}^g$. The associated~$g$-dimensional discrete Gaussian probability mass function~$\mathbb{P}_{e,\tau}$ is supported on~$\mathbb{Z}^g$, and for~$n \in \mathbb{Z}^g$ it is defined by
\begin{equation}\label{eqn:discrete_gauss_def_intro}
    \mathbb{P}_{e,\tau}(n) = \frac{1}{C} \exp\left( \i \pi (n- e) \cdot \tau (n-e)\right) 
\end{equation}
where~$C$ is a normalization constant given by~$C = \theta(- \tau e ; \tau) \exp(\i \pi e \cdot \tau e)$, with the theta function~$\theta$ defined as in~\eqref{eqn:theta_def} below. It is straightforward to check from the definition of the theta function that this distribution can equivalently be characterized by its moment generating function as follows; if~$X = (X_1,\dots, X_g)$ is distributed according to~$\mathbb{P}_{e,\tau}$, then
\begin{equation}\label{eqn:dc_mgf_intro}
    \mathbb{E}\left[e^{(2\pi \i) z \cdot X }\right] =
\frac{\theta(z - \tau e ; \tau)}{\theta(- \tau e ; \tau) }, \qquad z \in \mathbb{C}^g.
\end{equation}
Clearly, the moment generating function of~$X - \mathbb{E}[X]$, whose coefficients are centered joint moments of~$X$, is given by
\begin{equation}\label{eqn:dc_mgf_intro_ctrd}
    \mathbb{E}\left[e^{(2\pi \i) z \cdot (X-\mathbb{E}[X]) }\right] =
\frac{\theta(z - \tau e ; \tau) e^{c \cdot z}}{\theta(- \tau e ; \tau) }, 
\end{equation}
where~$c = -\frac{1}{\theta(-\tau e; \tau)} \nabla \theta (- \tau e; \tau) $.

It turns out that even at leading order, the asymptotic distribution of~$(Z_1,\dots, Z_g)$ (which is discrete Gaussian, as we explain below) retains an~$N$ dependence via the shift parameter  in~\eqref{eqn:discrete_gauss_def_intro}, which is given by~$e = e_{\mathrm w_{0,0}}^{(k N)}$; this quantity is defined precisely in Section~\ref{subsec:theta_prime}, and the notation follows~\cite{BB23}. Since our discrete component has mean zero, we will take our shift~$e = e_{\mathrm w_{0,0}}^{(k N)}$ as an element of~$ \mathbb{R}^g/\mathbb{Z}^g$; note that changing~$e$ by an element of~$\mathbb{Z}^g$ in~\eqref{eqn:discrete_gauss_def_intro} does not affect mean-subtracted statistics of a~$\mathbb{P}_{e,\tau}$-distributed random variable (which are described by~\eqref{eqn:dc_mgf_intro_ctrd}). The discrete time evolution~$N \mapsto e_{\mathrm w_{0,0}}^{(k N)} \in \mathbb{R}^g/\mathbb{Z}^g$ is exactly the linear flow on the Jacobi variety studied in~\cite[Section 5]{BB23}. By Remark 5.6 of that work, this quantity can also be characterized in terms of the limit shape. Here only and not throughout the rest of the paper, we make use of slightly different continuum coordinates~$(u, v) = -\frac{1}{k \ell N}(x, y)$ (here~$(x,y) \in \mathbb{Z}^2$ are coordinates indexing fundamental domains, as explained in Section~\ref{subsec:comb}). In these coordinates, the limit shape~$\overline{h}(u, v) $ is linear in the~$i$th gaseous facet with \emph{integer slopes}~$n_i,n_i'$. Define~$(u_i,v_i)$ as any point in the~$i$th facet, and define~$H_i(u_i,v_i) \coloneqq \overline{h}(u_i, v_i) - (u_i n_i + v_i n_i')$, which clearly is independent of the choice of~$(u_i,v_i)$. Up to a fixed overall constant~$e_{\mathrm w_{0,0}}^{(0)}\in \mathbb{R}^g/\mathbb{Z}^g$, determined from the spectral data introduced in~\cite{KO06}, see Equation~\eqref{eq:flow_plus_spectral_data}, we have
\begin{equation}\label{eqn:edef}
e_{\mathrm w_{0,0}}^{(k N)} = k \ell N (H_1(u_1,v_1) ,\dots, H_g(u_g,v_g)) + e_{\mathrm w_{0,0}}^{(0)} \quad  \text{ mod } \mathbb{Z}^g.
\end{equation}
In this way, the~$N$ dependence is described by a linear evolution in~$\mathbb{R}^g / \mathbb{Z}^g$.

Before stating the theorem, we must also define the scale matrix appearing in our theorem, which will play the role of~$\tau$ in~\eqref{eqn:discrete_gauss_def_intro}; it is given in terms of the \emph{period matrix}~$B$ of the~\emph{spectral curve}~$\mathcal{R}$. As reviewed in Section~\ref{subsec:theta_prime}, the period matrix of the Harnack curve~$\mathcal{R}$ is pure imaginary, symmetric, and has positive definite imaginary part. Therefore,~$-B^{-1}$, satisfies the same three properties. In the theorem below we assume for definiteness that the~\emph{Abel map} (see Section~\ref{subsec:theta_prime}) is normalized so that~$u(p_{\infty,1}) = 0$; otherwise,~$e_{\mathrm w_{0,0}}^{(k N)}$ should be replaced by~$e_{\mathrm w_{0,0}}^{(k N)} + u(p_{\infty,1})$ in the theorem. Now we may state our theorem characterizing asymptotically the distribution of the discrete component.

\begin{theorem}[Corollary~\ref{eqn:discrete_gauss_corg} in the text] \label{thm:discrete_gauss_intro}
Let~$B$ be the period matrix of the spectral curve~$\mathcal{R}$. At leading order as~$N \rightarrow \infty$, the joint moments of the discrete component~$Z = (Z_1,\dots,Z_g)$ match those of a~\emph{discrete Gaussian distribution} with shift parameter~$e_{\mathrm w_{0,0}}^{(k N)} \in \mathbb{R}^g/\mathbb{Z}^g$ and scale parameter~$-B^{-1}$: If~$n_1,n_2,\dots, n_g \in \mathbb{Z}_{\geq 0}$, 
\begin{equation}\label{eqn:DC_leading_moments_intro}
\mathbb{E}\left[\prod_{j=1}^g Z_j^{n_j} \right] = \frac{1}{(2\pi \i)^{n_1+\cdots +n_g}}\partial_{z_1}^{n_1}\cdots \partial_{z_g}^{n_g} \frac{\theta(z + B^{-1} e_{\mathrm w_{0,0}}^{(k N)} ; -B^{-1}) \exp(c \cdot z)}{\theta(B^{-1} e_{\mathrm w_{0,0}}^{(k N)} ; -B^{-1}) } |_{z=0}  + o(1),
\end{equation}
   where~$c = c(N) = -\frac{1}{\theta(B^{-1} e_{\mathrm w_{0,0}}^{(k N)}; -B^{-1})} \nabla \theta (B^{-1} e_{\mathrm w_{0,0}}^{(k N)}; -B^{-1}) $.
\end{theorem}

The covariance matrix of a discrete Gaussian random variable is positive definite, see Remark~\ref{rmk:covarm}. In particular, by the above theorem each~$Z_i$ satisfies~$\var[Z_i] > 0$ for~$N$ large enough.

In general, the shift~$e_{\mathrm w_{0,0}}^{(k N)}$ is not periodic in~$N$. However, along any subsequence~$N_n$ such that~$e_{\mathrm w_{0,0}}^{(k N_n)}$ converges, Theorem~\ref{thm:discrete_gauss_intro} implies convergence in distribution of~$(Z_1,\dots,Z_g)$ to a mean-subtracted discrete Gaussian. In Section~\ref{subsec:2by2_ex}, we explicitly compute~$e_{\mathrm w_{0,0}}^{(k N)}$ and explain how to compute~$B$ in the one parameter genus~$1$ model studied in~\cite{CJ16} (and first studied in~\cite{FSG14},~\cite{CY14}) in order to provide a concrete application of Theorem~\ref{thm:discrete_gauss_intro}; see in particular Corollary~\ref{cor:discrete_gauss_a}. 

We observe the shift parameter~$e = e_{\mathrm w_{0,0}}^{(kN)}$ via our asymptotic analysis of correlation functions, and this is ultimately traced back to the fact that the finite~$N$ correlation functions are described by the same shift, whose evolution in~$N$ is the linearization of the integrable discrete dynamical system analyzed in Section 5 of~\cite{BB23}. The fact that finite~$N$ correlation functions involve the same shift, in turn, may be viewed as a consequence of the integrability of the model. 


\subsubsection{Informal statement of the main result}
 Heuristics from several physics papers, as well as rigorous mathematical results, which apply to various models exhibiting the same universal behaviors, seem to imply that the scale matrix and the shift parameter should be related to asymptotic expansions of certain refined partition functions of the model. In Section~\ref{subsec:heur} we informally explain an adaptation these arguments to our setting. In that section we also match our scale matrix, and partially match our shift, to predictions coming from these more general heuristics. 
 
 Another way to informally phrase the results of the previous two theorems is as follows. Let~$\Delta$ denote the Laplace operator on~$\mathcal{R}_0$ which, after choosing a concrete Riemannian metric on~$\mathcal{R}_0$ compatible with the conformal structure, maps functions to functions. In particular, the restriction of~$-\Delta$ to~$L^2$ functions on~$\mathcal{R}_0$ satisfying Dirichlet boundary conditions has a discrete spectrum, which we denote by~$0 < \lambda_1 < \lambda_2 < \cdots$ with corresponding eigenfunctions~$\varphi_1,\varphi_2,\cdots$. Denoting by~$\{\mathcal{N}_j\}_{j\geq1}$ a sequence of i.i.d. standard normals, which are independent of~$Z_1,\dots, Z_g$ (whose joint distribution is an~$N$-dependent discrete Gaussian as in Theorem~\ref{thm:discrete_gauss_intro}), we have the approximate identity in distribution, in an appropriate space of generalized functions on~$\mathcal{F}_R$ (with~$q : \mathcal{F}_R \rightarrow \mathcal{R}_0$ the uniformizing map and~$f_i$ as in~\eqref{eqn:intro_fidef2} and~\eqref{eqn:intro_fidef3})
\begin{equation}\label{eqn:intro_heightconv2}
    h_N - \mathbb{E}[h_N] \approx \sum_{j=1}^\infty \mathcal{N}_j \frac{1}{\sqrt{\lambda_j}} \varphi_j\circ q + \sum_{i=1}^g Z_i f_i \circ q.
\end{equation}
We remark that~\eqref{eqn:intro_heightconv2} is only an informal statement, in the following sense. We expect that it should not be difficult to extend our result to obtain the joint convergence in distribution of the pairing of~$\tilde{h}_N$ with finitely many test functions to the corresponding random Gaussian vector associated to the GFF (along the lines of~\cite[Theorem 5.6]{BF14}), however we have not pursued this here.

\subsection{Proof outline}

Our paper essentially consists of two pieces; an analytic part, where we perform asymptotic analysis, and an algebraic part, where we analyze and simplify the closed form expressions we obtain in the first part. The starting point of our analysis is one of the main results of~\cite{BB23}, which is a collection of exact contour integral formulas for entries of the inverse Kasteleyn matrix; the formulas involve a double contour integral of essentially explicit meromorphic forms on the spectral curve~$\mathcal{R}$.

\paragraph{For the analytic part,} the essential computation of leading order joint moments of the height function is given in Section \ref{subsec:limiting_hm}, and follows the general scheme of several previous works, which was first brought to fruition in~\cite{BF14}. In the following paragraphs, we briefly review this general scheme, and then attempt to explain the new components in our work.

As input to the proof in Section \ref{subsec:limiting_hm}, we must compute either the leading order terms (along with error estimates), or bounds, for entries of the inverse Kasteleyn when either of the vertices is in any of four different regimes: \eqref{item:first} in the bulk (in the liquid region away from the arctic boundary), \eqref{item:second} near the edge (in the liquid region but near the arctic boundary), \eqref{item:third}  exactly at the edge, and \eqref{item:fourth} inside of a frozen or gaseous facet. (See Definition~\ref{def:regimes} in the text.) The estimates for neighboring regions must glue together in a small overlapping region in order to be used in the proof in Section \ref{subsec:limiting_hm}. 

 The computation in Section~\ref{subsec:limiting_hm} then proceeds by summing up height increments along dual paths in the Aztec diamond and substituting the steepest descent estimates into the determinantal formulas for correlation functions. Then, after simplifying the leading order contributions and observing that they form a Riemann sum for an iterated contour integral over~$\mathcal{R}$, we also must bound the error terms, which include parts of paths near arctic boundaries or in facets; ultimately, we prove that the parts of dual paths inside of gaseous facets do not contribute, so only parts in the liquid region contribute.

One aspect distinguishing the analytic part of our proof from previous works is that the steepest descent arguments used to obtain these estimates take place entirely on the (arbitrary genus) spectral curve~$\mathcal{R}$. In our steepest descent arguments, we exploit the useful observation of~\cite{BB23} that viewing contours as subsets of the amoeba of~$\mathcal{R}$ immediately clarifies how to deform the contours, independently of the genus of~$\mathcal{R}$. When each of the (macroscopically far apart) vertices are either in the liquid region or inside of a facet, we use the contour deformation arguments of~\cite{BB23}; when at least one vertex is near the edge (which requires a contour deformation not covered in that work), we modify and adapt those arguments in order to deform the contours. The remaining steps to obtain estimates involve obtaining certain bounds and performing local computations. Even though these remaining steps involve mostly local arguments, we must be careful to adapt and generalize arguments from previous works to our setting (which involves more parameters and a higher genus curve, and thus is less explicit than settings considered before). 


In addition, the necessity of a separate analysis in the fourth regime~\eqref{item:fourth} mentioned above, ultimately due to gaseous facets, is new to our work. In particular, it turns out that bounding contributions from the single integral terms (or, the ``diffusive'' terms) in the formula for the inverse when both points are inside of the same facet requires special attention. This is done in the proof of Lemma~\ref{lem:facet_bound}, which appears in Section~\ref{sec:steepest_arguments}.

\paragraph{The algebraic step} involves a manipulation of certain closed form expressions; these arguments occur in Section~\ref{sec:GFF_DC}. We first simplify the integrand appearing in the formula for joint moments we obtain in Section~\ref{subsec:limiting_hm}. Then, we identify the moments of the height field, after subtracting discrete components, with those of a Gaussian free field, and we identify the 
 joint distribution of the discrete components themselves.

 In more detail, from our asymptotic analysis we derive iterated contour integral expressions (over contours in the spectral curve) for the leading order term of an arbitrary joint moment of~$(Z_1,\dots,Z_g)$ and some number of height function values at distinct faces. The integrand of the contour integral is a determinant of a certain kernel (the determinant is a one form in each variable). We first simplify this kernel, giving a relatively simple expression for it in terms of prime forms and theta functions, by analyzing its poles and zeros on the surface. Using this expression, we show that the two point function of~$\tilde h_N$ is the Green's function on~$\mathcal{R}_0$.

Then, we must prove the Wick rule for higher moments of~$\tilde h_N$, and characterize the distribution of~$(Z_1,\dots, Z_g)$. We initially attempted to generalize the pioneering computation in~\cite[Lemma 3.1]{Ken01}. However, in higher genus, the situation is slightly more subtle, and this can be traced back to the following fact: In genus~$0$, a holomorphic one form is necessarily identically~$0$, a fact which fails to hold in genus~$\geq 1$. Ultimately, we observe that passing from moments to cumulants is extraordinarily useful; it changes the integrand from a determinant into a sum over only the permutations which consist of a single cycle of maximal length (we record this result as Lemma~\ref{lem:momcum} for simplicity, though we believe it should already exist in some form in the literature on determinantal point processes). Using this, together with a careful analysis of some integral expressions, we are able to show that higher cumulants of~$\tilde{h}_N$ (defined above Theorem~\ref{thm:main_intro}) vanish as~$N \rightarrow \infty$. With similar arguments, we are able to show the independence of~$\tilde h_N$ and $(Z_1,\dots,Z_g)$.

 Finally, we then move on to identify the cumulants of~$(Z_1,\dots,Z_g)$ with those of the discrete Gaussian. It is not difficult to see that the integrand in the integral formula for the cumulants (of size~$\geq 3$) is holomorphic in each variable (this, in contrast to the moments where the integrand is meromorphic, is why dealing with cumulants is more straightforward). Nevertheless, to compute the cumulants and characterize the distribution, this holomorphic integrand must be computed exactly. We ultimately use an induction argument and, crucially, a degeneration of \emph{Fay's identity} for theta functions, Proposition 2.10 in~\cite{Fay73}, to prove that this integrand has a simple closed form expression. With this expression, we see that joint cumulants are expressed in terms of logarithmic derivatives of the theta function associated to the spectral curve; this is the content of Theorem~\ref{thm:discrete_comp_thm}. Then, using the \emph{modular transformation} for theta functions one obtains Theorem~\ref{thm:discrete_gauss_intro}, as explained above Corollary~\ref{eqn:discrete_gauss_corg}.

We remark that before the analysis described in the previous paragraph, the formula we start off with for the joint moment in the left hand side of Theorem~\ref{thm:discrete_gauss_intro} is (from Proposition~\ref{prop:DCcov})
\begin{equation}\label{eqn:most_general_joint_moment_intro}
\mathbb{E}[ \prod_{i=1}^g Z_i^{n_i}] 
\approx \frac{1}{(2\pi \i)^{m}}  \int_{B_1}\cdots \int_{B_1}
\cdots \cdots \int_{B_g} \cdots \int_{B_g} 
\det\bigg( (1-\delta_{l p})\omega_0(q_p', q_{l}') \bigg)_{p, l=1}^{m}.
 \end{equation}
Above,~$m=n_1+\cdots+n_g$, and the integrals are over~$B$ cycles in the spectral curve ($n_1$ integrations over~$B_1$, and so on), and the result of our analysis of the integrand is that we can take
$$
    \omega_0(q,q') = \frac{\theta\left(\int_{q'}^q\vec{\omega}-e_{\mathrm w_{0,0}}^{(kN)}\right)}{\theta\left(e_{\mathrm w_{0,0}}^{(kN)}\right)E(q,q')},
$$
where~$E$ is the prime form and~$\vec{\omega}$ is a vector consisting of a basis of holomorphic one forms on~$\mathcal{R}$; these objects are defined in Section~\ref{subsec:theta_prime}.
It was not obvious to us at all apriori that~\eqref{eqn:most_general_joint_moment_intro} should be related to the derivatives of a theta function as in Theorem~\ref{thm:discrete_gauss_intro}. Thus, in order to prove the theorem, it was crucial that we first guess the form of the answer, and then prove it by analyzing the cumulants, as we described above. In particular, our initial guess, based on universality considerations inspired by the robust heuristic arguments presented in~\cite{BDE00} and~\cite{Gor21}, was that a discrete Gaussian should appear, and from the formulas it was clear that the scale matrix should be related to the period matrix~$B$ of the spectral curve, and that the quasi-periodic behavior should be manifested via~$e_{\mathrm w_{0,0}}^{(kN)}$ as the shift. Then, after first checking and proving the result in the genus~$1$ case and (partially) matching it to the prediction coming from the adaptation of arguments in~\cite[Section 24.1]{Gor21} to our setting, we were able to exactly guess the parameters (scale matrix and shift) in higher genus setting and ultimately prove the general result.

\subsection{Plan of the paper}
\begin{itemize}
    \item In Section~\ref{sec:bg}, we precisely define our combinatorial conventions, such as microscopic and macroscopic coordinates, the Kasteleyn matrix, and the height function. We also recall the definition of the spectral curve and briefly review the associated objects which are used in our work. Finally, we recall the exact formula of~\cite{BB23} for the inverse of the Kasteleyn matrix, which is the~\emph{starting point} of our work.
    \item In Section \ref{sec:heightflucts}, we first state several lemmas to record the result of the steepest descent analysis, and then prove an integral formula for the leading order asymptotic of joint height moments.
    \item In Section~\ref{sec:GFF_DC}, we define the discrete components, and we also use theta functions to simplify the formula from the previous section. Using these ingredients, we characterize the limiting distribution of the height field (in terms of a GFF and a discrete Gaussian).
    \item In Section \ref{sec:steepest_arguments}, we perform the steepest descent analysis to prove the lemmas stated in Section \ref{subsec:lemmastatements}. 
\end{itemize}


\subsection{Acknowledgments}
We thank Vadim Gorin for valuable discussions, and thank Alexei Borodin for discussions and encouragement throughout, and for pointing us to the identities in Chapter 2 of~\cite{Fay73}, which ultimately played a crucial role in the identification of the moments of the discrete component in genus larger than~$1$. We are grateful to Sunil Chhita for sharing code to generate simulations. MN was partially supported by the NSF grant No. DMS 2402237. TB was partially supported by the Knut and Alice Wallenberg Foundation grant KAW 2019.0523 and by A.~Borodin's Simons Investigator grant.
\section{Background}
\label{sec:bg}

In this section, we review the essential notation and define the objects that we need. In particular, we state the exact formula for the inverse Kasteleyn from~\cite{BB23} that we will use throughout our work. Our notation for combinatorial objects matches the notation of that work, so we will be brief. Throughout this section, and the rest of this work, we fix arbitrary positive integers~$k$ and~$\ell$.

\subsection{Coordinates, Kasteleyn matrix, transition matrices, and height function}
\label{subsec:comb}

\begin{figure}
\centering
\raisebox{-.6pt}{
\includegraphics[scale=1]{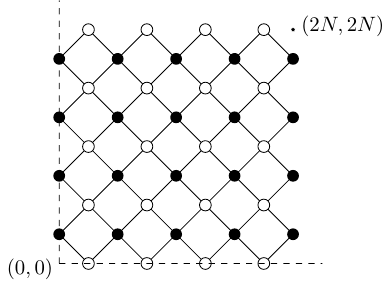}}
\qquad
\includegraphics[scale=.7]{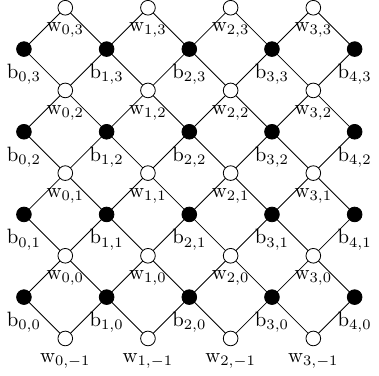}
    \caption{A size~$4$ Aztec diamond. Left: Its emedding in~$\mathbb{R}^2$. Centers of faces have integer coordinates, which gives coordinates on the set of faces. Right: Our convention for indexing the black and white vertices.}
    \label{fig:Aztec_coords}
\end{figure}

The Aztec diamond can be embedded in the plane~$\mathbb{R}^2$ as illustrated in Figure~\ref{fig:Aztec_coords}, left. With this embedding, the black vertex at position~$(2 (\ell x + i), 2(k y + j)+1)$ will be denoted~$\mathrm b_{\ell x + i, k y + j}$, and the white vertex at position~$(2 (\ell x + i)+1, 2(k y + j)+2)$ by~$\mathrm w_{\ell x + i, ky + j}$ (see Figure~\ref{fig:Aztec_coords}, right). Note that the embedding is such that the ``bottom left'' white vertex~$\mathrm w_{0, -1}$ is at position~$(1,0)$ and the ``bottom left'' black vertex~$\mathrm b_{0,0}$ is at position~$(0,1)$.

We must also define indexing on the faces: We simply use the coordinates coming from the embedding in the plane illustrated in Figure~\ref{fig:Aztec_coords}, left. With this embedding, the centers of faces have coordinates in~$\mathbb{Z}^2$, and we use these coordinates to index a face. To relate this to the indexing of vertices, the face~$\mathsf{f}$ directly above the black vertex~$\mathrm b_{\ell x + i, k y + j}$ is~$\mathsf{f} = (2(\ell x + i), 2(k y + j)+2)$.

The edge weights are determined by positive real numbers~$\alpha_{j,i}, \beta_{j, i}, \gamma_{j, i}$,~$i=1,\dots, \ell$,~$j=1,\dots, k$. The edge weights repeat periodically with period~$\ell$ in the horizontal direction and~$k$ in the vertical direction (with respect to the coordinates in Figure~\ref{fig:Aztec_coords}, right). See Figure~\ref{fig:Aztec_example} for an illustration in the~$k = \ell = 2$ case. After choosing such edge weights~$\{\nu_e\}$, the dimer model probability measure on perfect matchings~$M$ is defined by
\begin{equation}\label{eqn:dimer_defn}
    \mathbb{P}(M) = \frac{1}{Z}\prod_{e \in M} \nu_e
\end{equation}
where the \emph{partition function}~$Z$ is a normalization constant.

Next, we define the \emph{Kasteleyn matrix}, first introduced by Kasteleyn~\cite{Kas61},  that we use for this Aztec diamond dimer model. 
\begin{definition}[Kasteleyn Matrix]
The Kasteleyn matrix is defined by:
\begin{equation}\label{eqn:Kdef}
K(\mathrm w_{\ell x_1 + i_1, ky_1 + j_1}, \mathrm b_{\ell x_2 + i_2, ky_2 + j_2}) = 
\begin{cases}
    \alpha_{j_1 + 1, i_1 + 1}, & (\ell x_2 + i_2, ky_2 + j_2) = (\ell x_1 + i_1, ky_1 + j_1 + 1), \\
    \gamma_{j_1 + 1, i_1 + 1}, & (\ell x_2 + i_2, ky_2 + j_2) = (\ell x_1 + i_1, ky_1 + j_1), \\
    \beta_{j_1 + 1, i_1 + 1}, & (\ell x_2 + i_2, ky_2 + j_2) = (\ell x_1 + i_1 + 1, ky_1 + j_1 + 1), \\
    -1, & (\ell x_2 + i_2, ky_2 + j_2) = (\ell x_1 + i_1 + 1, ky_1 + j_1), \\
    0, & \text{otherwise}.
\end{cases}
\end{equation}
\end{definition}
It is a classical result that~$Z = |\det K|$, i.e. the determinant of~$K$ computes the partition function, and moreover that minors of the matrix~$K^{-1}$ compute edge inclusion probabilities.

\begin{figure}
\centering
\includegraphics[scale=.25]{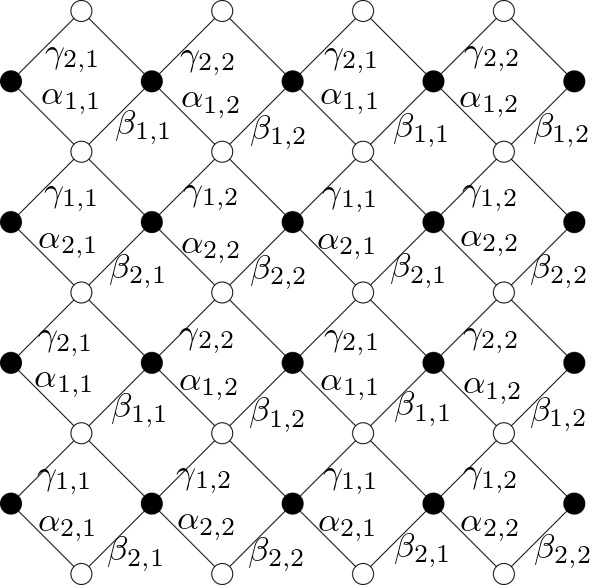}
    \caption{A size~$4 = k \ell 1$ Aztec diamond with~$k \times \ell = 2 \times 2$ periodic weights. The edges with no label have weight~$1$. Furthermore it is these edges with a negative sign in the Kasteleyn weighting, and also these edges which are used in the reference matching for computing the height function.}
    \label{fig:Aztec_example}
\end{figure}

Another discrete object we must define are the \emph{symbols}~$\phi_m(z)$, which are~$k \times k$ matrices with entries which are meromorphic in~$z$. These matrices appear in the exact formula for the inverse Kasteleyn matrix, stated in Section~\ref{subsec:exact_formula} below, and they are related to certain transition weights coming from the ``transfer matrix'' formulation of the model via nonintersecting paths utilized in~\cite{BB23}. We simply define them here and refer to that work for details; each matrix depends on the edge weights~$\alpha_{i,j},\beta_{i, j}, \gamma_{i,j}$. For~$i=1,\dots,\ell$, the odd transition matrices are defined by
\begin{equation}\label{eqn:oddphi}
\phi_{2 i -1}(z) = \begin{pmatrix}
\gamma_{1, i} & 0 & \cdots & 0 & \alpha_{k,i} z^{-1} \\
 \alpha_{1, i} & \gamma_{1, i} & \cdots & 0 & 0  \\
 \vdots & \vdots & \ddots & \vdots & \vdots \\
 0 & 0 & \cdots & \gamma_{k-1, i} & 0 \\
 0 & 0 & \cdots & \alpha_{k-1,i} & \gamma_{k, i}
\end{pmatrix} 
\end{equation}
and the even ones are defined by
\begin{equation}\label{eqn:evenphi}
\phi_{2 i}(z) = \frac{1}{1-\beta_i^v z^{-1}} \begin{pmatrix}
1 & \prod_{j=2}^k \beta_{j, i} z^{-1} & \cdots & \beta_{k, i} z^{-1} \\
\beta_{1,i} & 1 & \cdots & \beta_{k, i} \beta_{1,i} z^{-1} \\
\vdots & \vdots & \ddots  & \vdots \\
\prod_{j=1}^{k-1}\beta_{j,i} & \prod_{j=2}^{k-2} \beta_{j,i} & \cdots & 1
\end{pmatrix} .
\end{equation}
Above we have used the notation
\begin{equation}\label{eqn:betav}
\beta_i^v \coloneqq \prod_{j=1}^k \beta_{j,i}, \qquad i=1,\dots,\ell.
\end{equation}
Moreover, following~\cite{BB23}, define
\begin{equation}\label{eqn:capitalPhi}
\Phi(z) \coloneqq \prod_{m=1}^{2 \ell} \phi_m(z) .
\end{equation}

Finally, we define the height function. For the reference matching, we use the edges which have a~$-1$ in~\eqref{eqn:Kdef}; we call this reference matching~$M_0$. The reference matching is not a perfect matching of the Aztec diamond itself, but it extends to a perfect matching of the full square lattice (and thus it can be used to define the height function). In more detail:
\begin{definition}\label{def:h_N}
The height function~$h_N$ of the size~$k \ell N$ Aztec diamond is defined by the following rule: if a step between neighboring faces~$\mathsf{f} \rightarrow \mathsf{f}'$ crosses an edge~$e$ with the white vertex on the right, then
\begin{equation}\label{eqn:h_N}
h_N(\mathsf{f'}) = h_N(\mathsf{f}) + \mathbf{1}_{e \in M} - \mathbf{1}_{e \in M_0}.
\end{equation}
\end{definition}  
In other words, if we cross an edge in the matching (reference matching) with white on the right,~$h_N$ increases (decreases) by~$1$. See Figure~\ref{fig:Aztec_ht} for an example;~$M_0$ consists of the orange edges.

\subsection{The spectral curve and assumptions on edge weights}\label{subsec:spectral}

The Kasteleyn matrix defined in~\eqref{eqn:Kdef} can be extended to the entire bipartite lattice~$\mathbb{Z}^2$ containing the Aztec diamond. Then, modding out by the edge weight and vertex color-preserving action of~$\mathbb{Z}^2$ by translations, we obtain an edge weighted graph with Kasteleyn signs defined on the torus. If we choose a fundamental domain such that it contains vertices~$\{\mathrm b_{i, j}, \mathrm w_{i, j}\}_{i=0,\dots, \ell-1,j=0,\dots, k-1}$, and add extra (complex) edge weights~$z^{\pm 1}$,~$w^{\pm 1}$ (with signs chosen based on an orientation) along edges crossing the horizontal, vertical boundary of the fundamental domain, respectively, then we obtain the \emph{magnetically altered} Kasteleyn matrix~$K_{G_1}(z, w)$ defined on the torus. This procedure is originally described in general in~\cite{KOS06}, and~$K_{G_1}(z, w)$ is defined in more detail in our setting in~\cite[Section 2.5]{BB23}; we follow the notation from there throughout. Let
\begin{equation}\label{eqn:char_poly}
P(z, w) = \det K_{G_1}(z, w).
\end{equation}
Note that~$P$ has real coefficients.

Consider the set~$\mathcal{R}^{\circ} \coloneqq \{(z, w) \in (\mathbb{C}^*)^2 : P(z, w) = 0 \}$. This defines an algebraic curve which is invariant under complex conjugation~$(z, w) \mapsto (\overline{z}, \overline{w})$. In the cases we consider there are generically~$2 (k + \ell)$ ``points at infinity'', where either~$z=0,\infty$ or~$w=0,\infty$. These points at infinity are the so-called~\emph{angles}, and are denoted by~$\{p_{0,j}\}_{j=1}^k, \{p_{\infty,j}\}_{j=1}^k, \{q_{0,j}\}_{j=1}^\ell, \{q_{\infty,j}\}_{j=1}^\ell $; in that order the groups of angles correspond to points where~$z = 0, z= \infty, w = 0, w = \infty$, respectively.
\begin{definition}
    \label{def:spectral_curve}
    The~\emph{spectral curve}~$\mathcal{R}$ is defined as~$\mathcal{R}^{\circ} $ with~$2 (k + \ell)$ points, or angles, glued in, where either~$z=0,\infty$ or~$w=0,\infty$. As a set,~$\mathcal{R} \coloneqq \mathcal{R}^{\circ} \cup \{p_{0,j}\}_{j=1}^k \cup  \{p_{\infty,j}\}_{j=1}^k\cup \{q_{0,j}\}_{j=1}^\ell  \cup \{q_{0,j}\}_{j=1}^\ell$.
\end{definition}

For generic edge weights, the genus of~$\mathcal{R}$ is~$g = (k-1)(\ell -1)$, and there are~$2 (k + \ell)$ distinct angles.  For non generic sets of weights, some subsets of angles within each group may merge together, and the curve may have smaller genus and develop singularities known as~\emph{real nodes}. If there are no real nodes, then~$\mathcal{R}$ is smooth. We allow edge weights where angles merge, but we do assume there are no real nodes.

\begin{ass}\label{ass:1}
We assume that our edge weights are such that the genus of~$\mathcal{R}$ is~$g = (k-1) (\ell-1)$. This is the only assumption that we make on the edge weights.
\end{ass}

 A natural way to precisely define~$\mathcal{R}$  is to embed~$\mathcal{R}^\circ \subset (\mathbb{C}^*)^2$ into a projective space of high dimension via the \emph{moment map} and then take the closure in that space, see~\cite[Section 3.2]{BB23} for an exposition in our exact context, and see references therein for more details. 

It is well known~\cite{KOS06} that the spectral curve~$\mathcal{R}$ is a so-called \emph{harnack curve}. For us, this fact is best described in terms of the so-called \emph{amoeba} corresponding to~$P(z,w)$. The~\emph{amoeba} of~$\mathcal{R}^{\circ}$ is the image~$\mathcal{A}$ of~$\mathcal{R}^{\circ} \subset (\mathbb{C}^*)^2$ under the map~$\Log : (\mathbb{C}^*)^2\rightarrow \mathbb{R}^2$ given by
\begin{equation}\label{eqn:logmap}
    \Log(z, w) = (\log|z|, \log |w|).
\end{equation}
In other words,~$\mathcal{A} = \Log(\mathcal{R}^\circ)$. The amoeba arises naturally through computations (dating back to~\cite{KOS06}) as the phase diagram of the dimer model. Points in the interior of the amoeba parameterize liquid phase ergodic, translation invariant Gibbs measures for the dimer model, the noncompact boundary arcs (connected components of the outer boundary) correspond to frozen phases, and the compact \emph{inner oval} boundary components correspond to gaseous phases. For generic edge weights, there are~$g = (k-1)(\ell -1)$ of these inner boundary components (recall~$g$ is the genus of~$\mathcal{R}$). The following proposition is a fundamental result which holds in general for periodic dimer models.
\begin{prop}[\cite{KOS06}]
   Away from the real points~$\mathcal{R}^\circ \cap \mathbb{R}^2$ of~$\mathcal{R}^\circ$, $\Log : \mathcal{R}^\circ \rightarrow \mathcal{A}$ is~$2$-to-$1$. At the real points it is~$1$-to-$1$.
\end{prop}
In other words, $\Log$ is a bijection up to complex conjugation, away from the points of~$\mathcal{R}^\circ$ where~$(z, w) \in \mathbb{R}^2$. This is equivalent to~$\mathcal{R}$ being a Harnack curve~\cite{Mik00}.

This last fact is important for us; it implies that, topologically,~$\mathcal{R}$ can be obtained by gluing together two copies of the amoeba along their boundary, and adding points at~$\infty$ corresponding to \emph{tentacles} of the amoeba which go off to~$\infty$ (see Figure~\ref{fig:amoeba3by3}). Our steepest descent arguments (which follow ideas originally developed in~\cite{BB23}) will make use of this fact, and furthermore all of our pictures of contours on~$\mathcal{R}$ will be depicted via the amoeba. 

In fact, an important object for us will be~$\mathcal{R}_0$, which is a distinguished ``top half'' of the Riemann surface; indeed, the involution~$\sigma$ given by conjugation separates~$\mathcal{R}^\circ \setminus \mathbb{R}^2$ into two connected components. We can choose~$\mathcal{R}_0$ to be the one with positively oriented boundary. By the previous paragraph,~$\mathcal{R}_0 \setminus \{\text{angles}\}$ is in bijection with the amoeba~$\mathcal{A}$ under~$\Log$. We denote the~$g+1$ components of the boundary of~$\mathcal R_0$ (the real part of~$\mathcal R$) by~$A_i$,~$i=0,\dots,g$, and refer to~$A_0$ as the \emph{outer oval} and~$A_i$,~$i=1,\dots,g$ as the \emph{compact ovals}. All angles lie on the outer oval, and we denote the components of~$A_0\cap\mathcal R^\circ$ by~$A_{0,i}$,~$i=1,\dots,2(k+\ell)$, where~$A_{0,1}$ lies between a~$p_{\infty,i}$ angle and a~$q_{\infty,j}$ angle,~$A_{0,2}$ between two~$q_{\infty,j}$ angles, and so on. See Figure~\ref{fig:amoeba3by3} for their images in the amoeba.

One important interpretation of the surface~$\mathcal{R}$ is that it encodes the eigenvectors, for a generic fixed~$z \in \mathbb{C}$, of the matrix~$\Phi(z)$ defined in~\eqref{eqn:capitalPhi}. Recall~$\mathcal{R}^{\circ}$ from the previous discussion. 
\begin{lem}[Proposition 3.1 of \cite{BB23}]
The spectral curve is equivalently defined by
\begin{equation}\label{eqn:eivenvecs}
\mathcal{R}^\circ = \{ (z,w) \in (\mathbb{C}^*)^2 :  \prod_{i=1}^{\ell} (1-\beta_i^v z^{-1}) \det( \Phi(z) - w I)  = 0 \}
\end{equation}
with~$\beta_i^v$ as in~\eqref{eqn:betav}.
\end{lem}
As a result, for~$(z,w) \in \mathcal{R}^{\circ}$ (which in particular means~$z \neq \beta_i^v$ for any~$i$), it makes sense to define corresponding left and right eigenvector~$\psi_{0,-}(z,w) $ and~$\psi_{0,+}(z,w)$ of~$\Phi(z)$, which satisfy
\begin{align}
\Phi(z) \psi_{0,+}(z,w) &= w  \psi_{0,+}(z,w) \\
\psi_{0,-}(z,w)  \Phi(z) &= w \psi_{0,-}(z,w) .
\end{align}
Of course,~$\psi_{0,\pm}(z,w)$ are only defined up to a constant. We use the same definition as in~\cite[Section 5.3]{BB23}, where~$\psi_{0,+}(z,w)$  is defined as the first column of a matrix proportional to the adjugate of~$\Phi(z) - w I$, and~$\psi_{0,-}(z,w)$ is defined as a linear combination of rows of the adjugate of~$\Phi(z) - w I$. Though the exact definitions are not important for us, we note that the entries of~$\psi_{0,\pm}(z,w)$ can be analytically continued to meromorphic functions on~$\mathcal{R}$. We remark in particular that while in (5.13) and (5.14) in~\cite{BB23} the entries of those vectors are viewed as one forms (note the factors of~$d z$ there), we will consider them to be meromorphic functions, unless if we explicitly include the factor of~$d z$.

The particular form of~$\psi_{0,\pm}$ is chosen to be well-adapted to the analysis of the matrix refactorization procedure in~\cite[Section 5]{BB23}. As part of that same refactorization procedure, matrices~$\Phi_j(z)$, for~$j=0,1,\dots,k N$, are iteratively defined, and these matrices have left and right eigenvectors~$\psi_{j, \pm}(z,w)$. The entries of each of the~$\psi_{j, \pm}(z,w)$ are again meromorphic functions on~$\mathcal{R}$. The precise definition of matrices~$\Phi_j(z)$ and their left and right eigenvectors~$\psi_{j, \pm}(z,w)$ are not important for us; however, the exact formulas for the eigenvectors~$\psi_{0, \pm}(z,w)$ and~$\psi_{k N, \pm}(z,w)$ derived in~\cite[Proposition 5.4]{BB23}, which is restated as Proposition~\ref{prop:psipmexact} in the next subsection below, \emph{will} be important for us, since we will use these formulas to simplify our expressions for height moments. Moreover, we will need to use the basic identity 
\begin{equation}\label{eqn:psipluspsiminus}
\psi_{k N,-}(z,w) \psi_{k N, +}(z,w) = w^{k N} \psi_{0,-}(z,w) \psi_{0, +}(z,w)
\end{equation}
which appears in the proof of~\cite[Lemma 6.4]{BB23}, and can be derived from tracing through the definitions of~$\Phi_{j}$,~$j=0,1,\dots$, and~$\psi_{k N, \pm}$ given there.

\begin{figure}
    \centering
    \includegraphics[width=0.6\linewidth]{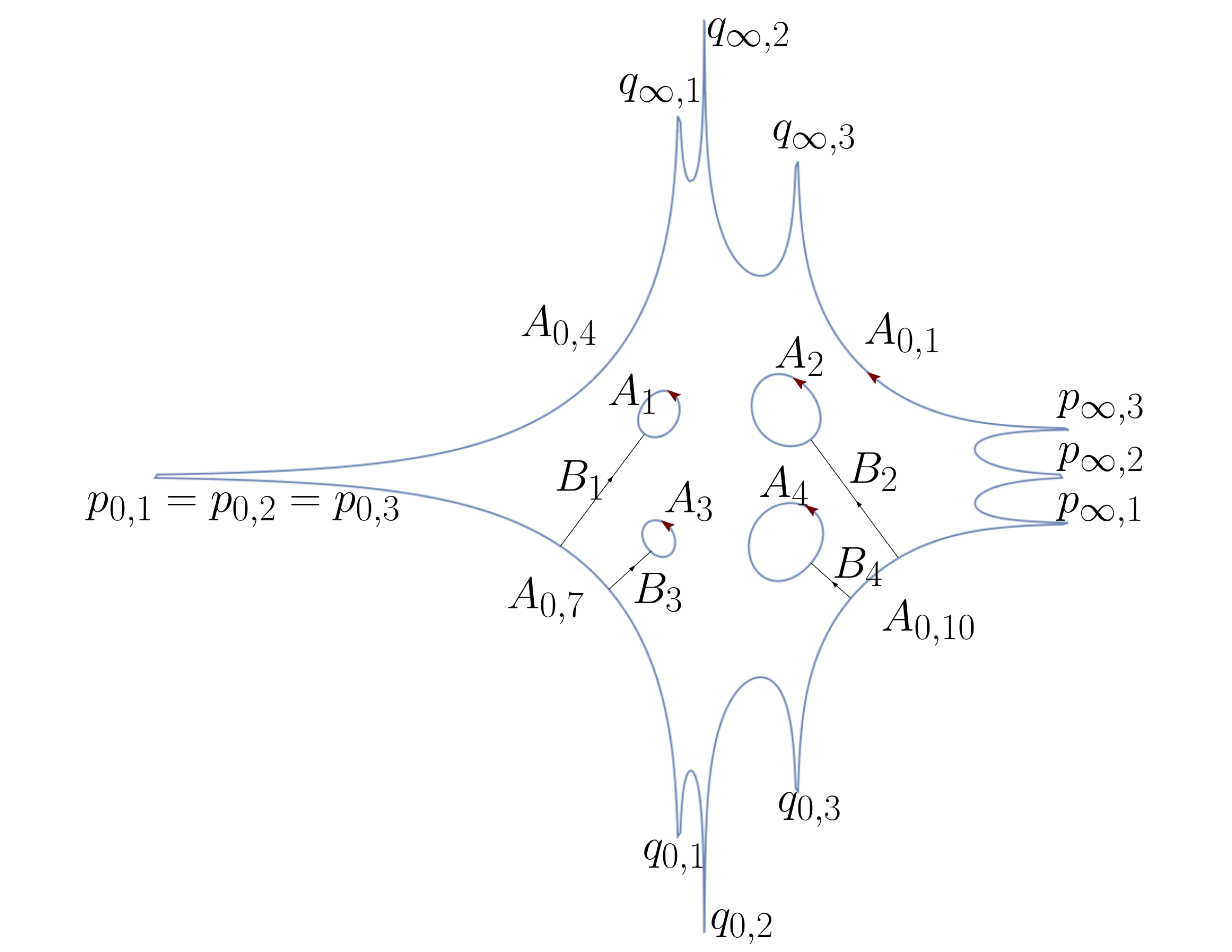}
    \caption{The amoeba associated to a spectral curve~$P(z,w) = 0$ with~$3 \times 3$ periodic weights. The tentacles are labeled with the corresponding angles, and images of~$A$ and~$B$ cycles are shown as well. In this example, the three angles~$p_{0,i}$ have merged together.}
    \label{fig:amoeba3by3}
\end{figure}

\subsection{A and B cycles, theta functions, and prime forms}
\label{subsec:theta_prime}

In this section we will very briefly outline a few necessary facts about prime forms and theta functions. We will be very brief; our main goal here is to simply set up the notation used in the remainder of the paper. 

First we need some basic facts about compact Riemann surfaces. One elementary topological fact is that on any compact Riemann surface~$\mathcal{R}$ there exists a so-called \emph{canonical basis} of cycles~$A_1, \dots, A_g, B_1, \dots B_g$, which form a basis of~$H_1(\mathcal{R}; \mathbb{R})$ and have intersection numbers~$A_i \circ A_j = B_i \circ B_j = 0$,~$A_i \circ B_j = \delta_{i j}$. In our setting, the~$A$ and~$B$ cycles are chosen as depicted in Figure~\ref{fig:amoeba3by3}. In particular, the~$g$ number of~$A$ cycles are the \emph{compact ovals} which are identified with compact boundary components of the amoeba~$\mathcal{A}$ under~$\Log$.

On a compact Riemann surface of genus~$g$, there always exists a basis in~$g$-dimensional complex vector space of holomorphic~$1$ forms which is~\emph{dual} to a canonical basis of cycles. We denote our basis of one forms as~$\omega_1,\dots, \omega_g$; being a dual basis means that~$\int_{A_i} \omega_j = \delta_{i j}$. Define the~$g \times g$ \emph{period matrix} by~$B_{i j} \coloneqq \int_{B_i} \omega_j$. It is always true that $B$ is a symmetric matrix with positive definite imaginary part.  Since~$\mathcal{R}$ is a harnack curve, it is also a so-called~\emph{M curve}, which implies that~$B$ is purely imaginary~\cite[Lemma 11]{BCT22}.

The \emph{theta function} associated to~$\mathcal{R}$ is a holomorphic function~$\theta : \mathbb{C}^g \rightarrow \mathbb{C} $ defined by the absolutely convergent series
\begin{equation}\label{eqn:theta_def}
\theta(z) = \theta(z; B)  \coloneqq \sum_{n \in \mathbb{Z}^g} e^{\i \pi  (n\cdot B n + 2 n\cdot z ) }.
\end{equation}
The function~$\theta(z)$ is periodic under translations by elements of~$\mathbb{Z}^g$ and quasi-periodic under translations by elements of~$B \mathbb{Z}^g$. Denote by~$\vec{\omega} = (\omega_1,\dots, \omega_g)$ the~$g$ column vector consisting of the chosen basis of holomorphic one forms, and let~$\tilde q_0$ be a reference point in the universal cover~$\widetilde{\mathcal{R}}$ of~$\mathcal{R}$. The theta function satisfies the properties that the function~$\Theta(\tilde q;e) \coloneqq \theta\left( \int_{\tilde q_0}^{\tilde q} \vec{\omega} + e \right)$ defined on the universal cover~$\widetilde{\mathcal{R}}$, if it is not identically zero, has a well-defined on~$\mathcal{R}$ divisor~$(\Theta(\cdot;e))$ which satisfies~$u((\Theta(\cdot;e))) = -e + \Delta$, where~$\Delta \in J(\mathcal{R})$ is a special point called the \emph{vector of Riemann constants}; here
\begin{equation*}
u(q)=\int_{q_0}^q\vec{\omega} \mod J(\mathcal{R})
\end{equation*}
is the \emph{Abel map}, and~$J(\mathcal{R}) = \mathbb{C}^g/(\mathbb{Z}^g + B \mathbb{Z}^g)$ is the \emph{Jacobi variety}. Moreover, in the case~$\Theta(\tilde q;e)$ is not identically zero this relation uniquely determines the divisor (a-priori several different divisors could map to the point~$-e + \Delta \in J(\mathcal{R})$). We remark that, similarly to~\cite{BB23}, the function~$\Theta(\tilde q;e)$ will never be identically zero in the situations we consider in the text. Here we follow the notations of~\cite{BB23} exactly, and we refer the reader to Section 3.3 there for more details about the Abel map and for the precise qausi-periodicity properties satisfied by~$\theta$.

Another object we will use is the \emph{prime form}~$E(\tilde p, \tilde q)$. This is defined on the universal cover~$\widetilde{\mathcal R} \times \widetilde{\mathcal R}$, and it satisfies the property that if~$\sum_{i=1}^n Q_i - \sum_{j=1}^n P_j$ is the divisor of a meromorphic function~$f$ on~$\mathcal{R}$, then~$f(q) = c \frac{\prod_{j=1}^n E(\tilde Q_j, q) }{\prod_{j=1}^n E(\tilde P_j, q)}$, where~$c$ is some constant,~$\tilde Q_j$ and~$\tilde P_j$ are appropriate lifts of~$Q_j$ and~$P_j$, and~$q$ is any choice of lift; the expression is in fact well-defined on~$\mathcal{R}$. On its own, however,~$E$ is not a meromorphic function; in local coordinates~$z$ and~$w$ near~$\tilde p$ and~$\tilde q$ in~$\widetilde{\mathcal{R}}$ it is has the form
\begin{align*}
    E(z, w) =\frac{f(z, w)}{\sqrt{d z} \sqrt{d w}}   
\end{align*}
where the~$\sqrt{d  z} \sqrt{d w}$ in the denominator indicates how it transforms under changes of variables (i.e. which line bundle it is a section of). 

The properties of the prime form that we need are that it is holomorphic everywhere (has no poles), it satisfies 
\begin{align}
E(q_1,q_2) = - E(q_2,q_1) 
\end{align}
and that as~$q_1 \rightarrow q_2$, in local coordinates~$z_1 = z(q_1),z_2 = z(q_2)$ we have the behavior
\begin{equation}\label{eq:prime_form_diagonal}
E(q_1,q_2) = \frac{z_2-z_1}{\sqrt{d z_1 dz_2}} \left( 1+ O(|z_1-z_2|) \right)
\end{equation}
and also~$E(q_1,q_2) \neq 0$ for~$q_1 \neq q_2$. For a list of its quasi-periodicity and other properties, see~\cite[Fact 3.3]{BB23}.

One of the main results of~\cite{BB23}, which will also be important for us, is the derivation of the following exact formula for the left and right eigenvectors~$\psi_{j, \pm}(z, w)$ described in the previous subsection in terms of prime forms and theta functions. In the proposition below for compactness we use the notation~$\Theta(q; e) \coloneqq \theta(\int_{q_0}^{q} \vec{\omega} + e)$, for~$e \in \mathbb{C}^g$ (where~$q_0 \in A_0$ is a base point which is fixed throughout).
\begin{prop}[Proposition 5.4 of \cite{BB23}]\label{prop:psipmexact}
For each~$m=0,1,\dots$, and~$j=0,\dots, k-1$, there exist~$e_{\mathrm b},  e_{\mathrm w}, e_{\mathrm b_{0,j}}^{(m)}, e_{\mathrm w_{0,j}}^{(m)} \in \mathbb{R}^g$ such that the~$(j+1)$st entry~$\psi_{m,\pm}^{(j+1)}$ of~$\psi_{m,\pm}$ is given by
\begin{equation}\label{eqn:psimplus}
\psi_{m,+}^{(j+1)}(z,w) dz= c_{j,+}^{(m)} \Theta(q; e_{\mathrm b}) \Theta(q; e_{\mathrm w_{0,j}}^{(m)}) \frac{ \prod_{i=1}^{\ell m} E(q_{0,i},q) \prod_{i=1}^{j-1} E(p_{0,i},q)}{\prod_{i=1}^{\ell m} E(p_{0,j-i}, q) \prod_{i=1}^{j+1} E(p_{\infty,i},q)} 
\end{equation}
and
\begin{equation}\label{eqn:psimmin}
\psi_{m,-}^{(j+1)}(z,w) dz= c_{j,+}^{(m)} \Theta(q; e_{\mathrm b_{0,j}}^{(m)}) \Theta(q; e_{\mathrm w}) \frac{ \prod_{i=1}^{\ell m} E(p_{0,j+1-i},q) \prod_{i=2}^{j} E(p_{\infty,i},q)}{\prod_{i=1}^{\ell m} E(p_{\infty,i}, q) \prod_{i=1}^{j} E(p_{0,i},q)} 
\end{equation}
for some constants~$c_{j,+}^{(m)}, c_{j,-}^{(m)} \in \mathbb{C}^*$. The indices of the angles~$p_{0/\infty,i}$ and~$q_{0/\infty,i}$ are taken modulo~$k$ and~$\ell$, respectively. 
\end{prop}
In addition, from the proof of the proposition in~\cite{BB23},~$e_{\mathrm w_{0,j}}^{(m)}$ and~$e_{\mathrm b_{0,j}}^{(m)}$ can be computed in terms of the Abel map applied to the angles. The vector~$e_{\mathrm w_{0,0}}^{(kN)} \in \mathbb{R}^g/\mathbb{Z}^g$ is the same one appearing in Theorem \ref{thm:discrete_gauss_intro}. The vector~$e_{\mathrm w_{0,0}}^{(k N)} \in \mathbb{R}^g / \mathbb{Z}^g$ plays a central role in our results; it is defined by the following formulas.

First, denote by~$D$ the divisor of common zeros in~$\mathcal{R}$ of the entries in the column of the adjugate matrix~$\adj K_{G_1}(z,w)$ indexed by~$\mathrm w_{0,0}$. Then, define
\begin{equation}\label{eqn:ew00first}
e_{\mathrm w_{0,0}} = \Delta - u(D)
\end{equation}
where~$u$ is the Abel map as defined earlier in this section, and let 
\begin{equation}\label{eqn:ew000}
e_{\mathrm w_{0,0}}^{(0)} =e_{\mathrm  w_{0,0}}+u(q_{0,1})-u(p_{0,k}).
\end{equation}
Finally, we have
\begin{equation}\label{eqn:ew00}
e_{\mathrm w_{0,0}}^{(k N)} - e_{\mathrm w_{0,0}}^{(0)} = N \left(k \sum_{i=1}^{\ell} u(q_{0,i})-\ell \sum_{i=1}^k u(p_{0,i})\right).
\end{equation}
In an early version of~\cite{BB23} there is a sign error in both~\eqref{eqn:ew000} and~\eqref{eqn:ew00}, both of which are accounted for here. 

\begin{remark}\label{rmk:q0dep}
Both~\eqref{eqn:ew000} and~\eqref{eqn:ew00} depend on the lift~$\tilde q_0$ of a base  point~$q_0 \in A_0$, a choice on which the Abel map depends. In Section~\ref{subsec:DCdist}, we explicitly assume we have made the choice~$q_0 = p_{\infty, 1}$, so that~$u(p_{\infty, 1}) = 0$. We also use this choice in the statement of Theorem~\ref{thm:discrete_gauss_intro} in the introduction. However, in Section~\ref{subsec:rewriting} we do not explicitly use this assumption, which leads to the appearance of~$u(p_{\infty,1}) + e_{\mathrm w_{0,0}}^{(k N)}$ inside of the theta functions in~\eqref{eqn:omega0_statement}. 
\end{remark}

The divisor in~\eqref{eqn:ew00first} is a part of the spectral data introduced in~\cite{KOS06} and~\cite{KO06}. Let~$D_\mathrm{v}$ be the divisor of the common zeros of the entries of the column or vector of~$\adj K_{G_1}(z,w)$ indexed by~$\mathrm v$, then
\begin{equation*}
    t+\mathbf d(\mathrm v)=\Delta-u(D_\mathrm{v}),
\end{equation*}
for some constant~$t\in \RR^g/\ZZ^g$ and where~$\mathbf d$ is the discrete Abel map. The discrete Abel map is defined from the union of the vertices of the graph and its dual graph and is locally defined via the angles. The constant~$t$ is a point on the real part of the Jacobian, and form, together with the spectral curve~$\mathcal R$, the spectral data that parametrizes the weights modulo gauge equivalence of the fundamental domain. See~\cite{BCT22}, in particular Remark 50 therein, and~\cite[Section 5.4]{BB23} for a specialization of the discrete Abel map to our setting. Using the convention of the latter reference, we note that
\begin{equation*}
    \mathbf d(\mathrm w_{0,0})=-u(q_{0,1})+u(p_{0,k})-u(p_{\infty,1}).
\end{equation*}
In particular,~\eqref{eqn:ew00first} and~\eqref{eqn:ew000} implies that
\begin{equation*}
    e_{\mathrm w_{0,0}}^{(0)}+u(p_{\infty,1})=e_{\mathrm  w_{0,0}}-\mathbf d(\mathrm  w_{0,0})=t,
\end{equation*}
and, hence, the shift in Theorem~\ref{thm:discrete_gauss_intro} is given by
\begin{equation}\label{eq:flow_plus_spectral_data}
    e_{\mathrm w_{0,0}}^{(kN)}+u(p_{\infty,1})=N \left(k \sum_{i=1}^{\ell} u(q_{0,i})-\ell \sum_{i=1}^k u(p_{0,i})\right)+t.
\end{equation}

\subsection{Exact formula for the inverse Kasteleyn}
\label{subsec:exact_formula}

Throughout this work, our convention (following~\cite{BB23}) is that rescaled coordinates of vertices are in~$[-1,1]^2$, which plays the role of the ``limit'' of the rescaled Aztec diamond. We say a black vertex~$\mathrm b_{\ell x+i, k y + j}$ or white vertex~$\mathrm w_{\ell x+i, k y + j}$ has \emph{rescaled coordinates}
\begin{equation}\label{eqn:rescaled}
(\xi, \eta) = (\frac{2}{k N}x-1, \frac{2}{\ell N}y-1).
\end{equation}
We will also sometimes say that the face~$\mathsf{f} = (2(\ell x + i), 2(k y + j)+2)$ has rescaled coordinates given by~\eqref{eqn:rescaled}.

Before stating the result we recall the~\emph{action function}~$F$ defined in~\cite[Definition 4.2]{BB23}, and the meromorphic function~$f$ also defined there which appears as part of the action function. Strictly speaking~$f$ and~$F$ are only well defined on the universal cover~$\widetilde{R}$; however, we omit this from the notation in what follows due to the fact that~$\Re[F(q; \xi, \eta)]$, which will be the central object in our asymptotic analysis, and any other expressions involving~$F$ and~$f$ that we will study are all well defined on~$\mathcal{R}$. 

We restate the definition of the action function given in~\cite[Definition 4.2]{BB23}, except as discussed above we omit from our notation the dependence on lifts to the universal cover. First, define 
\begin{equation}\label{eqn:fdef}
  f(q) \coloneqq \frac{\prod_{i=1}^\ell E(q_{0,i},q)^k }{\prod_{j=1}^k E(p_{0,j},q)^{\ell} }.
\end{equation}
This will be used in the definition of the action function below.
\begin{definition}
Let~$\xi, \eta \in (-1,1)^2$, and let~$q = (z,w) \in \mathcal{R}$. Then with~$f$ defined as in~\eqref{eqn:fdef},
\begin{equation}\label{eqn:Fdef}
    F(q;\xi, \eta)  \coloneqq \frac{k}{2}(1-\xi) \log w - \frac{\ell}{2}(1-\eta) \log z - \log f(q).
\end{equation}
\end{definition}

The following lemma gives an exact double contour integral, which is an iterated contour integral of a meromorphic $(1,1)$ form (a one form in each variable, which is well-defined on~$\mathcal{R}\times \mathcal{R}$) over a contour in~$\mathcal{R}$. Let~$\Gamma$ be a closed contour in the spectral curve~$\mathcal{R}$, invariant under conjugation, whose image in~$\mathcal{A}$ is a segment beginning in~$A_{0,k+\ell+1}$ and ending at~$A_{0,1}$. Similarly, let~$\Gamma_s$ and~$\Gamma_l$ be two closed contours of the same form, with the property that~$\Gamma_s$ and~$\Gamma_l$ do not intersect, and the image of~$\Gamma_s$ intersects~$A_{0,k+\ell+1}$ and~$A_{0,1}$ at points with smaller horizontal~$\log |z|$ coordinate; i.e.~$\Gamma_s$ is ``to the left of'' of~$\Gamma_l$ in the amoeba representation. 

Define~$(\xi_{j,N}, \eta_{j,N}) = (\frac{2}{k N}x_j-1, \frac{2}{\ell N}y_j-1)$,~$j=1,2$, c.f.~\eqref{eqn:rescaled}. We keep~$N$ as a subscript to emphasize that these rescaled coordinates correspond exactly to a lattice site in finite size Aztec diamond. The formula in the theorem below also depends on~$\psi_{0,\pm}(z,w)$ and~$\psi_{k N,\pm}(z,w)$ defined in the discussion at the end of Section~\ref{subsec:spectral}, and also appearing in Proposition~\ref{prop:psipmexact}.

\begin{lem}[Theorem 2.9 and Proposition 6.2 of \cite{BB23}]\label{lem:double_int_simple}
 Under Assumption~\ref{ass:1} on the edge weights, we have
\begin{multline}\label{eqn:kast_inv_2}
      K^{-1}(\mathrm b_{\ell x+i,ky+j},\mathrm w_{\ell x'+i',ky'+j'}) = \\
      I_1(\mathrm b_{\ell x_2+i_2,ky_2+j_2},\mathrm w_{\ell x_1+i_1,ky_1+j_1})  + I_2(\mathrm b_{\ell x_2+i_2,ky_2+j_2},\mathrm w_{\ell x_1+i_1,ky_1+j_1}) 
 \end{multline}
where 
\begin{multline}
  I_2(\mathrm b_{\ell x_2+i_2,ky_2+j_2},\mathrm w_{\ell x_1+i_1,ky_1+j_1}) 
    = \frac{1}{(2\pi\i)^2}\int_{\Gamma_s} \int_{ \Gamma_l}
    G_{i_1, i_2}(q', q)_{j_1+1,j_2+1}  \\
     e^{N (F(q'; \xi_{1,N}, \eta_{1,N})- F(q'; \xi_{2,N}, \eta_{2,N}))}  \frac{1}{z(z-z')} d z d z' .\label{eqn:double_int_simple}
\end{multline}
In the formula above,~$G_{i', i}(q', q)$ is a~$k \times k$ matrix valued function with meromorphic entries, given for~$i, i' = 0,\dots, \ell-1$ and~$q' = (z', w')$,~$q = (z, w)$ by
 \begin{align}
G_{i', i}(q', q) &= 
\left(\prod_{m=1}^{2i'+1}\phi_m(z')\right)^{-1} \frac{f(q')^N}{(w')^{k N}}\left(\frac{\psi_{0,+} \psi_{k N,-}}{\psi_{0,-} \psi_{0,+}}\right)|_{(z',w')} \notag \\
   & \times f(q)^{-N} \left(  
    \frac{\psi_{k N,+} \psi_{0,-}}{\psi_{0,-} \psi_{0,+}} \right)|_{(z,w)} \prod_{m=1}^{2 i}\phi_m(z)  .\label{eqn:Gdef22}
 \end{align}
 The function~$G_{i',i}$ is uniformly (in~$N$) bounded in compact subsets of the form~$U \times V \subset \mathcal{R} \times \mathcal{R}$ such that neither~$U$ nor~$V$ contain any angles. Furthermore, for any~$i',i=1,\dots,\ell$, if~$q' = (z,w')$ and~$q = (z,w)$ with~$w' \neq w$, then~$G_{i',i}(q',q) =0$.
 
 In addition, for the single integral we have 
 \begin{multline}\label{eqn:sing_int_simple}
     I_1(\mathrm b_{\ell x_2+i_2,ky_2+j_2},\mathrm w_{\ell x_1+i_1,ky_1+j_1}) =\\ 
     -\left(\frac{\onee \{\ell x_2+i_2 > \ell x_1+i_1\}}{2\pi\i}\int_{\Gamma}
  G_{i_1,i_2}(q,q)
 \frac{z^{y_1-y_2}}{w^{y_1-y_2}}\frac{d z}{z}\right)_{j'+1,j+1} 
 \end{multline}
 and
 \begin{equation}
 G_{i_1,i_2}(q,q) = 
\left(\prod_{m=1}^{2i_1+1}\phi_m(z)\right)^{-1}
\left(  
    \frac{\psi_{0,+} \psi_{0,-}}{\psi_{0,-} \psi_{0,+}} \right)|_{(z,w)}
     \prod_{m=1}^{2i_2} \phi_m(z).
 \end{equation}

\end{lem}
\begin{proof}
   If the edge weights satisfy Assumption 4.1 of~\cite{BB23}, then these formulas are exactly a consequence of Proposition 6.2 and Lemma 6.4 of~\cite{BB23}. Indeed, after replacing $i'$ and~$i$ there with~$2 i_1+1$ and~$2i_2$, the double integral~\eqref{eqn:double_int_simple} with~$G$ defined as in~\eqref{eqn:Gdef22} exactly matches the double integral in the statement of Proposition 6.2; one can check the match directly using the definition of~$F$, or by using Lemma 6.4 and then using (6.3) and (6.4) to observe that~$G(q_1,q_2)$ there (with the replacements indicated above) is exactly equal to our~$G_{i_1,i_2}$ here.
   
   By an analytic continuation argument, the same formulas are still valid with only Assumption~\ref{ass:1}. For completeness we give a proof in Appendix~\ref{app:A}.
   
\end{proof}

\begin{remark}
    Throughout this paper, expressions of the form~$d z_2 d z_1$, or~$\prod_i d z_i$, should not be confused with~$d z_2 \wedge d z_1$, or~$dz_1 \wedge \cdots$, respectively. The integrals in this paper are all~\emph{iterated contour integrals}, and the ``surface of integration'' (in the case of multiple variables) does not have a natural orientation, though each contour itself does. We will refer to quantities~$f(z,w) dz$ which transform as a~$1$ form in~$z$ and as a function in~$w$ as~$(1,0)$ forms. Similarly we refer to quantities~$f(z,w) \sqrt{dz} \sqrt{d w}$ as a~$(\frac{1}{2},\frac{1}{2})$ form, and so on.
\end{remark}

\subsection{Gaussian free field on~$\mathcal{R}_0$}
\label{subsec:Gff}
By an appropriate version of the Riemann uniformization theorem,~$\mathcal{R}_0$ can be conformally mapped to the unit disc with~$g$ circular holes cut out, which we call~$D$. For concreteness, throughout this section the reader may wish to identify~$\mathcal{R}_0$ with the domain~$D$. The domain~$D$ admits unique solutions to the Dirichlet boundary-value problem associated to the Laplacian, and on such a domain there exists a unique Green's function, see for example IV.2.8 in~\cite{FK92}. Since there exists a conformal  isomorphism~$\psi 
: \mathcal{R}_0 \rightarrow D$, this provides one way to define the Green's function on~$\mathcal{R}_0$;~$\mathcal{G}_{\mathcal{R}_0}(q, q') =\mathcal{G}_{D}(\psi(q), \psi(q')) $.

A simple characterization of the Green's function~$\mathcal{G}_{\mathcal{R}_0}$ on~$\mathcal{R}_0$ is the following: Suppose that for any~$q'$,~$G(q,q')$ vanishes as~$q \rightarrow \partial \mathcal{R}_0$, is harmonic in~$q$ for~$q \neq q'$ with respect to the Laplace operator~$\Delta = \partial \overline{\partial}$ on~$\mathcal{R}_0$ (which maps functions to~$2$ forms), and behaves as~$-\frac{1}{2 \pi} \log|z - z'| + O(1)$ for~$z$ near~$z'$ in local coordinates. This latter condition is independent of the choice of local coordinate. Then~$G = \mathcal{G}_{\mathcal{R}_0}$ is the Green's function on~$\mathcal{R}_0$.

Now we may define the \emph{Gaussian free field} on~$\mathcal{R}_0$, which we denote as $\mathfrak{g}_{\mathcal{R}_0}$. Although we consider this as a random ``function'' and will observe it as a limit of our dimer model height functions, it is actually not well defined at a point, and one should consider it as a stochastic process indexed by an appropriate set of measures. 

In particular, the \emph{Gaussian free field on~$\mathcal{R}_0$} is a random distribution, such that for any sufficiently regular test measures~$\mu_1,\dots, \mu_r$ on~$\mathcal{R}_0$,~$(\mathfrak{g}_{\mathcal{R}_0}(\mu_j))_{j=1}^r$ is a Gaussian random vector with covariance matrix
    \begin{equation}\label{eqn:gffmeasurecov}
       \mathbb{E}\left[  \mathfrak{g}_{\mathcal{R}_0}(\mu_i)\mathfrak{g}_{\mathcal{R}_0}(\mu_j)\right]=  \int_{\mathcal{R}_0}\int_{\mathcal{R}_0}\mathcal{G}_{\mathcal{R}_0}(x, y) \mu_i(dx) \mu_j(d y).
    \end{equation}
We remark that if we take a Gaussian free field~$\mathfrak{g}_{D}$ on the circle domain~$D$ described above, then we will have~$\mathfrak{g}_{\mathcal{R}_0} \coloneqq \psi^*  \mathfrak{g}_{D}$, where~$\psi$ is the conformal isomorphism from above and~$\psi^*$ denotes the pullback. Therefore, the reader may identify~$\mathcal{R}_0$ and~$D$, and simply consider~$\mathfrak{g}_{D}$ instead. See the surveys~\cite{She07, PW20} for a more detailed discussion and definition of the Gaussian free field on domains in~$\mathbb{R}^d$ (which by the previous remark is sufficient).

Although it is not well-defined at a point, this object can be thought of as the unique Gaussian process on~$\mathcal{R}_0$ which vanishes on~$\partial \mathcal{R}_0$, and its covariance structure can be formally defined as~$\mathbb{E}[\mathfrak{g}_{\mathcal{R}_0}(x) \mathfrak{g}_{\mathcal{R}_0}(y)] = \mathcal{G}_{\mathcal{R}_0}(x, y)$; covariances are given by the Green's function. In addition, higher moments can be formally computed as 
\begin{equation}\label{eqn:gffmoments}
\mathbb{E}\left[\prod_{i=1}^r \mathfrak{g}_{\mathcal{R}_0}(q_i)\right]
=
\begin{cases}
 \sum_{\text{pairings } \pi} \prod_{i=1}^{r/2} \mathcal{G}_{\mathcal{R}_0}(q_{\pi(2 i -1)}, q_{\pi(2 i)}) & r \text{ even}\\
0 & \text{otherwise}
\end{cases}.
\end{equation}
Though it is only a formal heuristic,~\eqref{eqn:gffmoments} is the characterization of the Gaussian free field that we will use, in the sense that~\eqref{eqn:gffmoments} is the expression that will arise from our calculations.

\subsection{Discrete Gaussian random variables}\label{subsec:discGauss}
The discrete Gaussian distribution arises naturally in various subfields of theoretical computer science such as~\emph{differential privacy}~\cite{CKS22} and~\emph{cryptography}~\cite{AR05}~\cite{Reg09}. In the setting of statistical mechanics, it appears in the description of large~$N$ statistics of eigenvalues of~$N \times N$ random Hermitian matrices in the~\emph{multi-cut} setting~\cite{BDE00}~\cite{BG24}~\cite{Shc13}.

The~\emph{discrete Gaussian distribution} with shift~$e \in \mathbb{R}^g$ and~\emph{scale matrix}~$\tau \in  \mathbb{C}^{g \times g}$, which we assume to be symmetric and pure imaginary with positive definite imaginary part, is given by the probability mass function
\begin{equation}\label{eqn:disc_theta}
    \mathbb{P}_{e,\tau}(n) = \frac{\exp\left(\i \pi (n-e) \cdot \tau (n-e)\right)}{\theta(- \tau e ; \tau) \exp(\i \pi e \cdot \tau e)} \qquad \qquad n \in \mathbb{Z}^g.
\end{equation}
This probability distribution is supported on~$\mathbb{Z}^g$. Here we have made the dependence on~$\tau$ in the theta function explicit to emphasize that~$\tau$ is a parameter and not necessarily equal to the period matrix, which we call~$B$, of the spectral curve~$\mathcal{R}$ of our dimer model. (In fact, when it appears in our work, the scale matrix of the discrete Gaussian will be equal to~$-B^{-1}$.)

The moment generating function of the discrete Gaussian can be computed as follows: For~$z \in \mathbb{C}^g$, and~$X = (X_1,\dots,X_g)$ distributed according to~$\mathbb{P}_{e,\tau}$,
\begin{equation}\label{eqn:DC_MGF}
\mathbb{E}_{e, \tau}\left[e^{(2\pi \i) z \cdot X }\right] =
\frac{\theta(z - \tau e ; \tau)}{\theta(- \tau e ; \tau) }.
\end{equation}
The description in terms of a theta function makes apparent the opportunity to relate probabilistic observables of the dimer model to the discrete Gaussian via analytic data on the spectral curve. In Section~\ref{subsec:DCdist}, we are able to match the~\emph{cumulants} of certain dimer model observables to the cumulants of the distribution~\eqref{eqn:disc_theta}, for an appropriate~$\tau$. By definition, these are given by logarithmic derivatives of~\eqref{eqn:DC_MGF}, see Appendix~\ref{subapp:jmjc}. For example, specializing to~$g=1$, we have for the mean
\begin{equation}\label{eqn:DGmean}
    \mathbb{E}_{e,\tau}[X] =\frac{1}{2\pi \i}\partial_z \log \theta(z-\tau e; \tau)|_{z=0} =  \frac{1}{2\pi \i}(\partial_z \log \theta)(-\tau e; \tau)
\end{equation}
where above~$\partial_z \log \theta \coloneqq \partial_z \log \theta(z; \tau)$. Similarly, for the variance
\begin{equation}\label{eqn:DGvar}
    \var_{e,\tau}[X] = \frac{1}{(2\pi \i)^2}(\partial_z^2 \log \theta)(-\tau e; \tau).
\end{equation}

\begin{remark}
   For generic parameters, \emph{the parameter~$e$ is not equal to the mean}. Indeed,~$e = \mathbb{E}_{e,\tau}[X]$ is an equation for the vanishing of a~$g$-tuple of meromorphic functions of~$e$ (c.f.~\eqref{eqn:DGmean} for~$g=1$), which for generic~$e, \tau$ is not satisfied. If~$g = 1$, for generic~$\tau$ it is only satisfied for~$e = 0$ mod~$\mathbb{Z}$ or~$e=\frac{1}{2} $ mod~$\mathbb{Z}$. 
\end{remark} 

\begin{remark}\label{rmk:covarm}
   Suppose~$X = (X_1,\dots,X_g)$ is distributed as~$\mathbb{P}_{e,\tau}$. The general genus~$g$ analog of formula~\eqref{eqn:DGvar} for the variance is the covariance formula
   \begin{equation}\label{eqn:DGcovar}
\Sigma_{i j} \coloneqq \mathbb{E}\left[(X_i-\mathbb{E}[X_i]) (X_j-\mathbb{E}[X_j])\right]   =   \frac{1}{(2\pi \i)^2}(\partial_{z_i} \partial_{z_j} \log \theta)(-\tau e; \tau).
   \end{equation}
It is clear from the definition of the density~\eqref{eqn:disc_theta} that~$v\cdot X = \sum_{i=1}^g v_i X_i$ has nonzero variance for any~$0 \neq v \in \mathbb{R}^g $. In other words, the matrix~$\Sigma$ given by~\eqref{eqn:DGcovar} is strictly positive definite. Moreover, from~\cite[Corollary 2.9]{AA19}, the set of pairs~$(\mu, \Sigma)$, with~$\mu \in \mathbb{R}^g$ and~$\Sigma \in \mathbb{R}^{g \times g}$ positive definite, is in bijection with the set of pairs~$(e, \tau) \in \mathbb{R}^g \times \i \mathbb{R}^{g \times g}$ of parameters for the discrete Gaussian distribution; the bijection taking~$(e,\tau)$ to~$(\mu, \Sigma)$ is given by taking the mean and covariance of~$X \sim \mathbb{P}_{e,\tau}$.
\end{remark}

\section{Fluctuations of the height function}\label{sec:heightflucts}
In this section we derive the leading order term of the moments of the height function. Our calculations rely on the large size limit of the inverse Kasteleyn matrix. We state the relevant asymptotics in Section~\ref{subsec:lemmastatements} and postpone their proofs to Section~\ref{sec:steepest_arguments}.

\subsection{Steepest descent lemmas}\label{subsec:lemmastatements}
In this subsection we state the various steepest descent lemmas we will utilize in our computation of the limiting height moments.

We will record the asymptotic behavior of the inverse Kasteleyn entries~$K^{-1}(\mathrm b,\mathrm w)$ in several different regimes, depending on which part of the domain contains the pairs of macroscopic coordinates~$(\xi_{1,N},\eta_{1,N})$ and~$(\xi_{2,N},\eta_{2,N})$ of~$\mathrm w$ and~$\mathrm b$, respectively. The following definitions depend on two positive constants~$c_1,c_2 > 0$, which we henceforth think of as fixed throughout. Their exact values are not important, and they can be chosen arbitrarily subject to the validity of bounds in the lemmas below; it turns out that the choice~$c_1=c_2=1$ works, so the reader may think that from here onwards. We separate the domain into four regions:
\begin{definition}[Regimes]\label{def:regimes}
 Let~$(\xi_{N},\eta_{N})$ be a rescaled position in the size~$k \ell N $ Aztec diamond. We define four subregions of~$[-1,1]^2$, which depend on~$N$, as follows. We define~$\eta_{fb} $ such that~$(\xi_{N},\eta_{fb}) \in \partial \mathcal{F}_R$ and~$\eta_{fb}$ is the closest value of~$\eta$ to~$\eta_N$ with this property. (In all situations where uniqueness is relevant,~$\eta_{fb}$ will clearly be unique.)
\begin{enumerate}[(I)]
    \item We say $(\xi_{N},\eta_{N})$ is \emph{in the bulk away from the edge} if~$(\xi_{N},\eta_{N}) \in \mathcal{F}_R$, and~$|\eta_{fb} -\eta_{N}| \geq c_1 N^{-\frac{1}{3}}$. In other words,~$(\xi_{N},\eta_{N})$ is at a distance~$\gtrsim N^{-\frac{1}{3}}$ from the arctic boundary.
    \label{item:first}
    
    \item  We say $(\xi_{N},\eta_{N}) $ is \emph{in the bulk near the edge} if~$(\xi_{N},\eta_{N}) \in \mathcal{F}_R $ and~$ c_1 N^{-\frac{1}{3}} \geq |\eta_{N}-\eta_{fb}| \geq c_2 N^{-\frac{2}{3}} $. \label{item:second}
    \item Suppose that for small~$\delta > 0$,~$(\xi_N,\eta_{fb}+\delta)$ is in a frozen or gaseous facet (in the other situation, when~$(\xi_N,\eta_{fb}-\delta)$ is in the facet, we make the obvious modification to the following definition). We say~$(\xi_{N},\eta_{N}) $ is \emph{at the edge} if~$N^{\frac{-2}{3} +\frac{1}{100}} \geq \eta_{N} - \eta_{fb} \geq -c_2 N^{-\frac{2}{3}}$.
    \label{item:third}
    
    \item We say that~$(\xi_{N},\eta_{N})$ is \emph{inside a facet} if it is inside of a frozen or gaseous facet (a connected component of~$[-1,1]^2 \setminus \mathcal{F}_R$) and if we have~$|\eta_{fb} - \eta_{N}| > N^{\frac{-2}{3} +\frac{1}{100}}$. 
    ~\label{item:fourth}
\end{enumerate}
\end{definition}

\begin{remark}\label{rmk:reg}
    The choices of exponents in each of the regimes above are important; the asymptotic behavior of~$K^{-1}(\mathrm b, \mathrm w)$ will have a different form for each pair of regimes occupied by the two vertices~$(\mathrm b, \mathrm w)$, so each pair of regimes requires a separate analysis. Regime~\eqref{item:first} is the main regime; it is when both vertices are in the bulk, i.e. in the interior of the liquid region sufficiently far from the arctic curve.  Regime~\eqref{item:third} is a small band around the arctic curve, and regime~\eqref{item:second} is a crossover regime between \eqref{item:first} and~\eqref{item:third}. Regime~\eqref{item:fourth} is the set of vertices sufficiently far into the interior of a facet.

    We note the choice that ``distance to the arctic curve'' is measured with the coordinate~$\eta$, rather than, e.g. with the coordinate~$\xi$. This is because in all of our steepest descent lemmas below, we use~$z$ as a distinguished local coordinate; if we had used~$w$ instead, it would have been natural to replace~$\eta$ with~$\xi$ in the definitions of the regimes above. These two choices are essentially equivalent when the vertex in question is away from points in the arctic curve with a horizontal or a vertical tangent. These correspond to branch points of~$(z,w) \mapsto w$ and~$(z,w) \mapsto z$, respectively. Hence, choosing~$z$ as a preferred local coordinate leads to choosing~$\eta$ as the preferred direction.
\end{remark}

Next, we describe the behavior of the critical point map at the arctic curve. We work in terms of the local coordinate~$z$, and we denote the composition of the critical point map with the map~$q = (z,w) \mapsto z$ by~$z(\xi,\eta)$.

\begin{lem}\label{lem:sqrt_singularity}
    Let~$(\xi, \eta_{fb}) \in \partial \mathcal{F}_R$ be a rescaled position on the arctic curve which does not have a vertical tangent, and which is not at a cusp and or a tangency point. Suppose that~$(\xi, \eta) $ is inside of~$\mathcal{F}_R$, the liquid region, for~$|\eta - \eta_{fb}| = \epsilon > 0$ sufficiently small. Denoting~$z_+(\xi,\eta) \coloneqq z(\xi, \eta)$ and~$z_-(\xi,\eta) \coloneqq \overline{z(\xi,\eta)}$, we have the following asymptotic equivalence in the local coordinate~$z$ for small~$\epsilon > 0$: For some nonzero~$a \in \mathbb{R}$,
    \begin{align}
z_{\pm}(\xi, \eta) = 
z_{\pm}(\xi, \eta_{f b}) \pm \i a \sqrt{| \eta - \eta_{fb}|} + O(\epsilon) .\label{eqn:equivimz}
\end{align}
Second, if we consider a point~$(\xi, \eta)$ which is just inside the facet such that~$|\eta - \eta_{fb}|$ is small, then both critical points near~$z(\xi, \eta_{fb})$, which we again denote by~$z_{\pm}(\xi, \eta)$, are real valued, and they have the asymptotic behavior
\begin{align}\label{eqn:equivimz2}
z_{\pm}(\xi, \eta) =z(\xi, \eta_{fb})\pm  a \sqrt{| \eta - \eta_{fb}|} + O(\epsilon) 
\end{align}
for the same nonzero real constant~$a$.

Finally, in the setting of either~\eqref{eqn:equivimz} or~\eqref{eqn:equivimz2}, we have, using the same notation for either case
\begin{align}
         |F''(z_{\pm}(\xi, \eta); \xi, \eta_{fb} \pm \epsilon)| &= b \sqrt{\epsilon} + O(\epsilon) \label{eqn:Fpp} 
    \end{align} 
for some constant~$b > 0$. 

Furthermore, the error terms in the approximations above hold uniformly for~$(\xi, \eta_{f b})$ in the arctic boundary~$\partial \mathcal{F}_R$, as long as~$(\xi, \eta_{f b})$ stays bounded away from a neighborhood of cusps, tangency points, and points of~$\partial \mathcal{F}_R$ with slope~$\infty$.
\end{lem}

In the next lemma, we record the asymptotic of~$K^{-1}(\mathrm b_{\ell x_2+i_2,k y_2+j_2},\mathrm w_{\ell x_1+i_1,k y_1+j_1})$  when both vertices are in region~\eqref{item:first}, in the liquid region sufficiently far from the arctic curve. In the lemma below, and in the lemmas which follow, we denote the macroscopic coordinates of~$\mathrm w$ and~$\mathrm b$ by~$(\xi_{j,N}, \eta_{j,N}) = (x_j \frac{2}{k N}-1,y_j \frac{2}{\ell N}-1)$, for~$j=1,2$, and we denote~$F_1(q) \coloneqq F(q; \xi_{1,N}, \eta_{1,N})$ and~$F_2(q) \coloneqq F(q; \xi_{2,N}, \eta_{2,N})$. Furthermore, denote by~$q_j = (z_j ,w_j)$ the critical point of~$F_j$, for~$j=1,2$, and also denote by~$F_j''(z)$ the second derivative~$(\frac{d}{d z})^2 F_j(z, w(z))|_{z = z(q)}$ of ~$F_j$ when it is written in terms of the local coordinate~$z$ near~$q = (z,w)$. Recall also the definition of~$G_{i', i}(q_1, q_2)$ from~\eqref{eqn:Gdef22}.

\begin{lem}[Steepest descent, both points in the bulk]\label{lem:steepest}
Suppose that both~$(\xi_{j,N}, \eta_{j,N}) \in \mathcal{F}_R$,~$j=1,2$, are in region~\eqref{item:first}, and are bounded away from the cusps, tangency points, and points in the arctic curve with a vertical tangent. Assume in addition that~$|(\xi_{1,N}, \eta_{1,N})-(\xi_{2,N},\eta_{2,N})| > N^{-\frac{1}{16}}$. Then we have
 \begin{align}
&K^{-1}(\mathrm b_{\ell x_2+i_2,ky_2+j_2},\mathrm w_{\ell x_1+i_1,ky_1+j_1})   \notag \\
   &=- \frac{1}{2 \pi} \bigg( e^{N (F_1(q_1)- F_2(q_2))} \frac{1}{N \sqrt{-F_1''(z_1) } \sqrt{F_2''(z_2) } } \frac{1}{z_2 (z_2- z_1)} G_{i_1,i_2}(q_1, q_2)_{j_1+1,j_2+1}   \label{eqn:dint1}  \\
     &+ e^{N (F_1(q_1)- F_2(\overline{q_2}))} \frac{1}{N \sqrt{-F_1''(z_1) } \sqrt{F_2''(\overline{z_2}) } } \frac{1}{\overline{z_2} (\overline{z_2}- z_1)} G_{i_1,i_2}(q_1, \overline{q_2})_{j_1+1,j_2+1}  
     \label{eqn:dint2} \\
     &+ e^{N (F_1(\overline{z_1})- F_2(z_2))} \frac{1}{N \sqrt{-F_1''(\overline{z_1}) } \sqrt{F_2''(z_2) } } \frac{1}{z_2 (z_2- \overline{z_1})} G_{i_1,i_2}(\overline{q_1}, q_2)_{j_1+1,j_2+1}  
     \label{eqn:dint3} \\
     &+ e^{N (F_1(\overline{q_1})- F_2(\overline{q_2}))} \frac{1}{N \sqrt{-F_1''(\overline{q_1}) } \sqrt{F_2''(\overline{q_2}) } } \frac{1}{\overline{z_2} (\overline{z_2}- \overline{z_1})} G_{i_1,i_2}(\overline{q_1}, \overline{q_2})_{j_1+1,j_2+1} \bigg)   \label{eqn:dint4} \\
     &+  e^{N (\Re[F_1(q_1)]- \Re[F_2(q_2)])} \frac{1}{2 \pi N \sqrt{|F_1''(z_1)|} \sqrt{|F_2''(z_2)|}}O(N^{-\frac{1}{8}}) .\notag 
\end{align}
\end{lem}

\begin{remark}\label{rmk:branchrmk}
    If we swap indices~$1$ and~$2$ in Lemma~\ref{lem:steepest}, a factor of~$\frac{1}{\sqrt{F_1''(z_1^{\pm})}}$ appears, where~$z_1^+ = z_1$ and~$z_1^- = \overline{z}_1$, rather than~$\frac{1}{\sqrt{-F_1''(z_1^{\pm})}}$. In both cases the branches of the square root are will be clear from the proof (given in Section~\ref{sec:steepest_arguments}), and the only property of the square roots that we need is that they satisfy the relation 
\begin{equation}\label{eqn:sqrtbranch}
   \i  \frac{1}{\sqrt{-F_1''(z_1^{\pm})}} = \frac{1}{\sqrt{F_1''(z_1^{\pm})}} .
\end{equation}

\end{remark}

In the next lemma, we will prove an asymptotic equivalence in the case that at least one of the vertices is near the edge of the liquid region,~\eqref{item:second}. We keep the same notation as in Lemma~\ref{lem:steepest}. We again use the notations defined before the statement of Lemma~\ref{lem:steepest}. We again suppose that both points remain bounded away from the cusps in the arctic curve for all~$N$. We will also assume that they are away from vertical tangents, which are branch points of the covering~$(z,w) \mapsto z$ from~$\mathcal{R} \rightarrow \mathbb{C}\cup \{\infty\}$.

\begin{lem}[Steepest descent close to the edge]\label{lem:nearedge_steepest}
Suppose that at least one of, or possibly both of~$(\xi_{j,N}, \eta_{j,N})$, are in region~\eqref{item:second}, approaching the arctic boundary~$\partial \mathcal{F}_R$, subject to the assumptions described above. If only one point satisfies this, then suppose that the other point is in the bulk, region~\eqref{item:first}. Suppose also that~$|(\xi_{1,N},\eta_{1,N}) - (\xi_{2,N},\eta_{2,N})| > N^{-\frac{1}{16}}$. Then for some~$C_2 > 0$,
 \begin{equation}\label{eqn:edge_bound}
|K^{-1}(\mathrm b_{\ell x_2+i_2,ky_2+j_2},\mathrm w_{\ell x_1+i_1,ky_1+j_1})|   
   \leq C_2 \frac{e^{N (\Re[F_1(q_1)]- \Re[F_2(q_2)])} |G_{i_1,i_2}(q_1, q_2)_{j_1+1,j_2+1} |}{N \sqrt{-F_1''(z_1) } \sqrt{F_2''(z_2) } |z_2- z_1|}  .
\end{equation}
\end{lem}

Now we assume that at least one of the points is at the edge in region~\eqref{item:third}. More precisely, suppose that one or both of~$(\xi_{j,N}, \eta_{j,N})$,~$j=1,2$, is in regime~\eqref{item:third}, and if not both, the other is in region~\eqref{item:first} or \eqref{item:second}. We also make the same assumptions on each pair of coordinates as in the previous two lemmas, namely that they are bounded away from cusps, tangency points, and points with slope~$\infty$ on the arctic curve.

We need a notation for the double critical point at the arctic curve. Suppose that the point~$(\xi_{j,N}, \eta_{j,N})$ is in region~\eqref{item:third}. By definition, there is a nearby point~$(\xi, \eta_{j,fb} ) \in \partial \mathcal{F}_R$ in the arctic curve. Define
\begin{align}
	\tilde{F}_j(q) \coloneqq F(q; \xi_{j, N}, \eta_{j, f b}).
\end{align}
By its definition,~$d \tilde{F}_j(q)$ has a double critical point, which we denote by~$\tilde{q}_j$.

Because the statement is similar if we swap the black and white vertices, it suffices to consider the case that only~$(\xi_{1,N}, \eta_{1,N})$ is in region~\eqref{item:third} at the edge, and~$(\xi_{2,N}, \eta_{2,N})$ may or may not also be.

\begin{lem}[Steepest descent, at least one point at the edge]\label{lem:edge_steepest}
 Suppose~$|(\xi_{1,N},\eta_{1,N}) - (\xi_{2,N},\eta_{2,N})| > N^{-\frac{1}{16}}$. If only~$(\xi_{1,N}, \eta_{1,N})$ is in region~\eqref{item:third}, then there exists some~$C_3 \geq 0$ such that for all~$N$ large enough we have
 
 \begin{equation}\label{eqn:one_edge_bound}
|K^{-1}(\mathrm b_{\ell x_2+i_2,ky_2+j_2},\mathrm w_{\ell x_1+i_1,ky_1+j_1})|  
   \leq C_3 \frac{e^{-N \Re[F_2(q_2)]}}{\sqrt{N} \sqrt{|F_2''(q_1)| }}
   \frac{e^{ N \Re[F_1(\tilde{q}_1)]} |G_{i_1,i_2}(q_1, q_2)_{j_1+1,j_2+1} |}{N^{\frac{1}{3}}  |z_2- z_1|}.
\end{equation}

Otherwise if both points are in region~\eqref{item:third}, then for some~$C_3>0$ 
 \begin{equation}\label{eqn:both_edge_bound}
|K^{-1}(\mathrm b_{\ell x_2+i_2,ky_2+j_2},\mathrm w_{\ell x_1+i_1,ky_1+j_1})|  
   \leq C_3 \frac{e^{N (\Re[F_1(\tilde{q}_1)]-\Re[F_2(\tilde{q}_2)])}|G_{i_1,i_2}(q_1, q_2)_{j_1+1,j_2+1} |}{N^{\frac{2}{3}}  |z_2- z_1|}.
\end{equation}
\end{lem}

Finally, we consider the setting when at least one point is in regime~\eqref{item:fourth} in a frozen or gaseous facet, and the other point is in any regime. Both points are again subject to the same assumptions as described prior to the statements of the previous three lemmas (bounded away from a finite number of special points on the arctic curve). 
\begin{lem}[Moment bound, at least one edge in a facet]\label{lem:facet_bound}
Suppose we have~$m$ edges~$\mathrm e_1,\dots, \mathrm e_m$, such that for any pair, the~$(\xi, \eta)$ coordinates of their black vertices are at least~$N^{-\frac{1}{16}}$ distance apart. Suppose that at least one edge is in region~\eqref{item:fourth} in any frozen or gaseous facet. Then, for all~$N$ large enough,
\begin{equation}\label{eqn:facet_moment_bound}
\left|\mathbb{E}\left[\prod_{j=1}^m (\mathbf{1}\{\mathrm e_k\} - \mathbb{E}\mathbf{1}\{\mathrm e_k\} )\right]\right| \leq e^{-N^{\frac{1}{200}}}.
\end{equation}
\end{lem}

\subsection{Limiting height moments}\label{subsec:limiting_hm}

The goal of this section is to prove Theorem~\ref{thm:full_moments} below, using the lemmas from the previous section. Tweaking the proof of the theorem slightly, we are also led to Corollaries~\ref{cor:facet_joint_moments_same} and~\ref{cor:facet_cov}. We will also use the notation
\begin{equation}\label{eqn:mean_sub_h}
    \overline{h}_N(\mathsf{f}) \coloneqq h_N(\mathsf{f}) - \mathbb{E}[h_N(\mathsf{f}) ]
\end{equation}
throughout this section; this should not be confused with~$\tilde{h}_N$ or~$\overline{h}$ from Section~\ref{subsec:res}.

Recall~$\psi_{0,\pm}(q)$ and~$\psi_{kN,\pm}(q)$ from the end of Section~\ref{subsec:spectral}. Define the~$(1,0)$ form~$\omega_0$ on~$\mathcal{R} \times \mathcal{R}$ via 
\begin{equation}\label{eq:omega_0_initial}
        \omega_0(q,q')=\frac{\psi_{k N,-}(q)\psi_{k N,+}(q')}{\psi_{0,-}(q')\psi_{0,+}(q')}\frac{d z }{w^{k N}(z'-z)}.
\end{equation}
As is apparent from the right hand side,~$\omega_0$ actually depends on~$N$. Later in Section~\ref{subsec:rewriting} we will write~$\omega_0$ (up to a conjugation) in terms of theta functions and prime forms on the Riemann surface~$\mathcal{R}$, which will be useful for understanding the formulas for the limiting height moments. In particular, it will clarify the precise form of the~$N$ dependence. The main result of this section is the following.

\begin{theorem}[General moment formula]
\label{thm:full_moments}
Fix an arbitrary compact subset~$K$ of the liquid region, and also fix a compact subset~$K_i$ of each gaseous facet, for~$i=1,\dots,g$. Assume that we have~$r$ faces~$\mathsf{f}_m$,~$m=1,\dots,r$, each with macroscopic coordinates~$(\xi_{m, f}, \eta_{m, f})$, some of which are in the compact subset of the liquid region and some of which are in a compact subset of a gaseous facet. Suppose that if (macroscopic coordinates of)~$\mathsf{f}_m,\mathsf{f}_l \in \mathcal{F}_R$, then~$|(\xi_{m,f}, \eta_{m,f}) - (\xi_{l,f}, \eta_{l,f})| > \epsilon$ for some~$\epsilon > 0$, and that for any pair of faces~$\mathsf{f}_m,\mathsf{f}_l$ in the same gaseous facet,~$|(\xi_{m,f}, \eta_{m,f}) - (\xi_{l,f}, \eta_{l,f})|  >N^{-\frac{1}{16}}$. Define~$q_m \coloneqq q(\xi_{m,f}, \eta_{m,f})$ if~$\mathsf{f}_m$ is in the liquid region, and otherwise let~$q_m$ be any point on the compact oval~$A_i$ corresponding to the gaseous facet containing~$(\xi_{m,f}, \eta_{m,f})$. Then, as~$N \rightarrow \infty$,
\begin{equation}\label{eqn:general_bulk_orgas}
\mathbb{E}\left[\prod_{j=1}^r (h_N(\mathsf{f}_j)-\mathbb{E}h_N(\mathsf{f}_j))\right] = \frac{1}{(2\pi \i)^r}\int_{\bar q_1 }^{q_1} \cdots \int_{\bar q_r}^{q_r} \det\bigg( (1-\delta_{l m})\omega_0(q_m', q_{l}') \bigg)_{m, l=1}^r + o(1). 
\end{equation}
For~$q_j \in \mathcal{F}_R$,~$\int_{\bar{q}_j}^{q_j}$ denotes integration over a path in~$\mathcal{R}_0$ obtained by gluing a path in~$\mathcal{R}_0$ from some (arbitrary) point in~$A_0$ to~$q_j$ together with its conjugate. If~$q_j$ is in a compact oval~$A_i$ (corresponding to a facet) we interpret~$\int_{\bar{q}_j}^{q_j}$ as an integral over the corresponding~$B$ cycle,~$B_i$. The~$o(1)$ error is uniform, i.e. depends only on the compact subsets, the minimal distance~$\epsilon$, and on the number of points~$r$. 
\end{theorem}

We will prove Theorem~\ref{thm:full_moments} by first proving an intermediate lemma, which is algebraic, and which will be useful in the computation. Recall the definitions of the symbols~$\phi_m(z)$,~$m=1,\dots, 2 \ell$, in~\eqref{eqn:oddphi} and~\eqref{eqn:evenphi}. We define the notation, for~$q = (z, w) \in \mathcal{R}$,
\begin{align}
V^{(i, +)}(q) &= \left(\prod_{m=1}^{2i+1}\phi_m(z)\right)^{-1} \psi_{0,+}(q)  \label{eqn:V+def}\\
V^{(i, -)}(q) &= \psi_{0,-}(q) \prod_{m=1}^{2 i}\phi_m(z).\label{eqn:V-def}
\end{align}
Note that~$V^{(i, +)}(q)$ is a column vector and~$V^{(i_r, -)}(q)$ is a row vector. Note also that
\begin{multline}\label{eqn:GVeq}
    G_{i_1,i_2}(q_1,q_2)_{j_1+1,j_2+1} \\
    = \frac{f(q_1)^N}{f(q_2)^N w_1^{k N}} \frac{\psi_{k N,-}(q_1) \psi_{k N,+}(q_2)}{\psi_{0,-}(q_1) \psi_{0,+}(q_1) \psi_{0,-}(q_2) \psi_{0,+}(q_2)} V^{(i_1, +)}(q_1)_{j_1+1} V^{(i_2, -)}(q_2)_{j_2+1} 
\end{multline}
for~$j_1,j_2 = 0,\dots, k-1$.

\begin{lem}\label{lem:periodsumlemma}
   Let~$\gamma^* = \{(b_m w_m)^*\}_{m=1}^p$ be a face path in the Aztec diamond which projects to one full horizontal loop (in the positive~$x$ direction) in the associated torus graph~$G_1$ and crosses edges with white vertices on the right. Define~$\mathrm b_{i_m, j_m} $ as the black vertex in the fundamental domain (i.e., in~$G_1$) associated to vertex~$b_m$, and similarly define~$\mathrm{w}_{i_m', j_m'}$. With~$K_{G_{1}}(z,w)$ and~$P(z,w)$ as defined as in Section~\ref{subsec:spectral}, we have, for~$q_1,q_2 \in \mathcal{R}$, 
   \begin{equation}\label{eqn:vid}
\sum_{m=1}^p  V^{(i_m',+)}(q)_{j_m'+1}  K_{G_1}(z,w)_{\mathrm w_{i_m',j_m'}, \mathrm b_{i_m,j_m}} V^{(i_m,-)}(q)_{j_m+1} =  \frac{ z \partial_z P(z,w)}{w \partial_w P(z,w)} \psi_{0,-}(q) \psi_{0,+}(q).
   \end{equation}

If instead~$\gamma^*$ crosses one full vertical period with white vertices on the right, then
\begin{equation}\label{eqn:vid2}
\sum_{m=1}^p  V^{(i_m',+)}(q)_{j_m'+1}  K_{G_1}(z,w)_{\mathrm w_{i_m',j_m'}, \mathrm b_{i_m,j_m}} V^{(i_m,-)}(q)_{j_m+1} = \psi_{0,-}(q) \psi_{0,+}(q).
   \end{equation}
\end{lem}
\begin{proof}
    Observe that the second statement in Lemma 4.25 of~\cite{BB23} (also using Lemma 6.1 there) is equivalent to
    \begin{equation}\label{eqn:Vid2}
 -\frac{1}{\psi_{0,-}(q) \psi_{0,+}(q)}   V^{(i',+)}(q)_{j'+1} V^{(i,-)}(q)_{j+1} = \frac{\adj K_{G_1}(z,w)_{\mathrm b_{i, j}, \mathrm w_{i', j'}} }{ w \partial_w P(z,w)}.
    \end{equation}
    This observation together with an argument given in the proof of Theorem 4.5 in~\cite{KOS06} completes the proof. Translating to our notation,~$Q(z,w)$ there is given by~$\adj K_{G_1}(z,w)$ here, and thus up to the overall constant~$-\frac{\psi_{0,-}(q) \psi_{0,+}(q)}{w \partial_w P(z,w)}$, we can identify the entries of~$U$ there with~$(V^{(i',+)}(q)_{j'+1})_{i',j'}$ and of~$V$ there with~$(V^{(i,-)}(q)_{j+1})_{i,j}$. Then, up to carefully keep track of signs,~\eqref{eqn:vid} and~\eqref{eqn:vid2} correspond to the two displays which follow Equation (9) of~\cite{KOS06}, respectively.

    \textit{Alternative Proof of~\eqref{eqn:vid2}}. We can instead prove~\eqref{eqn:vid2} by a direct computation. Suppose the dual path is exactly ``vertical'' (we can make this choice as the result is independent of the choice of path as long as it is homologous to one vertical loop around the torus, and crosses edges with white vertices on the right), in the same way as the paths depicted in Figure~\ref{fig:paths}. Then, for some~$i =0,\dots, \ell-1$, we can write the left hand side of~\eqref{eqn:vid2} as
    \begin{multline}\label{eqn:matmul}
    \sum_{j=0}^{k-1}\sum_{j'=j-1}^j  V^{(i, -)}(q)_{j+1}
    K_{G_1}(z,w)_{\mathrm w_{i, j'}, \mathrm b_{i, j}}
    V^{(i, +)}(q)_{j'+1} \\
    = \sum_{j=0}^{k-1} V^{(i, -)}(q)_{j+1} \sum_{j'=j-1}^j \phi_{2i+1}(z)_{j+1,j'+1} V^{(i, +)}(q)_{j'+1}.
    \end{multline}
    Then we observe that the inner sum on the right hand side is a matrix multiplication, so that the right hand side of~\eqref{eqn:matmul} above is given by
    \begin{equation}\label{eqn:mmfinal}
     V^{(i, -)}(q) \phi_{2 i +1}(z) V^{(i, +)}(q) = \psi_{0,-}(q) \psi_{0,+}(q)
    \end{equation}
    as desired.
\end{proof}

\begin{remark}\label{rmk:pref}
We note that while~\eqref{eqn:vid2} has a direct self-contained proof involving only the~$\phi_m$ (namely the alternative proof above),~\eqref{eqn:vid} does not, in the sense that it requires an identity of the form~\eqref{eqn:Vid2}, which follows from Lemma 4.25 in~\cite{BB23}.
\end{remark}

Consider~$r$ dual paths~$\gamma_1^*, \dots, \gamma_r^*$ which each traverse exactly one full horizontal or vertical period in the graph, so that edges are crossed with white vertices on the right (this is for simplicity only; a minus sign should be included if edges are crossed with the opposite orientation), and suppose~$\gamma_s^*$ consists of dual edges $\{ (\mathrm b_{i_{m_s},j_{m_s}} \mathrm w_{i_{m_s}',j_{m_s}'})^*\}_{s}$. The purpose of Lemma~\ref{lem:periodsumlemma} above is that it will allow us, in the course of our proof of Theorem~\ref{thm:full_moments}, to compute quantities of the form

\begin{equation}\label{eqn:st}
\sum_{m_1,\dots,m_r } \prod_{s=1}^r K_{G_1}(z_s, w_s)_{\mathrm w_{i_{m_s}',j_{m_s}'}, \mathrm b_{i_{m_s}, j_{m_s}}} \frac{G_{i_{m_s}, i_{m_{s+1}}'}(q_s, q_{s+1})_{j_{m_s}+1,j_{m_{s+1}}'+1}  }{z_{s}-z_{s+1}} 
 \end{equation}
  Indeed, based on the leading order asymptotic of the double integral formula (Lemma~\ref{lem:double_int_simple}) given in Lemma~\ref{lem:steepest}, we are led to compute such quantities in order to obtain a closed form expression for joint height moments.

We can begin to compute~\eqref{eqn:st} inductively as follows. Suppose that~$\gamma_2^*$ is over a vertical period; then, summing over~$m = m_2$ first and ignoring factors of~$f(q_i)$ which will ultimately cancel, by~\eqref{eqn:GVeq} and by~\eqref{eqn:vid2} in the previous lemma, we obtain
\begin{multline}\label{eqn:firstcase_sum}
\sum_{m } G_{i_1, i_m}(q_1, q_2)_{j_1+1,j_m+1}
K_{G_1}(z_2, w_2)_{\mathrm w_{i_{m}',j_{m}'}, \mathrm b_{i_{m}, j_{m}}}
G_{i_m', i_3}(q_2, q_3)_{j_m'+1,j_3+1}  = (\text{terms in~$q_1,q_3$}) \\
\times  \frac{\psi_{k N,-}(q_1)\psi_{k N,+}(q_2) \psi_{k N,-}(q_2)\psi_{k N,+}(q_3)}{w_2^{k N} \psi_{0,-}(q_2)\psi_{0,+}(q_2)}
.
 \end{multline}
If instead~$\gamma_2^*$ traverses a single vertical period of the lattice, then by~\eqref{eqn:vid} we would instead obtain
\begin{equation}\label{eqn:secondcasesum}
    (\text{terms in~$q_1,q_3$}) 
\times  \frac{\psi_{k N,-}(q_1)\psi_{k N,+}(q_2) \psi_{k N,-}(q_2)\psi_{k N,+}(q_3)}{w_2^{k N} \psi_{0,-}(q_2)\psi_{0,+}(q_2)}
 \frac{ z \partial_z P(z,w)}{w \partial_w P(z,w)}
\end{equation}
for the left hand side of~\eqref{eqn:firstcase_sum}.
 
 Therefore, we may continue inductively to see that \eqref{eqn:st} is equal to (recalling the definition~\eqref{eq:omega_0_initial} of~$\omega_0$) 
 \begin{equation}
       \prod_{i=1}^r  \frac{\omega_0(q_i, q_{i+1})}{dz_i} \prod_{i : \gamma_i^*  \text{ hor.}} (-\frac{z_i}{w_i}  \frac{ d w_i}{d z_i}).
 \end{equation}
Above, the inner product is over indices such that~$\gamma_i^*$ traverses a horizontal period, and we have used the fact that as~$z$ and~$w$ vary along~$\{P(z,w) = 0\}$,~$\frac{ d w}{d z} =- \frac{ \partial_z P(z,w)}{ \partial_w P(z,w)}$.


Now we move on to the proof of Theorem~\ref{thm:full_moments}. In the first part of the proof below, we assume the existence of face paths~$\gamma_j^*$ with certain properties, and we omit a detailed construction. Then we compute the left hand side of \eqref{eqn:general_bulk_orgas} by summing up certain the increments of the mean-subtracted height function over these paths. We then apply the steepest descent lemmas from the previous section and simplify the result, and also bound subleading contributions, in order to arrive at the result. The argument follows the one given in~\cite{BF14} (and also pursued in~\cite{Pet15},~\cite{Kua14},~\cite{Dui13}), though in those works there are no gaseous facets, and controlling the effect of these is new to this work.



\begin{proof}[Proof of Theorem~\ref{thm:full_moments}]
We consider~$r$ dual paths~$\gamma_1^*,\dots, \gamma_r^*$ starting from the outer boundary of the size~$k \ell N$ Aztec diamond and with the following properties:

\begin{itemize}

\item  For any two paths~$\gamma_i^*, \gamma_j^*$ crossing edges containing vertices with rescaled coordinates~$(\xi_{j}, \eta_{j}) , (\xi_{i}, \eta_{i}) \in \mathcal{F}_R$ which are both in one of the regions~\eqref{item:first},~\eqref{item:second}, or~\eqref{item:third}, we have~$|(\xi_{j}, \eta_{j}) - (\xi_{i}, \eta_{i})| > \epsilon' >0$, where~$\epsilon' > 0$ is fixed and only depends on the compact subsets~$K, K_i$ in the theorem statement, on the positive integer~$r$, and on the minimum distance~$\epsilon$ in the theorem statement. Furthermore, for each vertex passed by~$\gamma_i^*$, its rescaled coordinates~$(\xi_{i}, \eta_{i})$ are at least~$\epsilon'$ distance away from the rescaled~$(\xi, \eta)$ coordinates of any cusp or tangency point, and from that of any point of the arctic curve with a vertical tangent. 

\item The pair of paths only comes within~$\epsilon'$ of each other when both are inside of a compact subset of the same gaseous facet; this compact subset may be larger than the corresponding~$K_i$. Furthermore, even in the gaseous facets, they stay at a distance at least~$c N^{-\frac{1}{16}}$ from each other, where~$c > 0$ depends only on~$K, K_i$,~$r$, and~$\epsilon$.

\item Except for possibly in a finite (at lattice scale) neighborhood of each target face~$\mathsf{f}_i$, each path~$\gamma_i^*$ traverses the lattice in groups of dual edges consisting of either an entire horizontal period or an entire vertical period. We also assume that the number of turns is bounded uniformly in~$N$.

\item Paths cross arctic curves moving vertically, and these crossings are transversal by the first bullet point above.

\end{itemize}
We remark that the positions of the starting points of these dual paths on the boundary of the Aztec diamond are irrelevant, since we are computing moments of mean-subtracted heights.

With these paths (and the notation~\eqref{eqn:mean_sub_h}), we will compute
\begin{equation}
\mathbb{E}[\prod_{i=1}^r \overline{h}_N(\mathsf{f}_i)] 
= \sum  \mathbb{E}[\prod_{i=1}^r \Delta \overline{h}_N(\mathsf{f}_i')] \label{eqn:inc_sum}
\end{equation}
where~$\mathsf{f}_i'$ are intermediate faces along~$\gamma_i^*$, and the increments of $\overline{h}_N$ along an entire horizontal or vertical period (or possibly part of one, near the end of a path) are denoted~$\Delta h_N(\mathsf{f}_i')$. The sum in~\eqref{eqn:inc_sum}, then, is over all~$r$-tuples of increments.

\begin{figure}
\centering
\includegraphics[scale=.2]{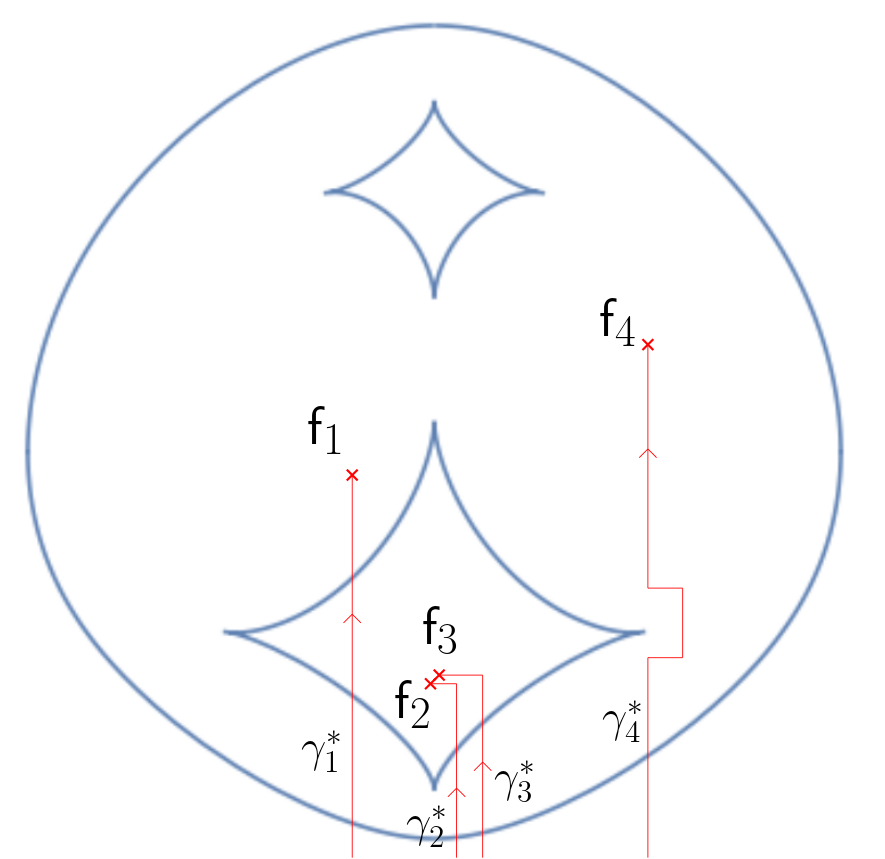}
\caption{Paths (in the dual) along which we choose to sum up increments~$\Delta h$. }
\label{fig:paths}
\end{figure}

 Now we would like to compute the leading order contribution from a tuple of increments. Before we begin, we note that for an~$r$-tuple of edges, by the determinantal structure of dimer statistics~\cite[Theorem 6]{Ken97}, the mean subtracted contribution of the product of height changes as paths cross these edges is the determinant of the corresponding submatrix of~$K^{-1}$, but with the diagonal entries set to~$0$, times the corresponding entries of the Kasteleyn matrix itself. For example, if~$r =2$, and we consider a pair of edges~$\mathrm e_1 = \mathrm{w}_{1}\mathrm{b}_{1}$ and~$\mathrm e_2 = \mathrm{w}_{2}\mathrm{b}_{2}$, then
 \begin{multline}\label{eqn:r2ex}
 \mathbb{E}\left[\left(\onee\{\mathrm e_1\}- \mathbb{E}[\onee\{\mathrm e_1\} ]\right)
\left(\onee\{\mathrm e_2\}- \mathbb{E}[\onee\{\mathrm e_2\} ]\right)  \right]\\ =
 -K(\mathrm w_1, \mathrm b_1) K(\mathrm w_2, \mathrm b_2)   
   K^{-1}(\mathrm b_{2},\mathrm w_{1}) K^{-1}(\mathrm b_{1},\mathrm w_{2}).
\end{multline}
In general, we get, for~$\mathrm e_i = \mathrm w_{i} \mathrm b_{i}$, 
\begin{equation}\label{eqn:sigma_term}
\mathbb{E}[\prod_{i=1}^r (\mathbf{1}\{\mathrm e_i\}-\mathbb{E}\mathbf{1}\{\mathrm e_i\} )] = \sum_{\sigma} \Sign(\sigma)  
\prod_{i=1}^c
\prod_{j=1}^{l_i}  K(\mathrm w_{a_j^i},\mathrm b_{a_j^i})  K^{-1}(\mathrm b_{a_{j+1}^i},\mathrm w_{a_j^i}) 
\end{equation}
where the sum is over permutations~$\sigma$ of~$\{1,\dots, r\}$ without fixed points, written in cycle notation as~$\sigma = (a_1^1 a_2^1,\dots,a_{l_1}^1) \cdots (a_1^c a_2^c,\dots,a_{l_c}^c)$. As a result the summand on the right hand side of~\eqref{eqn:inc_sum} is given by
\begin{equation}\label{eqn:deltah}
\mathbb{E}[\prod_{i=1}^r \Delta \overline{h}_N(\mathsf{f}_i')] = \sum_{\sigma} \Sign(\sigma)  
\prod_{i=1}^c
\sum_{\{(\mathrm b_{a_j^i} \mathrm w_{a_j^i})^*\}}\prod_{j=1}^{l_i}  K(\mathrm w_{a_j^i},\mathrm b_{a_j^i})  K^{-1}(\mathrm b_{a_{j+1}^i},\mathrm w_{a_j^i}) 
\end{equation}
where each inner summation is over a collection of dual paths (corresponding to the cycle in the permutation) moving from~$\mathsf{f}_i'$ to its translate by one horizontal or vertical period of the lattice.

\begin{figure}
\centering
\includegraphics[scale=.45]{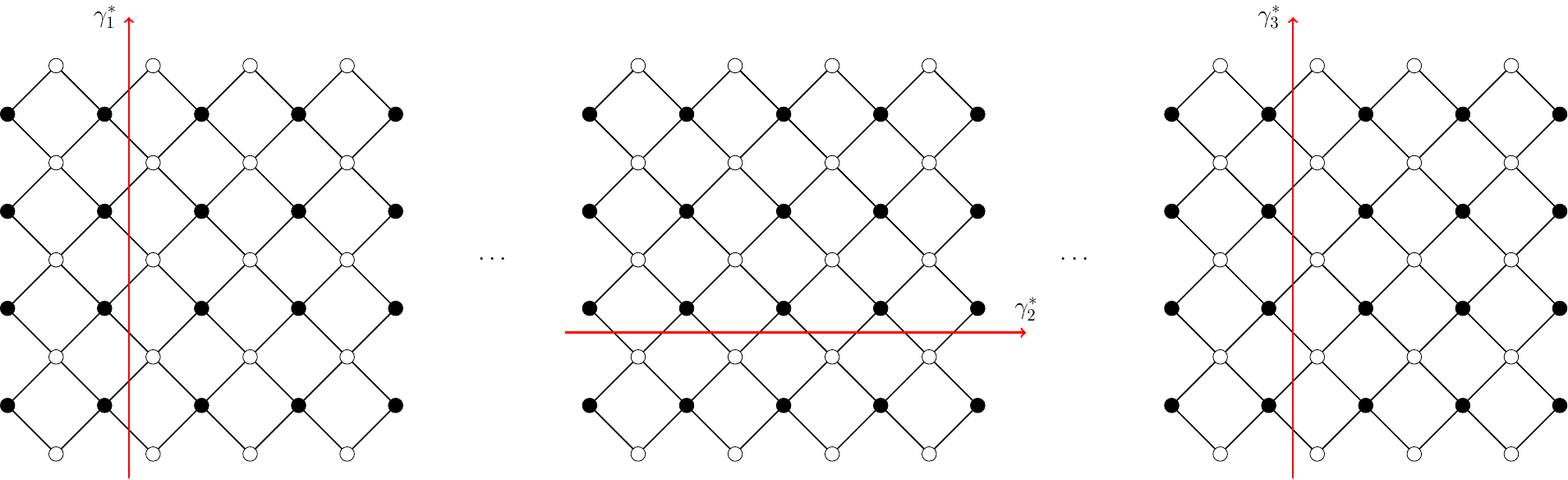}
    \caption{An illustration (for~$k=\ell=2$) of parts of dual paths contributing to a single term~\eqref{eqn:deltah} in the summation for the height moment. Here we have depicted~$r=3$ and we depict changes across a periods~$\mathrm f_1' \rightarrow \mathrm f_1' + (0,2 k)$,~$\mathrm f_2' \rightarrow \mathrm f_2' + (2 \ell,0)$, and~$\mathrm f_3' \rightarrow \mathrm f_3' + (0,2 k)$. Each picture occurs locally in a different part of the Aztec diamond.}
    \label{fig:term_in_sum}
\end{figure}

 The general approach we take follows that of~\cite{BF14} and~\cite{Pet15}; we will employ the steepest descent, Lemma~\ref{lem:steepest} to compute a summation of terms like~\eqref{eqn:deltah} over the~$r$ dual paths, and then observe that it is a Riemann sum for~\eqref{eqn:general_bulk_orgas}. We also need to control the contributions from parts of the sum near the arctic curves and inside of frozen regions, and also from the gaseous facets. The latter of these (bounding contributions from gas regions) is new to this work, while for the other parts, having established the requisite steepest descent computations with periodic weights, we follow previous works. Thus, we break up the sum~\eqref{eqn:inc_sum} into several cases, for which we utilize the subsets defined in Definition~\ref{def:regimes}.

\paragraph{Case 1, main contribution: The parts of paths where all of~$(\xi_{j,N}, \eta_{j,N})$,~$j=1,\dots,r$, are in regime~\eqref{item:first}, in the bulk away from the edge, at least~$c_1 N^{-\frac{1}{3}}$ away from the arctic boundary.}
Fix a permutation in~\eqref{eqn:deltah}, and fix a cycle in that permuation, and relabel the paths and vertices along paths so that the~$l$ edges in the cycle are~$\mathrm e_s =  \mathrm b_{\ell x_s + i_s, k y_s + j_s} \mathrm w_{\ell x_s' + i_s', \ell y_s' + j_s'} =:  \mathrm b_{s} \mathrm w_{s}$,~$s=1,\dots, l$. Now by Lemma~\ref{lem:steepest}, for each term in the inner summation in~\eqref{eqn:deltah} we obtain 
\begin{multline}\label{eqn:pm_prod0}
\sum_{\{+,-\}^{2 l}} \prod_{s=1}^l \frac{ K(\mathrm w_s, \mathrm b_s)}{N  |F_s''(z_s)| e^{\i \theta_s}} \prod_{s=1}^l  \frac{ G_{i_s',i_{s+1}}((q_s')^{\pm}, q_{s+1}^{\pm})_{j_s'+1,j_{s+1}+1} e^{N (F_s((q_s')^{\pm}) - F_{s+1}(q_{s+1}^{\pm}))}}{z_{s+1}^{\pm} (z_{s+1}^{\pm} - (z_s')^{\pm})} \\
+ O\left( \frac{N^{-\frac{1}{8}}}{N^l \prod_{s=1}^l |F_s''(z_s)|} \right) .
\end{multline}
In the summation above, one pair of elements from~$\{+,-\}$ is chosen for each term in the product,~$q_s^{+} \coloneqq q_s$,~$q_s^{-} \coloneqq \bar q_s$, and similarly for~$(q_s')^{\pm}$, and ~$\theta_s$ are certain angles. Moreover, above we have denoted~$q_s' \in \mathcal{R}_0$ as the critical point corresponding to normalized coordinates of~$\mathrm w_s$, and (slightly abusing notation)~$F_s((q_s')^{\pm}) \coloneqq F((q_s')^{\pm}; \xi_{s,N}', \eta_{s,N}')$, where~$(\xi_{s,N}', \eta_{s,N}')$ are normalized coordinates of~$(x_s', y_s')$. (Note that~$(x_s', y_s') = (x_s, y_s)$ unless if the edge crosses a fundamental domain.) Similarly,~$q_s$ is the critical point corresponding to~$\mathrm b_s$, and~$F_s(q_s^{\pm}) = F_s(q_s^{\pm}; \xi_{s,N},\eta_{s,N})$. 

For example, if~$l = r=2$, and~$(\ell x_s' + i_s', \ell y_s' + j_s') = (\ell x_s + i_s, \ell y_s + j_s) $ for~$s=1,2$, then the product of inverse Kasteleyn entries in~\eqref{eqn:r2ex} has the form
\begin{multline}
 \frac{1}{4 \pi^2} \bigg( e^{N (F_1(q_1)- F_2(q_2))} \frac{1}{N \sqrt{-F_1''(z_1) } \sqrt{F_2''(z_2) } } \frac{1}{z_2 (z_2- z_1)} G_{i_1,i_2}(q_1, q_2)_{j_1+1,j_2+1}   \\
     + e^{N (F_1(q_1)- F_2(\overline{q_2}))} \frac{1}{N \sqrt{-F_1''(z_1) } \sqrt{F_2''(\overline{z_2}) } } \frac{1}{\overline{z_2} (\overline{z_2}- z_1)} G_{i_1,i_2}(q_1, \overline{q_2})_{j_1+1,j_2+1}  
     \\
     + e^{N (F_1(\overline{q_1})- F_2(q_2))} \frac{1}{N \sqrt{-F_1''(\overline{z_1}) } \sqrt{F_2''(z_2) } } \frac{1}{z_2 (z_2- \overline{z_1})} G_{i_1,i_2}(\overline{q_1}, q_2)_{j_1+1,j_2+1}  
 \\
     + e^{N (F_1(\overline{q_1})- F_1(\overline{q_2}))} \frac{1}{N \sqrt{-F_1''(\overline{z_1}) } \sqrt{F_2''(\overline{z_2}) } } \frac{1}{\overline{z_2} (\overline{z_2}- \overline{z_1})} G_{i_1,i_2}(\overline{q_1}, \overline{q_2})_{j_1+1,j_2+1} \bigg)  
\\
\times \bigg( e^{N (F_2(q_2)- F_1(q_1))} \frac{1}{N \sqrt{-F_2''(z_2) } \sqrt{F_1''(z_1) } } \frac{1}{z_1 (z_1- z_2)} G_{i_2,i_1}(q_2, q_1)_{j_2+1,j_1+1}   \\
     + e^{N (F_2(q_2)- F_1(\overline{q_1}))} \frac{1}{N \sqrt{-F_2''(z_2) } \sqrt{F_1''(\overline{z_1}) } } \frac{1}{\overline{z_1} (\overline{z_1}- z_2)} G_{i_2,i_1}(q_2, \overline{q_1})_{j_2+1,j_1+1}  
     \\
     + e^{N (F_2(\overline{q_2})- F_1(q_1))} \frac{1}{N \sqrt{-F_2''(\overline{z_2}) } \sqrt{F_1''(z_1) } } \frac{1}{z_1 (z_1- \overline{z_2})} G_{i_2,i_1}(\overline{q_2}, q_1)_{j_2+1,j_1+1}  
 \\
     + e^{N (F_2(\overline{q_2})- F_1(\overline{q_1}))} \frac{1}{N \sqrt{-F_2''(\overline{z_2}) } \sqrt{F_1''(\overline{z_1}) } } \frac{1}{\overline{z_1} (\overline{z_1}- \overline{z_2})} G_{i_2,i_1}(\overline{q_2}, \overline{q_1})_{j_2+1,j_1+1} \bigg) 
\\
    +\frac{1}{N^2 |F_1''(z_1)| |F_2''(z_2)|} O(N^{-\frac{1}{8}}) \label{eqn:decomp}
\end{multline}
where we have used the expansion of $K^{-1}(\mathrm b_{\ell x_2+i_2,k y_2+j_2},\mathrm w_{\ell x_1+i_1,k y_1+j_1})$ given in Lemma~\ref{lem:steepest}. Expanding out the product leads to an expression of the form~\eqref{eqn:pm_prod0}.

When we take the product of expressions like~\eqref{eqn:pm_prod0} over all cycles (recall we restricted to one cycle of length~$l$), we again obtain an error term which is 
$$
O\left(\frac{N^{-\frac{1}{8}}}{N^r \prod_{s=1}^r |F_s''(z_s)|} \right).
$$
Since~$1/|F_m''|$ has a square root singularity (which is integrable) as we get near the arctic curve by Lemma~\ref{lem:sqrt_singularity}, the total contribution to the sum~\eqref{eqn:inc_sum} of the error terms as above for edges in region~\eqref{item:first} is~$O(N^{-\frac{1}{8}})$, and can be discarded; we will omit these errors in what follows. 

Before explaining how to compute the main contribution, we will bound the contribution from oscillatory terms. Namely, when we expand a product of inverse entries from applying Lemma~\ref{lem:steepest} as  in~\eqref{eqn:pm_prod0}, terms in which signs~$+,-$ in the exponential are not chosen consistently will be highly oscillatory. They will contain a factor~$e^{ \pm 2 \i N \Im[F_s(q_s)]}$, for some~$s$.

\textbf{Oscillatory terms in case 1:} We argue that the sum of highly oscillatory terms described in the previous paragraph do not contribute in the limit. For this, it is sufficient to bound a summation which, has the form
\begin{multline}
\sum_{(x_2,y_2),\dots, (x_r,y_r)}\frac{1}{N^{r-1} \prod_{m=2}^r|F_m''(q_m)|}
 \\
 \times \sum_{(x_1,y_1)}  \frac{ \phi((x_1/N,y_1/N),\dots, (x_{r}/N,y_{r}/N))}{N  |F_1''(q_1)| } e^{ 2 \i N \Im[F(q(\xi_{1,N}, \eta_{1,N}); \xi_{1,N}, \eta_{1,N})]} .\label{eqn:osc}
\end{multline}
Above the inner sum is over a set of~$(x_1, y_1)$ (with normalized coordinates~$(\xi_{1,N},\eta_{1,N})$) which index fundamental domains crossed by the path~$\gamma_1^*$ in region~\eqref{item:first}, and where~$\phi$ is a smooth  function which is bounded in the region containing the paths. We argue that we can bound the inner summation by a quantity which is~$O(N^{-\frac{1}{6}})$ uniformly in~$(x_2,y_2),\dots,(x_r,y_r)$. This is almost exactly the situation considered in~\cite[Lemma 4.1]{Kua14}, and in~\cite[Section 5.3]{BF14} in the proof of Theorem 1.3. In particular, compare~\eqref{eqn:osc} with (5.51) in the latter (note that there, the factors analogous to~$\prod_m |F_m''(q_m)| $ have been absorbed into~$\phi$). They proceed by grouping the inner summation of~\eqref{eqn:osc} into~$O(N^{\frac{2}{3}})$ many chunks of~$\sim N^{\frac{1}{3}}$ lattice sites, and we may proceed in the same way, making sure that along each chunk only one of~$x_1,y_1$ is varying. There is one essential difference in our case: For us it is not necessarily the case that
\begin{equation}\label{eqn:goodvals}
|e^{2 \i \frac{d}{d \eta_1} \Im[F_1(q_1)]}-1| \gtrsim N^{-\frac{1}{6}} ,
\end{equation}
and similarly, in case we sum over the horizontal direction, it is not necessarily true that $|e^{2 \i \frac{d}{d \xi_1} \Im[F_1(q_1)]}-1| \gtrsim N^{-\frac{1}{6}} $.
Compare this with the discussion after (5.53) in~\cite{BF14}. This issue also occurs in~\cite{Kua14}. 

In fact, with an appropriate choice of argument in the definition of~$F$, $ \frac{1}{\pi} |\frac{d}{d \eta_1} \Im[F_2(q_2)]| \in (0, k) $, which we can see from the relationship of this quantity with the~$\eta$ slope of the height function, explained in~\cite[Proposition 4.22]{BB23}. However, possibly after perturbing the path~$\gamma_1^*$ by an arbitrarily small (at macroscopic scale) distance,  there are only~$O(N^{\frac{5}{6}})$ ``bad values'' of~$(x_1,y_1)$, where~$q_1$ is bounded away from~$\partial \mathcal{R}_0$ and either $e^{2 \i \frac{d}{d \eta_1} \Im[F_1(q_1)]}$ or~$e^{2 \i \frac{d}{d \xi_1}\Im[F_1(q_1)]}$ is within~$O(N^{-\frac{1}{6}})$ of~$1$. This is because the set of points in the rescaled Aztec diamond such that~$\partial_{\eta} \Im[F(q(\xi,\eta); \xi, \eta)] \in \pi \mathbb{Z}$ is a finite collection of smooth curves, so (possibly after perturbing~$\gamma_1^*$), in a vicinity of each intersection with this set,~$\gamma_1^*$ will only cross~$O(N^{-\frac{1}{6}} N) = O(N^{\frac{5}{6}})$ fundamental domains not satisfying the property~\eqref{eqn:goodvals}. Therefore, a trivial bound on~\eqref{eqn:osc} for these values of~$(x_1,y_1)$ gives a negligible~$  O(N^{-\frac{1}{6}}) $ contribution to the sum, and the rest of the sum can be dealt with using the arguments of~\cite{BF14}. The argument there shows that we get a contribution of~$O(N^{-\frac{1}{6}})$ for the rest of the summation.

\textbf{Main terms in case 1:} Now we compute the main contribution, coming from non oscillatory terms. This computation also follows~\cite{BF14},~\cite{Pet15}, though the computation is slightly more involved due to the periodic edge weights. To handle summing over both horizontal and vertical periods, we will invoke Lemma~\ref{lem:periodsumlemma}.

First we must compute a leading order expression for~\eqref{eqn:deltah}. Thus, we must compute a sum of expressions of the form~\eqref{eqn:pm_prod0}, which we recall is the contribution to a term of the form~\eqref{eqn:sigma_term} of a single length~$l$ cycle of a fixed permutation~$\sigma$. The edges in the cycle are~$(\mathrm e_1 \mathrm e_2 \cdots \mathrm e_l)$. We sum up the contributions from~\eqref{eqn:pm_prod0} as the collection of edges in the cycle~$\mathrm e_s = \mathrm b_{\ell x_s + i_s, k y_s + j_s} \mathrm w_{\ell x_s' + i_s', k y_s' + j_s'}$,~$s=1,\dots,l$, range over a horizontal or vertical period for each~$s$.

From our earlier discussion about oscillatory terms, the terms in~\eqref{eqn:pm_prod0} where the signs are not chosen consistently can be ignored. First, we consider the contribution from the terms in~\eqref{eqn:pm_prod0} corresponding to taking all~$+$'s. We claim that this gives 

\begin{multline}\label{eqn:persum_asympt}
 \left(  \prod_{s=1}^l \sum  V^{(i_s',+)}(q)_{j_s'+1}  K_{G_1}(z,w)_{\mathrm w_{i_s',j_s'}, \mathrm b_{i_s,j_s}} V^{(i_s,-)}(q)_{j_s+1} \right) \\
 \times \frac{1}{(2 \pi \i)^{l} N^l } \prod_{s=1}^l \frac{1}{z_s F_s''(z_s)}\frac{\psi_{kN,-}(q_{s}) \psi_{kN,+}(q_{s+1})}{w_s^{k N}\psi_{0,-}(q_{s}) \psi_{0,+}(q_{s}) \psi_{0,-}(q_{s+1}) \psi_{0,+}(q_{s+1}) (z_s - z_{s+1})}
\end{multline}
where each inner sum on the first line is over the vertical or horizontal period, where all edges are crossed with white vertices on the right. To get~\eqref{eqn:persum_asympt} we used~\eqref{eqn:GVeq}. The reason for the appearance of~$K_{G_1}$ even though we are not on the torus graph, is the fact that edges crossing between fundamental domains will indeed pick up a factor of~$z^{-1}$ or~$w$ (depending on which boundary they cross) due to the fact that~$N F(q;\xi_{s,N},\eta_{s, N}) = - x_s \log w + y_s \log z + \log f(q)$, c.f.~\eqref{eqn:Fdef}. We also used the fact (see Remark~\ref{eqn:sqrtbranch}) that the branches of the square roots satisfy~$\frac{1}{\sqrt{-F_s''(z_s)}} = \frac{1}{\i\sqrt{F_s''(z_s)}}$.

We claim that
\begin{multline}
\eqref{eqn:persum_asympt} = \frac{(-1)^l}{(2\pi \i)^l N^l} \prod_{s \in H} \frac{2}{k}\partial_{\xi}z (\xi_{s,N}, \eta_{x, N})  \prod_{s \in V} \frac{2}{\ell} \partial_{\eta}z (\xi_{s,N}, \eta_{x, N}) \\
\times \prod_{s = 1}^l \frac{\psi_{kN,-}(q_{s}) \psi_{kN,+}(q_{s+1})}{w_s^{k N} \psi_{0,-}(q_{s+1}) \psi_{0,+}(q_{s+1}) (z_s - z_{s+1})} . \label{eqn:rs1}
\end{multline}
Here~$H$ is the set of indices~$s$ where the path moves across a horizontal period, while~$V$ are those moving along a vertical period. Indeed, we use~\eqref{eqn:vid} and~\eqref{eqn:vid2} from Lemma~\ref{lem:periodsumlemma}. By differentiating the relation~$(\partial_z F)(z(\xi, \eta); \xi, \eta) \equiv 0$ (here we are writing~$F$ in the local coordinate~$z$) in~$\eta$ and using the explicit form of~$F(q; \xi, \eta)$, we obtain
\begin{equation}\label{eqn:area_element}
\frac{2}{\ell} \partial_{\eta}  z(\xi_{m,N},\eta_{m,N}) =- \frac{1}{ F_m''(z_m) z_m},\qquad m = 1,2,\dots,l.
\end{equation}
and we observe that~$\frac{1}{N F_s''(z_s) z_s} = \frac{2}{\ell N} \partial_\eta z$, which accounts for the $V$ factors. Similarly, for the horizontal periods, we use
\begin{equation}\label{eqn:area_element_hor}
\frac{2}{k} \partial_{\xi} w(\xi_{m,N},\eta_{m,N}) =\frac{1}{ F_m''(w_m) w_m},\qquad m = 1,2,\dots,l
\end{equation}
where now~$F_m''(w_m)$ stands for~$F_m$ differentiated twice in~$w$ after writing it in terms of the coordinate~$w$, and evaluating it at~$w(\xi_{m,N},\eta_{m,N})$. Since~$d F_s(q_s) = 0$, we have~$\frac{1}{F_s''(z_s) z_s} = \frac{1}{(\frac{d w}{d z})^2 F_s''(w_s) z_s}    $. Now the horizontal periods also have an extra factor of~$\frac{z \partial_z P(z,w)}{w \partial_w P(z,w)} =-\frac{z}{w} \frac{d w}{d z}$ (Lemma~\ref{lem:periodsumlemma}), which leads to
$$
-\frac{z_s}{w_s} \frac{d w}{d z} \frac{1}{F_s''(z_s) z_s}
=
-\frac{d w}{d z} \frac{1}{(\frac{d w}{d z})^2 F_s''(w_s) w_s} = -\frac{2}{k} \frac{d z}{d w} \partial_{\xi } w = -\frac{2}{k}\partial_{\xi} z
$$

One can also obtain~\eqref{eqn:rs1} by a direct calculation in the case that all periods are vertical periods; see the alternative proof of Lemma~\ref{lem:periodsumlemma}, and also see Remark~\ref{rmk:pref}.

Now we sum~\eqref{eqn:rs1} over all~$(x_s,y_s)$ in region~\eqref{item:first}, that is, over the fundamental domains along the paths. Recall that the summand in~\eqref{eqn:rs1} is not uniformly bounded in~$N$ for all points, due to a singularity near the arctic curve. However, by Lemma \ref{lem:sqrt_singularity}, the factors~$F_j''(z(\xi_{j,N}, \eta_{j,N}))$ have a square root behavior as~$\eta_{j,N} \rightarrow \eta_{fb}$ with~$\xi$ fixed, where~$\eta_{fb}$ denotes the~$\eta$ coordinate of a nearby arctic boundary. The fact that~$\frac{1}{\sqrt{x}}$ is integrable at~$0$ implies that the sum of terms like~\eqref{eqn:rs1} indeed converges to an integral. Recall we were looking at a single cycle in a fixed permutation~$\sigma$. If we take the product over cycles and then sum over permutations, the limiting integral is a contour integral in~$\mathcal{R}$, which is equal to

\begin{equation}\label{eqn:general_bulk_plus_final}
\frac{1}{(2\pi \i)^r}\int_{q_0^{(1)}}^{q_1} \cdots \int_{q_0^{(r)}}^{q_r} \det\bigg( (1-\delta_{l m})\omega_0(q_m, q_{l}) \bigg)_{m, l=1}^r 
\end{equation}
where~$q_0^{(m)}$ are points on the outer oval of~$\mathcal{R}_0$ (recall that this corresponds to the liquid-frozen boundary). The error in approximating the sum by the integral is at most~$O(N^{-\frac{1}{50}})$ (this is not nearly optimal, but is all we need) by adding up the errors from the Riemann sum approximation, the oscillatory terms, and the error terms from steepest descent estimates as in~\eqref{eqn:pm_prod0}. Note that if any of the faces~$\mathsf{f}_i$ is in the~$m$th gaseous facet, then~$q_i$ is on the corresponding compact oval~$A_m$.

We may similarly compute contributions from remaining consistent choices of signs in~\eqref{eqn:pm_prod0}, and adding those to ~\eqref{eqn:general_bulk_plus_final} above, we  obtain
\begin{equation}\label{eqn:general_bulk_final}
\frac{1}{(2\pi \i)^r}\int_{\bar q_1}^{q_1} \cdots \int_{\bar q_r}^{q_r} \det\bigg( (1-\delta_{l m})\omega_0(q_m, q_{l}) \bigg)_{m, l=1}^r .
\end{equation}
This is exactly the expression in the theorem statement.

 All that remains is to show that the remaining parts of the sum over regions~\eqref{item:second}-\eqref{item:fourth} are negligible. Regions~\eqref{item:second} and~\eqref{item:third} correspond to cases (a/2) and (a/1) in~\cite{BF14}, respectively. Thus, we follow the arguments outlined there.

\paragraph{Case 2, error term: At least one of~$(\xi_{j,N}, \eta_{j,N})$ is in region~\eqref{item:second}, that is, in the bulk close to the edge, and all others are either also in~\eqref{item:second}, or in region~\eqref{item:first}.}
 Without loss of generality we assume that it is~$(\ell x_1 + i_1, k y_1 + j_1)$ which is close to the edge. Recall that the dual paths are vertical near the edge of the arctic curve. It suffices to, for fixed~$i_1,j_1,i_1',j_1'$, bound the contribution to~\eqref{eqn:inc_sum} of the sum over~$(x_1,y_1)$ of the contribution from edge~$\mathrm b_{\ell x_1 + i_1, k y_1 + j_1} \mathrm w_{\ell x_1' + i_1', k y_1' + j_1'}$ in the fundamental domain of~$(x_1,y_1)$; for then, summing over the (finitely many) edges inside of each fundamental domain yields a bound on the sum over the part of~$\gamma_1^*$ in region~\eqref{item:second}. Denote by~$S_2$ the set of indices~$ (x_1,y_1) $ such that for the edge $b_{\ell x_1 + i_1, k y_1 + j_1} \mathrm w_{\ell x_1' + i_1', k y_1' + j_1'}$ along~$\gamma_1^*$,~$(\ell x_1 + i_1, k y_1 + j_1)$ has normalized coordinates~$(\xi_{1,N}, \eta_{1,N})$ in region~\eqref{item:second}. Using Lemma~\ref{lem:nearedge_steepest}, for some~$C_2 > 0$ (possibly larger than in the lemma statement),   
\begin{align}
\sum_{(x_1,y_1) \in S_2} & \bigg|  K^{-1}(\mathrm b_{\ell x_2+i_2,ky_2+j_2},\mathrm w_{\ell x_1'+i_1',ky_1'+j_1'})  K^{-1}(\mathrm b_{\ell x_1+i_1,ky_1+j_1},\mathrm w_{\ell x_l+i_l,ky_l+j_l})   \bigg| \label{eqn:edgesum1} \\
&\leq C_2 \frac{e^{N F_l(q_l)} e^{-N F_2(q_2)}}{N \sqrt{|F_2''(z_2)| |F_l''(z_l)|}}  \sum_{(x_1,y_1) \in S_2} \frac{1}{N |F_1''(z_1)|}. \label{eqn:edgesum2}
\end{align}
Recall that by Lemma~\ref{lem:sqrt_singularity}, we have~$|F_1''(z_1)| \sim \sqrt{|\eta_{1,N} - \eta_{1,fb}|}$, where~$\eta_{1,fb}$ is the vertical coordinate of the nearby point on the arctic curve. Therefore, we have
\begin{align*}
 \sum_{S_2} \frac{1}{N |F_1''(q_1)|} &\leq C   \int_{c_2' N^{-\frac{2}{3}}}^{c_1' N^{-\frac{1}{3}}} \frac{1}{\sqrt{x}} dx \\
&\leq \tilde{C} N^{-\frac{1}{6}} .
\end{align*}

Now, the contribution of a single cycle in a fixed permutation~$\sigma$ in the formula~\eqref{eqn:sigma_term}, involves the product of~\eqref{eqn:edgesum1} with other inverse Kasteleyn entries (and Kasteleyn entries), and this is ultimately summed over the other paths. The outer summations (not shown in~\eqref{eqn:edgesum1}) can be upper bounded by a constant by similar arguments as in this step and the previous step, since the summand is bounded in compact subsets of the liquid region and has the integrable square root singularity at the edge.

\paragraph{Case 3, error term: At least one of~$(\xi_{j,N}, \eta_{j,N})$ is in region~\eqref{item:third} at the edge, and other points are in regions~\eqref{item:third}~\eqref{item:second}, or~\eqref{item:first}.}
Again, we follow the arguments of~\cite[Section 5.3]{BF14}. This case corresponds to case (a/1) there, except that we cut off the sum when the point goes deeper inside of the facet (see Case 4 below). We assume it is~$(\xi_{1,N}, \eta_{1,N})$ which is in region~\eqref{item:third}. Call~$S_3$ the set of such~$(x_1, y_1)$ such that the black vertex~$(\ell x_1 + i_1, k y_1 +j_1)$ of an edge at a fixed position in a fundamental domain has normalized coordinates in regime~\eqref{item:third}. By~\eqref{eqn:one_edge_bound} in Lemma~\ref{lem:edge_steepest}, we have
\begin{multline}
\sum_{(x_1,y_1) \in S_3}  \bigg|  K^{-1}(\mathrm b_{\ell x_2+i_2,ky_2+j_2},\mathrm w_{\ell x_1+i_1,ky_1+j_1})  K^{-1}(\mathrm b_{\ell x_1+i_1,ky_1+j_1},\mathrm w_{\ell x_l+i_l,ky_l+j_l})   \bigg| \\
\leq |(\text{terms in other variables})| \times   \sum_{(\xi_{1,N},\eta_{1, N}) \in S_3}   \frac{C_3}{N^{\frac{2}{3}}  }
\label{eqn:S3bound}
\end{multline}
for some constant~$C_3> 0$; the quantity ``terms in other variables'' is similar to the prefactor in the second line of~\eqref{eqn:edgesum1}, but the expression should be changed accordingly if the other points are in regime~\eqref{item:third}. There are~$O(N^{\frac{1}{3} +\frac{1}{100}})$ such points by the definition of region~\eqref{item:third}. We simply have a constant~$C_3 > 0$ in~\eqref{eqn:S3bound} by the fact that the points in the Aztec diamond (recall the properties of the paths we chose, and in particular that pairs of points along paths which are both in regions~\eqref{item:first},~\eqref{item:second}, or~\eqref{item:third} remain bounded away from each other) are separated, and since the inverse of the critical point map,~$q \mapsto (\xi(q), \eta(q))$, is~$C^1$ up to the boundary of~$\mathcal{R}_0$~\cite{BB24},~\cite{BB23}. Therefore, the overall contribution is of order~$O(N^{-\frac{1}{3}})$.

\paragraph{Case 4, error term: At least one of~$(\xi_{j,N}, \eta_{j,N})$ is inside of a facet~\eqref{item:fourth}, and the other points are either also inside a facet, or in regions~\eqref{item:first},~\eqref{item:second}, or~\eqref{item:third}.}
Assume~$(\xi_{1,N}, \eta_{1,N})$ is inside a facet. It follows from Lemma~\ref{lem:facet_bound} that such summands decay exponentially in a small positive power of~$N$, and thus do not contribute. 


This completes the proof.
\end{proof}

The proof of the previous theorem can also be used to provide a moment bound which we will need in our asymptotic analysis. Below, we again use the notation~$\overline{h}_N(\mathsf{f}) = h_N(\mathsf{f}) - \mathbb{E}[h_N(\mathsf{f})] $.

\begin{cor}\label{cor:facet_joint_moments_same}
Suppose the faces~$\mathsf{f}_1,\dots, \mathsf{f}_{r_1}$ are bounded away from each other by~$\epsilon>0$ at the macroscopic scale and inside of a compact subset~$K \subset \mathcal{F}_R$ of the liquid region, and that the faces~$\mathsf{f}_{1}',\dots, \mathsf{f}_{r_2}'$ are inside of compact subsets of gaseous facets. Let~$n_1,\dots, n_{r_2} \geq 1$ be integers. Then, there exists a constant~$C >0 $ depending only on~$\epsilon$,~$K$,~$r \coloneqq r_1 + r_2$, the compact subsets of the facets, and~$n_1,\dots, n_p$, such that for all~$N > 0$
\begin{equation}\label{eqn:gaseous_bdd}
\mathbb{E}\left[ \prod_{j=1}^{r_1}\overline{h}_N(\mathsf{f}_j) \prod_{j=1}^{r_2} \overline{h}_N(\mathsf{f}_j')^{n_{j}}\right] \leq C . 
\end{equation}
\end{cor}
\begin{proof}
We may again proceed as in the proof of Theorem~\ref{thm:full_moments}; in particular, we may use paths satisfying the same properties as described there. In particular, we may choose paths such that any pair of paths can only be within~$O(N^{-\frac{1}{16}})$ distance from each other when they are both within an~$O(N^{-\frac{1}{16}})$ neighborhood of the same endpoint~$\mathsf{f}_j'$ in the gaseous facet. 


We may bound the summation by looking at regions~\eqref{item:first}-\eqref{item:fourth} separately, as in the proof of Theorem~\ref{thm:full_moments}. In fact, the proof there implies the boundedness of the sum over parts of paths which do not include any pair of vertices within~$O(N^{-\frac{1}{16}})$ of each other. 

Thus, all that is left is the part of the summation where at least two edges are within~$O(N^{-\frac{1}{16}})$ of the same face~$\mathsf{f}_1'$ in a gaseous facet. If at least two edges, say just~$\mathrm e_1$ and~$\mathrm e_2$ for example, are within this distance of~$\mathsf{f}_1'$, and others, let us call them $\mathrm e_3,\dots, \mathrm e_r$, are far away from~$\mathsf{f}_1'$ and far from each other, then upon examining the proof of Lemma~\ref{lem:facet_bound}, we see that the only permutations (which arise in~\eqref{eqn:r2ex}) that can contribute asymptotically are ones which are permutations of~$\{1, 2\} \times \{3,\dots, r\}$; in other words, a decoupling occurs, and the correlation function factors as a product of the correlations of edges nearby to~$\mathsf{f}_1'$ with the correlations of those that are not.

As a result of the previous discussion, we see that it is sufficient to bound the contribution to the sum when all~$e_1,\dots, e_r$ are within~$O(N^{-\frac{1}{16}})$ of a single face~$\mathsf{f}$. As is observed in the proof of Lemma~\ref{lem:facet_bound}, this is nothing but a Gibbs measure calculation. In particular the claimed boundedness follows from the exponential decay of entries of the inverse (after a gauge) in the gaseous Gibbs measure. 
\end{proof}

A slightly more refined version of the proof above in the case~$r=2$ yields the following. Define, for~$i = 0,\dots,\ell-1$,
\begin{equation}\label{eqn:tilde_G}
    \widetilde{\mathcal{G}}_i(q',q) \coloneqq 
    \phi_{2 i + 1}(z) G_{i i}(q,q)  \phi_{2 i + 1}(z') G_{i i}(q',q').
\end{equation}
The corollary which follows (and the notation above) will not be used anywhere else in this work, but we record the computation in case it is of independent interest.

\begin{cor}[Of the proof of Theorem~\ref{thm:full_moments}]\label{cor:facet_cov}
    Suppose that the face~$\mathsf{f} = (2 \ell x + 2 i, 2 k y + 2 j + 2)$ is inside of a compact subset of a gaseous facet. Then,  we have
\begin{multline}\label{eqn:facet_var}
       \mathbb{E}[\overline{h}_N(\mathsf{f})^2] =
      -\frac{1}{(2\pi \i)^2}  \int_{\bar q_1}^{q_1}\int_{\bar q_2}^{q_2} \omega_0(q_1', q_2') \omega_0( q_2',q_1')   \\
      + \frac{1}{(2 \pi \i)^2}  \int_{\bar q_1}^{q_1} \int_{\bar q_2}^{q_2} \frac{ \Tr(\widetilde{\mathcal{G}}_{i_1}(q_1', q_2'))}{(z_1'- z_2')^2} d z_1' d z_2' + o(1)
    \end{multline}
as~$N\rightarrow \infty$, where~$q_1, q_2$ are any two points on the compact oval corresponding to the facet of~$\mathsf{f}$. 
\end{cor}
\begin{proof}

 To prove the second formula~\eqref{eqn:facet_var}, we proceed in a similar way as in the proof of Theorem~\ref{thm:full_moments}, with small modifications. We start with two vertical paths moving from the boundary of~$[-1,1]^2$ to~$(\xi_{1,N}, \eta_{1,N})$ and satisfying the same properties as in the proof of Corollary~\ref{cor:facet_joint_moments_same} (see the beginning of that proof). To be concrete, we choose one of them to start from the top boundary and one to start from the bottom boundary, and choose them to be vertical (possibly with small detours around cusps, tangency points, or points with vertical tangents on the arctic curve). In particular, in a neighborhood of~$\mathsf{f}$, the paths approach the face~$\mathsf{f}$ from directly above and directly below.

 With this choice of paths, the part of the double sum up until both points are inside of the same gaseous facet in region~\eqref{item:fourth}, will contribute the expression on the first line of~\eqref{eqn:facet_var}. This expression is the same one as in~\eqref{eqn:most_general_joint_moment}, and the convergence to that can be proved by following the proof of Theorem~\ref{thm:full_moments} in the special case~$r=2$.

 The second term is the contribution from the gaseous Gibbs measure which is observed inside of the facet. Indeed, from the part of the summation~\eqref{eqn:inc_sum} where the edge in the top path is within~$M \gtrsim N^{\frac{15}{16}}$ fundamental domains of the point~$(\xi_{1,f}, \eta_{1,f})$ in the gaseous region, and where the edge in the bottom path is within~$M$ fundamental domains of the point, we get the following contribution: 

 \begin{multline}
\sum_{y_1,y_2}\sum_{j_1,j_2} \frac{1}{(2 \pi \i)^2} \int_{\Gamma_1} \int_{\Gamma_1'}   
 \bigg(  \gamma_{j_1+1,i+1} \gamma_{j_2+1,i+1} G_{i i}(q, q)_{j_1+1,j_2+1} G_{i i}(q', q')_{j_2+1,j_1+1} \\
  + \alpha_{j_1,i+1}r_{j_1}(z) \gamma_{j_2+1,i+1} G_{i i}(q, q)_{j_1,j_2+1} G_{i i}(q', q')_{j_2+1,j_1+1}\\
+ \gamma_{j_1+1,i+1} \alpha_{j_2,i+1}r_{j_2}(z') G_{i i}(q, q)_{j_1+1,j_2+1} G_{i i}(q', q')_{j_2,j_1+1}\\
+ \alpha_{j_1,i+1} \alpha_{j_2,i+1} r_{j_1}(z) r_{j_2}(z') G_{i i}(q, q)_{j_1, j_2+1} G_{i i}(q', q')_{j_2,j_1+1} \bigg) \left(\frac{z}{z'} \right)^{y_1-y_2}  \frac{d z}{z}\frac{d z'}{z'}\\
= \frac{1}{(2 \pi \i)^2} \int_{\Gamma_1} \int_{\Gamma_1'} 
\Tr(\widetilde{\mathcal{G}}_i(q',q)) \sum_{y_1=0}^{M-1} \left(\frac{z}{z'} \right)^{y_1} \sum_{y_2=-M}^{-1} \left(\frac{z}{z'} \right)^{-y_2} \frac{d z}{z}\frac{d z'}{z'}.
\label{eqn:4block_var}
\end{multline}
Above, we used contours~$\Gamma_1$ and~$\Gamma_1'$ which both are single integral contours for the integral~$I_1$ in Lemma~\ref{lem:double_int_simple} (see the description of the contours before the lemma, noting that the contour is called~$\Gamma$ there). Also, we have used the notation~$r_j(z) \coloneqq \mathbf{1}\{j > 0\} + \frac{1}{z} \mathbf{1}\{j = 0\}$. The summations in the final line are geometric sums. Computing these leaves us with a term equal to the last line of~\eqref{eqn:facet_var}, and it also leaves us with terms of the form 
\begin{equation}\label{eqn:local_var_errors}
    \frac{1}{(2 \pi \i)^2} \int_{\Gamma_1} \int_{\Gamma_1'}  (\text{function independent of~$N$}) \times \left( \frac{z}{z'}\right)^M d z d z'
\end{equation}
(possibly with~$M$ replaced by~$2 M$). Such terms can be controlled by observing that the entries of~$G_{i i}(q,q)$ which contribute to the smooth part of the integrand do not have poles at~$\{q_{\infty, j}\}$ (this can be seen from Lemma 4.25 and Lemma 5.2 in~\cite{BB23}, note that in particular the right hand side of the display in Lemma 5.2 with~$i = i'$ has an empty product involving~$E(q_{\infty, m}, q)$). Therefore, the~$q$ contour can be deformed so that its image in the amoeba is horizontal and moves from the compact oval in the negative~$\log |z|$ direction. Also, the~$q'$ contour can be deformed to move horizontal in the positive~$\log |z'|$ direction. Therefore, the terms like~\eqref{eqn:local_var_errors} decay exponentially.
\end{proof}

\section{Gaussian free field and discrete component}\label{sec:GFF_DC}
In this section, we first define the discrete components in~\ref{subsec:disc_comp} below. Then we present an expression for the leading order asymptotic of joint moments of the discrete component together with several values of the height function at distinct locations, which we derive using the limiting expression for joint moments obtained in the previous section. Finally, we analyze the expression we derive in order to show that after subtracting off the discrete components, the height field converges to a Gaussian free field in the conformal structure of~$\mathcal{R}_0$. In addition, we show that the discrete components are asymptotically independent from this Gaussian free field. Finally, after this in Section~\ref{subsec:DCdist} we characterize the distribution of the tuple of discrete components; namely, we identify it with a certain~\emph{discrete Gaussian} distribution.

Throughout this section, in particular in Proposition~\ref{prop:DCcov}, Theorem~\ref{thm:multipt_conv}, Proposition~\ref{prop:ind_moments}, and Theorem~\ref{thm:discrete_comp_thm}, we make statements about large~$N$ asymptotics of various probabilistic quantities. We remark, however, that no new asympototic analysis is performed; the proofs contain only manipulations of limiting formulas resulting from the analysis in the previous section (which in turn relies on the asymptotic analysis in Section~\ref{sec:steepest_arguments}). Nevertheless, we stick to the language of asymptotics in the statements of this section in order to keep the probabilistic objects being studied at the forefront of the exposition. Throughout this section, it is important to note in several cases, perhaps most notably in Theorem~\ref{thm:discrete_comp_thm}, that the statements involve objects which generically are oscillatory in~$N$ and which~\emph{do not converge} as~$N \rightarrow \infty$.

\begin{figure}
    \centering
\includegraphics[width=0.5\linewidth]{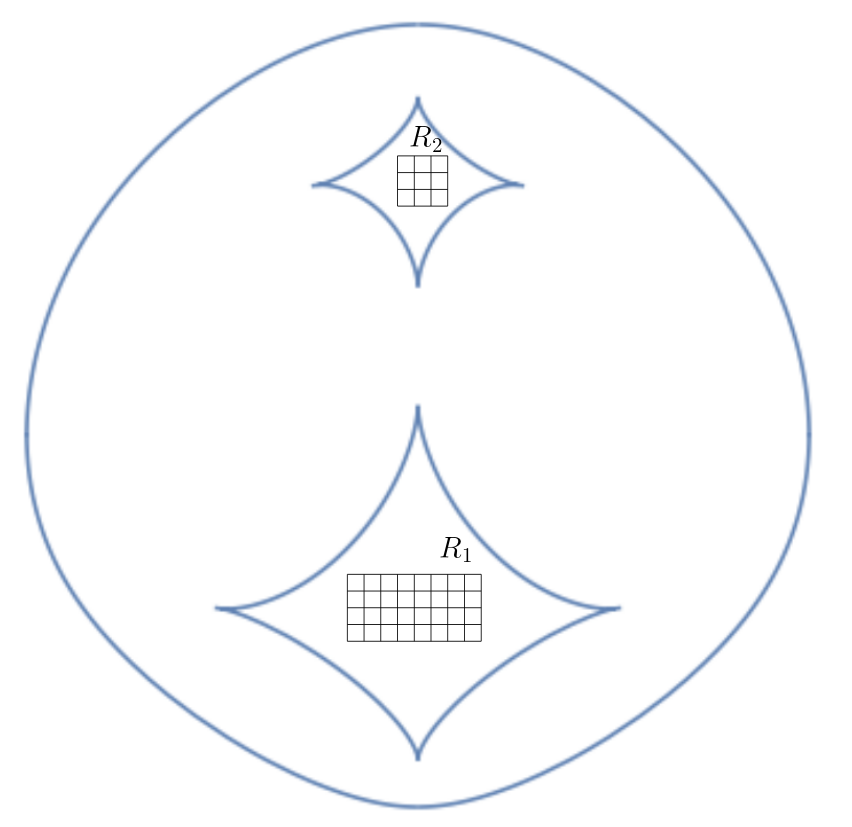}
    \caption{Each entry of the discrete component is defined as an average of mean-subtracted height function values at faces near the lattice sites of a mesoscopic grid inside of a rectangle in a gaseous facet. The grid spacing is~$\sim N^{-\frac{1}{16}}$.}
    \label{fig:dc_grids}
\end{figure}

\subsection{The discrete components}
\label{subsec:disc_comp}
Let~$\mathfrak{g}_1,\dots,\mathfrak{g}_g \subset [-1,1]^2$ denote the gaseous facets. Each one contains in its interior some rectangle of small enough macroscopic width and height. Consider a rectangle~$R_j \subset \mathfrak{g}_j$. Let~$(2 \ell x_0 + 2 i_{0}, 2 k y_{0} + 2 j_{0} + 2)$ be the face closest to the bottom left corner of~$R_j$. We will consider the set of faces~$\{\mathsf{f}_r\}_{r=1}^M \coloneqq \{ (2 \ell x_0 + 2 i_{0} + 2 \lfloor s N^\frac{15}{16}\rfloor, 2 k y_{0} + 2 j_{0} + 2  \lfloor p N^\frac{15}{16}\rfloor + 2) \}_{0\leq s \leq W_j,0 \leq p \leq H_j}$, where~$W_j$ is the maximal value of~$s$ such that~$(2 \ell x_0 + 2 i_0 + \lfloor s N^{\frac{15}{16}} \rfloor, 2 k y_0 + 2 j_0 + 2) \in R_j$, and similarly for~$H_j$. The exact values of~$W_j$ and~$H_j$ are not important. Moreover, the exact choices of the macroscopic rectangles (so long as they are macroscopic), are not important, and nor is the exact choice of mesoscopic scale~$N^{-\frac{1}{16}}$ (which leads to sampling faces at a spacing of~$\sim N^{\frac{15}{16}}$). The subscript~$r$ (used again below) indexes this set of faces in any order. See Figure~\ref{fig:dc_grids} for an example. Below we (as in Section~\ref{subsec:limiting_hm}) denote~$\overline{h}_N(\mathsf{f}) \coloneqq h_N(\mathsf{f}) - \mathbb{E}h_N(\mathsf{f})$.

\begin{definition}[Discrete components]\label{def:discrete_comp}
  For fixed~$j$, with the set of faces described above, we define the~\emph{discrete component}~$Z_j$ as follows  
  \begin{equation}\label{eqn:discrete_comp}
    Z_j = Z_j(N) \coloneqq \frac{1}{M} \sum_{r=1}^M \overline{h}_N(\mathsf{f}_r). 
\end{equation}
\end{definition}

In other words, we average the mean-subtracted values of~$h_N(\mathsf f)$ at faces inside of~$R_j$ at lattice points of a grid with mesh size~$O(N^{-\frac{1}{16}})$ at the macroscopic scale. There are~$M \sim N^{\frac{1}{8}}$ such points. Note that each~$Z_j$ actually depends on~$N$, i.e. each~$Z_j = Z_j(N)$ is a sequence of random variables, but in what follows we suppress this from the notation.

\begin{remark}\label{rmk:mes_s}
    As we remarked before Definition~\ref{def:discrete_comp}, the exact choice of rectangles and of mesoscopic scale~$N^{-\frac{1}{16}}$ are not important. We believe that, via minor modifications of our proofs, it is possible to show that the large~$N$ distribution of~$(Z_1,\dots,Z_g)$ investigated below would remain the same as long as the mesoscopic scale is~$o(1)$.  As an example, the same results should hold if we define~$Z_j$ as the average of~$h_N - \mathbb{E}[h_N]$ over~\emph{all} faces inside the rectangle~$R_j$, or inside of some other shape contained in the facet. We only choose~$N^{-\frac{1}{16}}$ for convenience in our proofs.
\end{remark}

Now we may use Theorem~\ref{thm:full_moments} and Corollary~\ref{cor:facet_joint_moments_same} to compute the joint moments of~$(Z_1,\dots,Z_g)$ asymptotically. In fact, we can compute the joint moments of this tuple of discrete components with the height function. Recall the definition of the~$(1,0)$ form~$\omega_0$ appearing in Theorem~\ref{thm:full_moments}, given in~\eqref{eq:omega_0_initial}.
\begin{prop}\label{prop:DCcov}
    Suppose that~$\mathsf{f}_j$,~$j=1,\dots, r$, are in a compact subset of the liquid region and their normalized coordinates are uniformly bounded away from each other. Let~$n_1,\dots, n_g$ be integers and let~$m =  \sum_{i=1}^g n_i$. As~$N \rightarrow \infty$, we have the asymptotic equivalence up to~$o(1)$ error
    \begin{multline}\label{eqn:most_general_joint_moment}
\mathbb{E}[ \prod_{j=1}^r \overline{h}_N(\mathsf{f}_j) \prod_{i=1}^g Z_i^{n_i}] 
\\
\approx \frac{1}{(2\pi \i)^{m+r}} \int_{\bar q_{1}}^{q_{1}} \int_{\bar q_{2}}^{q_{2}} \cdots
\int_{q_{r}}^{q_{r}} \int_{B_1}\cdots \int_{B_1}
\cdots \cdots \int_{B_g} \cdots \int_{B_g} 
\det\bigg( (1-\delta_{l m})\omega_0(q_m', q_{l}') \bigg)_{m, l=1}^{m+r}.
    \end{multline}
Above, the~$q_i$,~$i=1,\dots,r$, are as in Theorem~\ref{thm:full_moments}, and moreover there are~$n_1$ integrations over the cycle~$B_1$,~$n_2$ integrations over~$B_2$, and so on.
    
    
\end{prop}



\begin{proof}
    This follows from expanding the discrete components out as summations, and then expanding out the expectation. One may invoke Corollary~\ref{cor:facet_joint_moments_same} to see that the terms coming from discrete components~$Z_j$ where~$n_j \geq 2$ in which we pick the same face more than once will be negligible in the limit. This is because there are~$O(M^{m-1})$ such terms which are all bounded, and there is an overall prefactor of~$\frac{1}{M^m}$. Thus, invoking Theorem~\ref{thm:full_moments} for the terms when all faces are distinct (notice they are also~$> N^{-\frac{1}{16}} $ apart when inside of the same gaseous facet by assumption) gives the result. 
\end{proof}

\subsection{Rewriting the joint height moments}\label{subsec:rewriting}
Recall the definition of the~$(1,0)$ form~$\omega_0$ in~\eqref{eq:omega_0_initial}:
\begin{equation}\label{eq:omega_0_again}
\omega_0(q,q')=\frac{\psi_{k N,-}(q)\psi_{k N,+}(q')}{\psi_{0,-}(q')\psi_{0,+}(q')}\frac{d z  }{w^{k N}(z'-z)}.
\end{equation}
The goal of this section is to prove the following lemma, which provides a remarkably simple and useful expression for~$\omega_0$ in terms of theta functions and prime forms.

In this section we follow the notation and utilize results of~\cite[Section 5]{BB23}. In particular, we need~$e_{\mathrm w_{i,j}}^{(k N)},e_{\mathrm b_{i,j}}^{(k N)}\in \mathbb{R}^g/\mathbb{Z}^g$ from Proposition 5.4 there; see Proposition~\ref{prop:psipmexact} in Section~\ref{subsec:theta_prime} for a restatement of this proposition, and see the surrounding discussion about~$e_{\mathrm w_{i,j}}^{(k N)},e_{\mathrm b_{i,j}}^{(k N)}$, which includes a precise definition of the quantity~$e_{\mathrm w_{0,0}}^{(k N)}$ appearing below. Moreover, see~\eqref{eqn:edef} and the discussion leading up to it for the relationship of~$e_{\mathrm w_{0,0}}^{(k N)}$ with the limit shape. Finally, recall that~$u(p_{\infty,1}) \in J(\mathcal{R})$ denotes the image of the angle~$p_{\infty,1} \in \mathcal{R}$ under the Abel map (see Sections~\ref{subsec:spectral} and~\ref{subsec:theta_prime}).

\begin{lem}\label{lem:integrand_equivalence}
Define~$\omega_0(q,q')$ as in~\eqref{eq:omega_0_again}. Then, for some meromorphic~$\frac{1}{2}$-form~$g $ on $ \widetilde{\mathcal{R}} $, we have
\begin{equation}\label{eqn:omega0_statement}
 \omega_0(q,q') =  \frac{g(q)}{g(q')}\frac{\theta\left(\int_{q'}^q\vec{\omega}-u(p_{\infty,1})-e_{\mathrm w_{0,0}}^{(kN)}\right)}{\theta\left(u(p_{\infty,1})+e_{\mathrm w_{0,0}}^{(kN)}\right)E(q,q')}
\end{equation}
and this expression is well defined on~$\mathcal{R}$.

\end{lem}

\begin{remark}
    From the proof, the form~$g$ is given explicitly in terms of theta functions and prime forms. However, the explicit expression will not be relevant going forward; in fact, prefactors of~$g$ cancel out in computations of joint moments (c.f. Theorem~\ref{thm:full_moments}). Moreover, up to the prefactor involving~$g$, the expression on the right hand side of~\eqref{eqn:omega0_statement} appears in~\cite{Fay73}, serving as a higher genus generalization of~$\frac{1}{z-z'}$ in what may be viewed as higher genus analogs of the \emph{Cauchy determinant formula}.
\end{remark}

\begin{proof}


 Throughout this proof we follow the notation of~\cite{BB23}, so for~$e \in \mathbb{C}^g$ and~$q\in \mathcal{R}$ with lift~$\tilde{q} \in \widetilde{ \mathcal{R}}$, we denote~$\Theta(q;e) \coloneqq \theta(\int_{\tilde q_0}^{\tilde{q}} \vec{\omega} + e)$. The  apparent dependence on the choice of lift will eventually disappear, so we leave it out of the notation.

For fixed~$q' = (z', w')$, $q\mapsto \omega_0(q,q')$ is a meromorphic 1-form. Also, $q'\mapsto \omega_0(q,q')$ is a meromorphic function on $\mathcal R$. Using Proposition 5.4 in~\cite{BB23}, we get that for some~$e_{\mathrm w} \in \mathbb{R}^g$,
\begin{multline}\label{eq:product_nth_eigenvectors}
    \frac{1}{w^{k N}}\psi_{k N,-}(q)\psi_{k N,+}(q') dz dz'=\frac{\prod_{m=1}^{k\ell N}E(p_{0,m},q)}{\prod_{m=1}^{k\ell N}E(q_{0,m},q)}\frac{1}{\prod_{m=1}^{k-1}E(p_{0,m},q)}\frac{\Theta(q;e_{\mathrm w})}{E(p_{\infty,1},q)E(q_{0,1},q)} \\
    \times \frac{\prod_{m=1}^{k\ell N}E(q_{0,m},q')}{\prod_{m=1}^{k\ell N}E(p_{0,m},q')}\frac{1}{\prod_{m=1}^kE(p_{\infty,m},q')}\frac{\Theta(q';e_{\mathrm b})}{E(p_{0,k},q')} \\
    \times \sum_{j=0}^{k-1}c_{j,-}^{(kN)}c_{j,+}^{(kN)}\Theta\left(q;e_{\mathrm b_{0,j}}^{(kN)}\right)\Theta\left(q';e_{\mathrm w_{0,j}}^{(kN)}\right)\\
    \times \prod_{m=1}^jE(p_{\infty,m},q)\prod_{m=j+1}^{k-1}E(p_{0,m},q)\prod_{m=0}^{j-1}E(p_{0,m},q')\prod_{m=j+2}^kE(p_{\infty,m},q').
\end{multline}

The first line of the right hand side consists of~$k \ell N + k +1$ poles and~$k \ell N + g$ zeros of $q\mapsto \omega_0(q,q')$. In addition, the 1-form has a simple zero at $p_{\infty,m}$ for $m=1,\dots,k$ and a simple pole at $q=q'$, all coming from the factor $\frac{d z}{(z-z')}$; recall that the 1-form has no pole at points where $z=z'$ while $q\neq q'$. The poles and zeros described in the previous two sentences have a net contribution of~$g-2$ to the divisor (the number of zeros minus the number of poles) of~$q\mapsto \omega_0(q,q')$. The degree of the divisor of the one form~$q\mapsto \omega_0(q,q')$ is $2 g-2$ (the degree of the canonical divisor), hence, there are $g$ more zeros coming from the sum in the right hand side of \eqref{eq:product_nth_eigenvectors} which contribute to the divisor of~\eqref{eq:omega_0_again}. Our plan is to account for these~$g$ zeros using a theta function.

Let
\begin{equation}\label{eq:prefactor}
    g(q)=\frac{\prod_{m=1}^{k\ell N}E(p_{0,m},q)}{\prod_{m=1}^{k\ell N}E(q_{0,m},q)}\frac{\prod_{m=2}^kE(p_{\infty,m},q)}{\prod_{m=1}^{k-1}E(p_{0,m},q)}\frac{\Theta(q;e_{\mathrm w})}{E(q_{0,1},q)},
\end{equation}
then we will define a vector~$e_0 \in \mathbb{C}^g$ such that
\begin{equation}\label{eq:omega_0_2}
    \omega_0(q,q')=g(q)\frac{\Theta(q;e_0)}{E(q,q')} c(q')
\end{equation}
for some~$-\frac{1}{2}$ form $c(q')$ not depending on $q$. 

Let $D(q')$ be the divisor of $q\mapsto \omega_0(q,q')$, and $D_g$ be the divisor of $g$. By Abel's theorem~$u(D(q')) = 2 \Delta $, where~$\Delta $ is the vector of Riemann constants, so~$e_0$ must satisfy
\begin{equation}
    2\Delta=u(D(q'))=-e_0+\Delta-u(q')+u(D_g),
\end{equation}
and
\begin{multline}
    u(D_g)=-e_{\mathrm w}+\Delta-u(q_{0,1})+\sum_{m=2}^ku(p_{\infty,m},q)-\sum_{m=1}^{k-1}u(p_{0,m}) \\
    +N\left(\ell \sum_{m=1}^ku(p_{0,m})-k\sum_{m=1}^\ell u(q_{0,m})\right)
\end{multline}
where $u$ is the Abel map. Using~ $-e_{\mathrm w}-u(q_{0,1})=-e_{\mathrm w_{0,0}}^{(0)}-u(p_{0,k})$ (this is~\eqref{eqn:ew000}), and also using the formula~\eqref{eqn:ew00} for~$e_{\mathrm w_{0,0}}^{(kN)}$, we get that
\begin{equation}
    e_0=-u(q')-u(p_{\infty,1})-u((z))-e_{\mathrm w_{0,0}}^{(kN)}=-u(q')-u(p_{\infty,1})-e_{\mathrm w_{0,0}}^{(kN)},
\end{equation}
where $u((z))$ is the image of the Abel map of the divisor of the meromorphic function $q=(z,w)\mapsto z$, so $u((z))=0$ in~$J(\mathcal{R})$. So the theta function in~\eqref{eq:omega_0_2} can be written
\begin{equation}
    \Theta(q;e_0)=\theta\left(\int_{q'}^q\vec{\omega}-u(p_{\infty,1})-e_{\mathrm w_{0,0}}^{(kN)}\right).
\end{equation}

Now, from the definition of $\omega_0$, \eqref{eq:omega_0_again}, and the fact that~$\psi_{k N,-}(z,w) \psi_{k N, +}(z,w) = w^{k N}\psi_{0,-}(z,w) \psi_{0, +}(z,w) $ (see~\eqref{eqn:psipluspsiminus} and discussion around it), we get that
\begin{equation}
    \lim_{q'\to q}\omega_0(q,q')\frac{(z'-z)}{d z}=1.
\end{equation}
It follows from the behavior of the prime form at the diagonal \eqref{eq:prime_form_diagonal}, and from \eqref{eq:omega_0_2} that
\begin{equation}
    c(q')=\frac{1}{\theta\left(u(p_{\infty,1})+e_{\mathrm w_{0,0}}^{(kN)}\right)}\frac{1}{g(q')}.
\end{equation}
Hence,
\begin{equation}\label{eq:omega_0_final}
    \omega_0(q,q')=\frac{g(q)}{g(q')}\frac{\theta\left(\int_{q'}^q\vec{\omega}-u(p_{\infty,1})-e_{\mathrm w_{0,0}}^{(kN)}\right)}{\theta\left(u(p_{\infty,1})+e_{\mathrm w_{0,0}}^{(kN)}\right)E(q,q')},
\end{equation}
where $g$ is given in \eqref{eq:prefactor}. Finally, one may directly check that~\eqref{eq:omega_0_final} is well-defined on the surface by translating~$q$ around cycles and using quasi-periodicity properties of theta functions and prime forms; we omit the computation. This proves~\eqref{eqn:omega0_statement}.

\end{proof}

An immediate corollary is the following.

\begin{cor}\label{cor:twoptfunc}
    The leading order behavior of the height covariance for large~$N$ (this is the~$r=2$ case of Theorem~\ref{thm:full_moments}, and the notation is reused from there), is given by
\begin{equation}
\label{eqn:integral_final_lem}
\frac{1}{(2\pi \i)^2} \int_{\bar q_1}^{q_1}\int_{\bar q_2}^{q_2}\frac{\theta\left(\int_{q'}^q\vec{\omega}-u(p_{\infty,1})-e_{\mathrm w_{0,0}}^{(kN)}\right)\theta\left(\int_q^{q'}\vec{\omega}-u(p_{\infty,1})-e_{\mathrm w_{0,0}}^{(kN)}\right)}{\theta\left(u(p_{\infty,1})+e_{\mathrm w_{0,0}}^{(kN)}\right)^2 E(q,q')^2}.
\end{equation}
\end{cor}
\begin{proof}
It suffices to note that the expression for $\omega_0$ provides an expression for the integrand of the two point moment
\begin{equation}\label{eqn:omegathetafunc}
    -\omega_0(q,q')\omega_0(q',q)=\frac{\theta\left(\int_{q'}^q\vec{\omega}-u(p_{\infty,1})-e_{\mathrm w_{0,0}}^{(kN)}\right)\theta\left(\int_q^{q'}\vec{\omega}-u(p_{\infty,1})-e_{\mathrm w_{0,0}}^{(kN)}\right)}{\theta\left(u(p_{\infty,1})+e_{\mathrm w_{0,0}}^{(kN)}\right)^2 E(q,q')^2}.
\end{equation}
\end{proof}

\begin{remark}\label{rmk:q0_dep}
    Observe that~$\omega_0$ in~\eqref{eqn:omega0_statement} appears apriori to depend on the base point~$q_0$ via the Abel map, whereas integrals of expressions involving~$\omega_0$, such as~\eqref{eqn:integral_final_lem}, should not, since they are equal to observables of the dimer model. The resolution to this is the observation that~$e_{\mathrm w_{0,0}}^{(k N)}$ also in fact depends on~$q_0$, and its dependence exactly cancels out with that of~$u(p_{\infty,1})$, so in fact~$\omega_0$ does not depend on~$q_0$.
\end{remark}

\subsection{Convergence to Gaussian free field and independence}

For each~$j=1,\dots,g$, define the function~$f_j(q)$ as unique function which is harmonic in the interior of~$\mathcal{R}_0$, and satisfies boundary conditions
\begin{align*}
\begin{cases}
    f_j(q)  = 1 , & q \in A_j, \\
    f_j(q) = 0 , & q \in \partial \mathcal{R}_0 \setminus A_j . 
    \end{cases}
\end{align*}
Here and in the rest of this section, if we have a face~$\mathsf{f} = (2(\ell x + i), 2(k y + j)+2)$ in the Aztec diamond, then we will abuse notation and write~$q(\mathsf{f}) = q(\xi, \eta)$, if~$(\xi, \eta) = (\frac{2}{k N}x-1, \frac{2}{\ell N}y-1) \in \mathcal{F}_R$.

Having defined the discrete component and computed its joint moments with the mean-subtracted height function~$\bar h_N$, our goal in this section is to show two things: (1) The new random function
\begin{equation}\label{eqn:pregff}
\tilde{h}_N(\mathsf{f}) = \bar h_N(\mathsf{f}) - \sum_{j=1}^g f_j(q(\mathsf{f}))
\end{equation}
converges to a Gaussian free field, and (2)~$(\bar h_N(\mathsf{f}) - \sum_{j=1}^g f_j(q(\mathsf{f}))Z_j)_{\mathsf{f}}$ and~$(Z_1,\dots, Z_g)$ are asymptotically independent. 

Throughout the rest of this section we will also use the notation
\begin{equation}\label{eqn:omegadef}
\omega_2(q,q') = -\omega_0(q,q')\omega_0(q',q)
\end{equation}
which can be expressed in terms of prime forms and theta functions as in~\eqref{eqn:omegathetafunc}.

\begin{prop}\label{prop:greens_conv}
Let~$q_1 \neq q_2$ be two points in the interior of~$\mathcal{R}_0$ corresponding asymptotically to~$\mathsf{f}_1 = \mathsf{f}_{1,N}$ and~$\mathsf{f}_{2} = \mathsf{f}_{2,N}$, respectively. Then, with~$\tilde{h}_N$ as in~\eqref{eqn:pregff}
    \begin{equation}\label{eqn:covsub}
       \mathbb{E}[\tilde{h}_N(\mathsf{f}_1) \tilde{h}_N(\mathsf{f}_2) ]  \rightarrow \frac{1}{\pi} \mathcal{G}_{\mathcal{R}_0}(q_1, q_2)
    \end{equation}
where~$\mathcal{G}_{\mathcal{R}_0}$ is the Green's function of the Laplacian on~$\mathcal{R}_0$ with Dirichlet boundary conditions.
\end{prop}

\begin{proof}

Our starting point is  Proposition~\ref{prop:DCcov}, which says
$$
\mathbb{E}\left[
(\overline{h}_N(\mathsf{f}_1) - \sum_j f_j(q(\mathsf{f}_1)) Z_j)(\overline{h}_N(\mathsf{f}_2) - \sum_j f_j(q(\mathsf{f}_2)) Z_j)  \right] 
 \approx C(q_1,q_2)
 $$
 where~$\approx$ means that the difference of the two functions has vanishing error as~$N \rightarrow \infty$, and the quantity~$C(q_1,q_2)$ is defined as follows:
\begin{multline}\label{eqn:asymptotic_two_pt}
C(q_1,q_2) \\
\coloneqq \frac{1}{(2\pi \i)^2} \bigg( \int_{\bar q_1}^{q_1} \int_{\bar q_2}^{q_2} \omega_2(q,q') -\sum_j f_j(q_1) \int_{B_j} \int_{\bar q_2}^{q_2} \omega_2(q,q') - \sum_j f_j(q_2) \int_{\bar q_1}^{q_1} \int_{B_j}  \omega_2(q,q') 
\\
+\sum_{i, j}f_i(q_1) f_j(q_2) \int_{B_i} \int_{B_j} \omega_2(q,q')  \bigg).
\end{multline}
The right hand side apriori depends on~$N$, since~$\omega$ does; however, we will now show that~$C(q_1,q_2)$ is independent of~$N$. Note that~$C(q_1,q_2) = C(q_2, q_1)$. Furthermore, because the two form~$\omega_2$ is meromorphic and behaves as
\begin{equation}\label{eqn:local_expansion}
\omega_2(q_1,q_2) = \left(\frac{1}{(z_1-z_2)^2} + O(1)\right) dz_1 dz_2
\end{equation}
as~$q_1 \rightarrow q_2$, see~\eqref{eqn:omegathetafunc} and~\eqref{eq:prime_form_diagonal},~$C(q_1,q_2)$ is harmonic in~$q_1$ (and also in~$q_2$ by symmetry) away from~$q_1 = q_2$. 

We claim that
\begin{enumerate}[(A)]
    \item In local coordinates $C(q_1,q_2)$ behaves as~$-\frac{1}{2\pi^2} \log |z_1-z_2| + O(1)$,~$z_1 \rightarrow z_2$ \label{eqn:CLAIMAg1}
    \item and~$C(q_1,q_2) \rightarrow 0$ as~$q_1 \rightarrow \partial \mathcal{R}_0$. \label{eqn:CLAIMBg1}
\end{enumerate}

 For the first claim we perform a straightforward local computation for~$q_1 \approx q_2$; we give details for completeness. First, we note that the last three terms in~\eqref{eqn:asymptotic_two_pt} are smooth functions, so it suffices to analyze the first term in the right hand side of that display. Then, to set up the local computation, fix~$q_2$ and let~$q_{1,\delta}$ be fixed at the border of a small enough but fixed~$\delta$-neighborhood of~$q_{2}$, and let~$q_{2,\delta}$ be a different fixed point fixed at the border of the same neighborhood. Then, for~$q_1$ inside of this~$\delta$-neighborhood, we may express the first double integral in~\eqref{eqn:asymptotic_two_pt} as
\begin{align*}
\frac{1}{(2\pi \i)^2 }  \int_{\bar q_2}^{q_2} \left(\int_{\bar q_{1,\delta}}^{q_{1,\delta}} +
\int_{\bar q_1}^{\bar q_{1,\delta}}+
\int_{ q_{1,\delta}}^{q_1}\right)
\omega_2(q_2', q_1') &= \tilde{C}_1 + \frac{1}{(2\pi \i)^2 }  \int_{\bar q_2}^{q_2} \left(
\int_{\bar q_1}^{\bar q_{1,\delta}}+
\int_{ q_{1,\delta}}^{q_1}\right)
\omega_2(q_2', q_1') \\
&= \tilde{C}_2 + \frac{1}{(2\pi \i)^2 }  \left(\int_{\bar q_2}^{\bar q_{2,\delta}} + \int_{q_{2,\delta}}^{q_2}\right) \left(
\int_{\bar q_1}^{\bar q_{1,\delta}}+
\int_{ q_{1,\delta}}^{q_1}\right)
\omega_2(q_2', q_1')
\end{align*}
where~$\tilde{C}_1$ and~$\tilde{C}_2$ are constants not depending on~$q_1$. Then we compute, for the second term in the last line above, using the expansion~\eqref{eqn:local_expansion} in local coordinates 
\begin{align*}
   \frac{1}{(2\pi \i)^2 }  \left(\int_{\bar q_2}^{\bar q_{2,\delta}} + \int_{q_{2,\delta}}^{q_2}\right) \left(
\int_{\bar q_1}^{\bar q_{1,\delta}}+
\int_{ q_{1,\delta}}^{q_1}\right)
\omega_2(q_2', q_1') &=O(1) + \frac{1}{(2\pi \i)^2 }  \left(\int_{\bar q_2}^{\bar q_{2,\delta}} \int_{\bar q_1}^{\bar q_{1,\delta}} + \int_{q_{2,\delta} }^{q_2}\int_{ q_{1,\delta}}^{q_1}\right) \frac{dz_1' dz_2'}{(z_1' - z_2')^2}  \\
 &= O(1) + 
\frac{1}{(2\pi \i)^2 } 
 \left(\int_{\bar q_2}^{\bar q_{2,\delta}}  \frac{d z_2'}{\bar z_1 - z_2'}- \int_{q_{2,\delta} }^{q_2} \frac{d z_2'}{ z_1 - z_2'}\right) \\
 &= O(1) + 
\frac{1}{(2\pi \i)^2 } 
 \left(\log(\bar z_1 - \bar z_2) +\log( z_1 -  z_2)\right)\\
&= O(1)  
-\frac{1}{2 \pi^2  } 
\log|z_1 -  z_2| .
\end{align*}
Notice the extra factor of~$\frac{1}{\pi}$ compared to behavior of the usual Dirichlet Green's function. This matches~\eqref{eqn:CLAIMAg1}.

We now show the second claim. If~$q_1 \rightarrow A_0$ then it is clear that each term in~\eqref{eqn:asymptotic_two_pt} vanishes. If~$q_1 \rightarrow A_p$ for some~$p = 1,\dots, g$, then the first and second terms cancel, and the third and fourth terms cancel. Thus, we get~$0$. 

Thus, we have shown the two claims. It follows from the two claims that
\begin{equation}\label{eqn:greens_equiv}
C(q_1, q_2) = \frac{1}{\pi} \mathcal{G}_{\mathcal{R}_0}(q_1, q_2).
\end{equation}
(See the discussion of the Green's function in Section~\ref{subsec:Gff}). This completes the proof.
\end{proof}

To complete the proof of convergence to a Gaussian free field, we must prove that the moments satisfy Wick's formula. Instead of doing so directly, we will show the equivalent statement that joint cumulants of size more than~$2$ vanish. Towards this end, we first state a straightforward lemma which is most probably already known in some form. Though a more general statement can be made, we state it in a form which will be useful to us. We have collected several elementary properties of classical cumulants in Appendix~\ref{app:B}.

\begin{lem}\label{lem:momcum}
    Let~$\gamma_{j}$, for~$j=1,\dots,r$ be smooth nonintersecting contours, and let~$Q : \mathcal{R} \times \mathcal{R} \rightarrow \mathbb{C}$ be a kernel behaving as a one form in the first variable and as a function in the second. Suppose that~$Q$ is holomorphic in a neighborhood of~$\gamma_i \times \gamma_j$ for any~$i \neq j$, and vanishes on the diagonal (it does not have to be continuous there), and that~$(X_1,\dots, X_r)$ is tuple of mean zero random variables whose joint moments have the form, for tuples of distinct indices~$(i_1,\dots, i_p)$,~$1 \leq p < r$,
\begin{equation}\label{eqn:mom_integral}
    \mathbb{E}[X_{i_1} \cdots X_{i_p}] = \int_{\gamma_{i_1}} \cdots \int_{\gamma_p} \det Q(z_{i}, z_j)_{i,j=1}^p .
\end{equation}

Then the joint cumulants~$\kappa[X_{i_1},\dots, X_{i_p}]$, for tuples of distinct indices, are given by
\begin{equation}\label{eqn:cum_integral}
\kappa[X_{i_1},\dots, X_{i_p}] =(-1)^{p+1} \int_{\gamma_{i_1}} \cdots \int_{\gamma_{i_p}} \sum_{\substack{\text{$p$ cycles}\\
\sigma}} \prod_{s = 1}^p Q(z_s, z_{\sigma(s)}) .
\end{equation}
\end{lem}

\begin{proof}
This follows from a combinatorial fact about determinants. First, recall the relationship between cumulants and moments
\begin{equation}\label{eqn:cum_mom}
\mathbb{E}[X_{i_1}\cdots X_{i_p}] = \sum_{\pi}\prod_{\substack{\text{blocks } B \\
\text{in~$\pi$}}} \kappa[X_{i_j} : j \in B].
\end{equation} 
Above, the sum is over partitions~$\pi$ of the set~$\{1,\dots,p\}$.

Now by expanding the determinant, we have
\begin{align*}
\det Q(z_i, z_j)_{i,j=1}^p &= \sum_{\sigma}  \prod_{\substack{\text{cycles } \sigma^{i} \\
\text{in~$\sigma$}}} (-1)^{p_i+1}\prod_{j \in \sigma^i}
Q(z_j,z_{\sigma^i(j)}) \\
&= \sum_{\pi} \prod_{\substack{\text{blocks } B \\
\text{in~$\pi$}}}
(-1)^{|B|+1} \sum_{\substack{ \text{cycles} \\
\sigma : B \rightarrow B}}\prod_{j \in B} Q(z_j,z_{\sigma(j)}).
\end{align*}
In the first line above,~$\sigma^i$ is a cycle in~$\sigma$ and~$p_i$ denotes its length. On the second line above, the outer sum over~$\pi$ on the last line is over partitions of~$\{1,\dots, p\}$, and the inner sum is over the permutations of~$B$ which are a single cycle.

Applying~$ \int_{\gamma_{i_1}} \cdots \int_{\gamma_{i_p}} $ to both sides of our previous display and comparing with~\eqref{eqn:cum_mom} yields the result.
\end{proof}

Next, we state one more lemma which will be useful in the proof of the Wick formula.

\begin{lem}\label{lem:2Rform}
For~$r>2$, the~$(1,1,\dots,1)$-form (a one form in each of $r$ variables) on~$\mathcal{R}^{r}$ given by
\begin{equation}\label{eqn:rform}
    \sum_{\substack{r \text{cycles}\\
    \sigma}} \prod_{s=1}^r\omega_0(q_s, q_{\sigma(s)}) 
\end{equation}
is holomorphic. 
\end{lem}
\begin{proof}
     We only need to check using local coordinates that there are no poles. Fix all variables (to pairwise distinct values) except~$q_1$; note that the only possible poles are at~$q_2,\dots, q_{2 R}$. Let~$q_1 \rightarrow q_2$. Then we have an expansion of the form (in coordinates) 
    \begin{equation}
        \frac{\eqref{eqn:rform}}{\prod dz_k} = 
\frac{c_{-1}}{z_1-z_2} + O(1).
    \end{equation}
    Indeed, the expression~\eqref{eqn:rform} has at most a simple pole as~$q_1 \rightarrow q_2$ because an~$r$ cycle cannot give a higher order pole. Note that~$c_{-1}$ may depend on~$q_2$ apriori. However, we must have~$c_{-1} \equiv 0$ because the expression is invariant under the swap~$q_1 \leftrightarrow q_2$ (this is the same as conjugating all cycle permutations by a transposition, which leaves the sum~\eqref{eqn:rform} unchanged), which would not hold to leading order if there was an order~$1$ pole. Therefore all that remains is the~$O(1)$ term which is holomorphic.
\end{proof}

\begin{remark}
    The proof of the lemma above was inspired by~\cite[Lemma 3.1]{Ken01}; however, one should note the essential difference: Holomorphic does not mean identically zero in our case, since we are in genus~$g \geq 1$! This captures the essential difference between the (genus~$0$) dimer models studied in that work and the present work. See also Remark~\ref{rem:holomorphic_form} below.
\end{remark}

Now we continue on analyzing joint moments of~$\tilde{h}_N(\mathsf{f}) = \bar h_N(\mathsf{f}) - \sum_{j=1}^g f_j(q(\mathsf{f})) Z_j$. In Proposition~\ref{prop:greens_conv}, we analyzed the two point function. To obtain the result for~$ r > 2$, we prove the Wick formula by showing that the joint cumulants vanish for~$r > 2$.

\begin{theorem}
    \label{thm:multipt_conv}
Let~$q_1, q_2,\dots, q_r$,~$r > 2$, be distinct points in the interior of~$\mathcal{R}_0$ corresponding to faces~$\mathsf{f}_1,\mathsf{f}_2,\dots, \mathsf{f}_r$ in the Aztec diamond such that~$q(\mathsf f_m) \rightarrow q_m$. Then, the joint height cumulant vanishes asymptotically as~$N \rightarrow \infty$,
\begin{equation}\label{eqn:convWTS}
\kappa[(\tilde{h}_N(\mathsf{f}_m))_{m=1}^r] \rightarrow 0.
\end{equation} 

In particular, the asymptotic moments are given by
    \begin{equation}\label{eqn:cov_multi}
       \mathbb{E}\left[
\prod_{m=1}^r \tilde{h}_N(\mathsf{f}_m) \right]  \rightarrow
\begin{cases}
\frac{1}{\pi^{r/2}} \sum_{\text{pairings } \pi} \prod_{i=1}^{r/2} \mathcal{G}_{\mathcal{R}_0}(q_{\pi(2 i -1)}, q_{\pi(2 i)}) & r \text{ even}\\
0 & \text{ otherwise}
\end{cases}
    \end{equation}
where~$\mathcal{G}_{\mathcal{R}_0}$ is the Green's function of the Laplacian on~$\mathcal{R}_0$ with Dirichlet boundary conditions.
\end{theorem}

We begin with the genus 1 case simply because the notation is lighter, and we hope it clarifies the structure of the argument. Before beginning, we note that (for general genus) the second statement~\eqref{eqn:cov_multi} follows from the first by the correspondence between joint moments and joint cumulants; see Lemma~\ref{lem:wickformlem} for details. So it suffices to prove the first statement in the theorem.

\begin{proof}[Proof of Theorem~\ref{thm:multipt_conv}, genus 1 case]

 For the genus 1 case denote~$f = f_1$ and~$Z = Z_1$.

First, we compute the left hand side of~\eqref{eqn:convWTS}. By expanding the left hand side directly using multilinearity of the cumulants, and using Lemma~\ref{lem:momcum} to transform our integral formulas for moments into integral formulas for cumulants (note that each term in the expansion fits the assumptions of the lemma by Proposition~\ref{prop:DCcov}), we have
\begin{multline}\label{eqn:cumulant_expansion}
(-1)^{r+1} (2\pi \i)^r \kappa\left[ \overline{h}_N(\mathsf{f}_1) - f(q(\mathsf{f}_1) ) Z,\dots,\overline{h}_N(\mathsf{f}_r) - f(q(\mathsf{f}_r) ) Z \right]  \\
\approx \int_{\bar q_{1}}^{q_{1}} \int_{\bar q_{2}}^{q_{2}} \cdots
 \int_{\bar q_{r }}^{q_{r}}  \sum_{\substack{\text{$r$ cycles}\\
\sigma}}  \prod_{m=1}^r \omega_0(q_m', q_{\sigma(m)}')  \\
- \sum_{k=1}^{r} f(q_{k}) \int_{\bar q_{1}}^{q_{1}} \cdots \int_{B}^{k} \cdots
 \int_{\bar q_{r}}^{q_{r}}  \sum_{\substack{\text{$r$ cycles}\\
\sigma}} \prod_{m=1}^r \omega_0(q_m', q_{\sigma(m)}') \\
+ \sum_{1 \leq k < j \leq r} f(q_{k}) f(q_{j}) \int_{\bar q_{1}}^{q_{1}} \cdots \int_{B}^{k} \cdots \int_{B}^{j} \cdots \int_{\bar q_{r }}^{q_{r}}    \sum_{\substack{\text{$r$ cycles}\\
\sigma}} \prod_{m=1}^r \omega_0(q_m', q_{\sigma(m)}')  \\
\pm \cdots \\
+ 
\prod_{i=1}^{r} f(q_{i})  \int_{B} \cdots \int_{B}  \sum_{\substack{\text{$r$ cycles}\\
\sigma}} \prod_{m=1}^r \omega_0(q_m', q_{\sigma(m)}') 
\end{multline}
On the third and fourth lines above, symbols like~$\int_B^{k}$ indicate that the~$k^{th}$ integral, which was along a path from~$\bar q_{k}$ to~$q_{k}$ in the second line, is now replaced by an integral over the entire~$B$ cycle.

Recall $A_1$ is the compact oval. We now show that the display above vanishes if any variable goes to the boundary. Take~$q_j \rightarrow A_1$ for some~$j$; we claim that the display vanishes due to a   cancellation, similarly to a telescoping sum. Omitting the integrand itself for brevity, which is the same on each line, if~$q_j \rightarrow A_1$ the integrals in the~$p$th line can be written as
\begin{multline}\label{eqn:explanation}
    \sum_{j_1 < \dots < j_p} f(q_{j_1}) \cdots f(q_{j_p}) \int \cdots \int = (\sum_{\text{some } j_m = j} + \sum_{j_{m} \neq j \; \forall m}) f(q_{j_1}) \cdots f(q_{j_p}) \int \cdots \int \\
    = \sum_{j_1 < \cdots < j_{p-1}, \neq j} f(q_{j_1}) \cdots f(q_{j_{p-1}}) \int \cdots \int + \sum_{j_1 < \cdots < j_{p}, \neq j} f(q_{j_1}) \cdots f(q_{j_p}) \int \cdots \int
\end{multline}
where in the first line above all  integrations except those of indices~$j_1,\dots,j_p$ are along contours from~$\bar q_m$ to~$q_m$, and those of indices~$j_1,\dots, j_p$ are replaced by the~$B$ cycle. Note that~$q_j \rightarrow A_j$ means that the integral from~$\bar q_j$ to~$q_j$ is now over the~$B$ cycle. Thus, the first term in the second line of~\eqref{eqn:explanation} has~$p$ integrals over~$B$ cycles, and the second has~$p+1$ integrals over~$B$ cycles (both always including integral~$j$). With this decomposition of each line, we see a  cancellation of the terms in adjacent lines in the large display~\eqref{eqn:cumulant_expansion}, so that the entire expression vanishes.

 We showed above that in the limiting formula for the cumulant, if any variable~$q_j$ in~$\mathcal{R}_0$ (corresponding to a face in the Aztec diamond) approaches the boundary, the expression vanishes. Furthermore, by Lemma~\ref{lem:2Rform} the integrand is holomorphic. Therefore the expression is harmonic with zero boundary values, in say, the variable~$q_1$, and so it is identically~$0$, and we have proved the theorem.
\end{proof}

Now, we move on to the proof for general genus~$g$, which is very similar. 

\begin{proof}[Proof of Theorem~\ref{thm:multipt_conv}, genus~$g$ case] 
 First, we note that~\eqref{eqn:cumulant_expansion} in the genus~$1$ case can be rewritten in a simple way symbolically, if we use a notation where integration symbols can be ``multiplied'' and expanded by the distributive law. Indeed, denoting the integrand there, which is a 1 form in each variable, by $\Omega(q_1',\dots,q_r') \coloneqq \sum_{\sigma} \prod_{m=1}^r \omega_0(q_m', q_{\sigma(m)}')$ (recall the sum is over all~$r$ cycles), we have

\begin{equation}\label{eqn:cumulant_expansion_1}
\eqref{eqn:cumulant_expansion} = \prod_{m=1}^{r}\left( \int_{\bar q_{m}}^{q_{m}} - f(q_{m}) \int_{B} \right) \Omega.
\end{equation}

 Returning to the genus~$g$ case, we keep the notation~$\Omega$ for the integrand. One may check that expanding out the genus~$g$ joint cumulant leads to a straightforward generalization of~\eqref{eqn:cumulant_expansion_1}, i.e. 
\begin{multline}\label{eqn:cumulant_expansion_g}
    (-1)^{r+1} (2\pi \i)^r \kappa\left[ \overline{h}_N(\mathsf{f}_1) - \sum_{j=1}^g f_j(q(\mathsf{f}_1) ) Z_j,\dots,\overline{h}_N(\mathsf{f}_r) - \sum_{j=1}^g f_j(q(\mathsf{f}_r) ) Z_j \right] 
    \\
    \approx \prod_{m=1}^{r}\left( \int_{\bar q_{m}}^{q_{m}} - \sum_{j=1}^g f_j(q_{m}) \int_{B_j}^{m} \right) \Omega
\end{multline}
where above we again denote the integral of the variable~$q_{m}'$ over the cycle~$B_j$ by~$\int_{B_j}^{m}$.

 Now, proceeding as before, Lemma~\ref{lem:2Rform} implies that the integrand is holomorphic. Thus, it suffices to prove that if~$q_m \rightarrow A_{j_0}$ for any~$m= 1,\dots, 2 R$,~$j_0=0,\dots, g$, the expression vanishes. If~$j_0 = 0$ then this is obvious because all terms in the expansion of~\eqref{eqn:cumulant_expansion_g} vanish. If~$j_0 = 1,\dots, g$, then this can be seen by a similar cancellation as was observed in the genus~$1$ case. Alternatively, one may observe that as~$q_m \rightarrow A_{j_0}$, we have~$\int_{\bar q_m}^{q_m} - \sum_j f_j(q_m) \int_{B_j}^m  \rightarrow \int_{B_{j_0}} - f_{j_0}(q_m) \int_{B_{j_0}} = 0$ (recall~$f_{j_0} \equiv 1$ on~$A_{j_0}$ and for~$j_1 \neq j_0$,~$f_{j_1} \equiv 0$ on~$A_{j_0}$), where this last equality is as a linear functional on one forms in the variable~$q_m$. 

This completes the proof of Theorem~\ref{thm:multipt_conv} for general genus~$g$.
\end{proof}

Next, we prove that the new field~$\tilde{h}_N$, which approximates a Gaussian free field, is asymptotically independent from the discrete components~$(Z_1,\dots,Z_g)$.

\begin{prop}\label{prop:ind_moments}
     Consider faces~$\mathsf{f}_1,\dots, \mathsf{f}_r$ in a compact subset of the liquid region and with normalized coordinates bounded away from each other. Also, let~$n_1,\dots, n_g$ be nonnegative integers. As~$N \rightarrow \infty$, we have
\begin{equation}\label{eqn:0cumulants}
       \kappa[\overline{h}_N(\mathsf{f}_1) - \sum_j f_j(q(\mathsf{f}_1)) Z_j,\dots,\overline{h}_N(\mathsf{f}_r) - \sum_j f_j(q(\mathsf{f}_r)) Z_j,  Z_1,\dots, Z_1,Z_2,\dots, Z_2,\cdots, Z_g,\dots, Z_g]  \rightarrow 0 .
    \end{equation}
Above, there are~$n_1$ copies of~$Z_1$ in the joint cumulant,~$n_2$ copies of~$Z_2$, and so on.

As a consequence, for any~$r$ faces~$\mathsf{f}_1,\dots, \mathsf{f}_r$, whose normalized coordinates are converging to distinct points in~$\mathcal{F}_R$, the two tuples~$(\overline{h}_N(\mathsf{f}_m) - \sum_j f_j(q(\mathsf{f}_m)) Z_j)_{m=1}^r$ and~$(Z_1,\dots,Z_g)$ are asymptotically independent in the sense of moments.
\end{prop}

\begin{proof}

Define~$M = r + \sum_j n_j$. We use the notation~$\Omega = \Omega(q_1',\dots, q_M')$ to denote the holomorphic form from Lemma~\ref{lem:2Rform}, as in the other proofs in this section.

By Lemma~\ref{lem:momcum} and Proposition~\ref{prop:DCcov},~$(2\pi \i)^M (-1)^{M+1}$ times the joint cumulant~\eqref{eqn:0cumulants} is given by following expression, which we view as a function of~$(q_1,\dots, q_{r})$
\begin{equation}\label{eqn:jointcum1}
\mathfrak{c}(q_1,\dots,q_{r}) \coloneqq
\prod_{m=1}^{r} \left( \int_{\bar q_m}^{q_m}- \sum_{j=1}^g f_j(q_m)\int_{B_j}  \right) \int_B \cdots \int_{B}\Omega .
\end{equation}
Above, we use the same notation for integrals as in~\eqref{eqn:cumulant_expansion_1} and~\eqref{eqn:cumulant_expansion_g}, i.e. a product of sums of integration symbols should be expanded out in the obvious way. Also, the notation~$\int_B \cdots \int_B$ denotes~$n_1$ integrations (for variables~$q_{r+1}',\dots, q_{r+n_1}'$) over~$B_1$,~$n_2$ integrations over~$B_2$, and so on, corresponding to the~$n_1$ copies of~$Z_1$,~$n_2$ copies of~$Z_2$, and so on  in~\eqref{eqn:0cumulants}.

We now proceed by analyzing~$\mathfrak{c}(q_1,\dots,q_r)$ as a function of~$q_1 \in \mathcal{R}_0$. By Lemma~\ref{lem:2Rform} the integrand is holomorphic in~$q_1$, so~$\mathfrak{c}$ is harmonic in~$q_1$. Moreover, in a similar way to the proof of Theorem~\ref{thm:multipt_conv}, we see that~$\mathfrak{c}$ vanishes when~$q_1 \rightarrow \partial\mathcal{R}_0$. Therefore~$\mathfrak{c}$ is identically zero, which proves the proposition.

The independence in the sense of moments follows from properties of cumulants and the fact that the joint moments are determined algebraically by joint cumulants (see Appendix~\ref{app:B} and in particular Lemma~\ref{lem:indvars}).
\end{proof}

\begin{remark}\label{rem:holomorphic_form}
We emphasize that throughout this section (with the exception of Proposition~\ref{prop:greens_conv}) the only property of~$\Omega$ (the form in Lemma~\ref{lem:2Rform}) we used, was that it is holomorphic for~$r>2$. This means that the computations in this section have the potential to be directly applicable in future work. We also want to mention that the appearance of a random harmonic function is natural from the perspective of the holomorphicity of~$\Omega$. Indeed, for~$g\geq 1$, we may write~$\Omega(q_1',\dots,q_r')$ in a basis of holomorphic $1$-forms:
\begin{equation}
    \Omega(q_1',\dots,q_r')=\sum_{i_1,\dots,i_r=1}^g c_{i_1,\dots,i_r} \prod_{j=1}^r\omega_{i_j}'(q_j'),
\end{equation}
where~$\vec \omega'=(\omega_1',\dots\omega_g')$ is defined by~$\vec \omega'=B^{-1} \vec \omega$, that is, a basis of holomorphic 1-forms normalized by~$\int_{B_j}\omega_i'=\delta_{ij}$, and
\begin{equation}\label{eq:const_cumulants}
c_{i_1,\dots,i_r}=\int_{B_{i_1}}\dots\int_{B_{i_r}}\Omega(q_1',\dots,q_r').
\end{equation}
Note that the functions~$f_i$ introduced in the beginning of this section are given by~$f_i(q)=\int_{\bar q}^q\omega_i'$. In contrast to the genus-zero case, where holomorphicity implies that the higher cumulants vanish, the cumulants for~$g>0$ converge to the cumulants of a random harmonic function on~$\mathcal R_0$, with the distribution on compact ovals determined by~\eqref{eq:const_cumulants}. Identifying this distribution is the content of the next section.
\end{remark}

\subsection{The distribution of the discrete component}
\label{subsec:DCdist}
In this subsection we characterize the distribution of the discrete components~$(Z_1,\dots, Z_g)$. This, together with the results of previous sections, will give a complete, explicit characterization of the limiting distribution of the height function with doubly periodic weights.

We proceed by using Proposition~\ref{prop:DCcov} and Lemma~\ref{lem:momcum} to compute the joint cumulants
\begin{equation}\label{eqn:kappa_joint}
\kappa_{n_1,n_2,\dots,n_g} \coloneqq \kappa[Z_1,\dots, Z_1,Z_2,\dots, Z_2,\cdots,Z_g,\dots, Z_g].
\end{equation}
for any number~$n_j \geq 0$ of repetitions of each~$Z_j$. Here to lighten notation, we explicitly use an additional assumption that~$q_0 = p_{\infty, 1}$, where~$q_0$ is the basepoint used in the definition of the Abel map, so that~$u(p_{\infty, 1}) = 0$. See Remark~\ref{rmk:q0dep}. Let~$n = \sum_{j=1}^g n_j$.

\begin{theorem}
    \label{thm:discrete_comp_thm}
The random vector~$(Z_1,\dots,Z_g)$ asymptotically approximates a discrete Gaussian distribution, with an~$N$ dependent shift parameter~$e_{0,0}^{(k N)}$. More precisely, for large~$N$, up to~$o(1)$ error, joint cumulants of~$(Z_1,\dots,Z_g)$ are given, for~$n = 2$ by
\begin{equation}\label{eqn:DC_n2_cumulant}
(2\pi \i)^2 \kappa[Z_i, Z_j] = 2 \pi \i B_{i, j} +  \sum_{i',j'=1}^g (\partial_{z_{i'}} \partial_{z_{j'}} \log \theta )(e_{\mathrm w_{0,0}}^{(k N)}) B_{i,i'} B_{j,j'} + o(1)
\end{equation}
and for~$n \geq 3$ by
\begin{multline}\label{eqn:DC_joint_cumulant}
(2\pi \i)^n \kappa_{n_1,n_2,\dots,n_g}    \\
= \sum_{i_1,\dots,i_n=1}^g B_{1, i_1} \cdots B_{1, i_{n_1}} B_{2 ,i_{n_1+1}} \cdots B_{2 ,i_{n_1+n_2}} \cdots \cdots B_{g, i_n} (\partial_{z_{i_1}} \cdots \partial_{z_{i_n}} \log \theta) (e_{\mathrm w_{0,0}}^{(k N)}) + o(1).
\end{multline}
\end{theorem}

\begin{remark}\label{rmk:cgf}
 In other words, the joint cumulant generating function obtained by replacing~$\kappa_{n_1,\dots,n_g}$ with their leading order approximations, given by the right hand sides of~\eqref{eqn:DC_n2_cumulant} and~\eqref{eqn:DC_joint_cumulant} without the~$o(1)$ errors, is given by 
\begin{equation}\label{eqn:CGF}
    \sum_{n = (n_1,\dots,n_g)} \frac{(2\pi \i)^n \kappa_{n_1\dots n_g}}{n_1!\cdots n_g!} z_1^{n_1}\cdots z_g^{n_g} = \log \frac{\theta(B z + e_{\mathrm w_{0,0}}^{(k N)})}{\theta( e_{\mathrm w_{0,0}}^{(k N)})} + \frac{1}{2} (2\pi \i)z\cdot B z + c_1 \cdot z 
\end{equation}
where~$c_1 \in \mathbb{R}^g$ is chosen to make the linear in~$z$ term equal to~$0$.
\end{remark}

\begin{proof}
Throughout this proof we will use a notation similar to that of Fay. In the arguments of theta functions, if~$e \in \mathbb{C}^g$ and~$q_1,q_2 \in \mathcal{R}$, then
\begin{equation}\label{eqn:faynotation}
    \theta(q_1-q_2 -e) \coloneqq \theta(\int_{q_2}^{q_1} \vec{\omega} - e). 
\end{equation}

We start with~$n=2$. By Equation (39) in Fay's book~\cite{Fay73}, there is a formula for the integrand in our formula for the covariance of~$Z_i$ and~$Z_j$. This integrand (on the left hand side of the next display) is nothing but the integrand of~\eqref{eqn:omegathetafunc} cast in the notation~\eqref{eqn:faynotation}, and with~$ e_{0,0}^{(k N)}$ replaced by~$e$. The formula says
\begin{equation}\label{eqn:var_integrand}
   \frac{ \theta(q-q' - e) \theta(q'-q - e)}{\theta(e)^2 E(q',q)^2} = \omega(q,q') + \sum_{i,j=1}^g (\partial_{z_i} \partial_{z_j} \log \theta )(e)\omega_i(q')\omega_j(q).
\end{equation}
Above
$$\omega(q,q') = d_q d_{q'} \log \theta(\int_{q}^{q'} \vec{\omega} - f),$$ for~$f$ a nondegenerate odd characteristic. We will call the left hand side of~\eqref{eqn:var_integrand}~$\omega_e$ for now to emphasize the dependence on~$e$. 

Now from the decomposition~\eqref{eqn:var_integrand} and an explicit computation of the contribution from each term, we see that
\begin{multline}
\label{eqn:omega0term}
(2\pi \i)^2 \kappa[Z_i, Z_j] = (2\pi \i)^2 \mathbb{E}[Z_i Z_j] =  \int_{B_i} \int_{B_j} \omega_{e_{0,0}^{(kN)}}(q, q') + o(1) \\
= (2 \pi \i) B_{i j} + \sum_{i',j'=1}^g (\partial_{z_{i'}} \partial_{z_{j'}} \log \theta )(e_{0,0}^{(kN)}) B_{i,i'} B_{j,j'} + o(1).
\end{multline}

In other words, the (asymptotic) covariance matrix~$\Sigma = (\mathbb{E}[Z_i Z_j])_{i,j=1}^g$ times~$(2\pi \i)^2$ equals 
\begin{equation}\label{eqn:cov_mat}
(2\pi \i)^2\Sigma  =  (2\pi \i) B + B \text{Hess} (\log \theta)|_{e_{0,0}^{(k N)}} B.
\end{equation}
This agrees with the~$n=2$ case of the theorem.

Next, we start with Proposition~\ref{prop:DCcov} and Lemmas~\ref{lem:momcum} and~\ref{lem:2Rform} to obtain an expression for the higher cumulants. We call the integrand there~$\Omega(q_1,\dots, q_n)$, as before, except that we will actually include the sign from Lemma~\ref{lem:momcum} in the integrand now, so the cumulants are integrals over appropriate~$B$ cycles of~$\Omega$:
\begin{equation}\label{eqn:integral_cums}
    (2\pi \i)^n \kappa_{n_1,\dots,n_g} = \int_{B_1}\cdots \int_{B_1} \int_{B_2} \cdots \int_{B_2} \cdots \cdots \int_{B_g} \Omega
\end{equation}
where there are~$n_1$ distinct integrations over~$B_1$, and so on. 

We define, for~$q_1,\dots, q_n \in \mathcal{R}$, and~$e \in \mathbb{C}^g$, 
\begin{equation}\label{eqn:In_genusg_def}
I_n(q_1,\dots,q_n) \coloneqq \frac{\theta(q_1 - q_2 -e)\cdots \theta(q_n-q_1-e)}{\theta(e)^n E(q_1,q_2) \cdots E(q_n,q_1)}.
\end{equation}
Note that if we set~$e=e_{0,0}^{(k N)}$, then~$I_n$ equals the product~$\prod_i \omega_0(q_i, q_{i+1})$ (recall the notation~\eqref{eqn:faynotation} and compare with Lemma~\ref{lem:integrand_equivalence} and note that the factors of~$g$ cancel out). This corresponds to the cycle permutation~$(1 2 \cdots n)$ in~\eqref{eqn:rform}, and the averaged integrand~$\Omega$ is the sum over all~$(n-1)!$ distinct cycle permutations. Namely,
\begin{equation}\label{eqn:Omega_In}
    \Omega(q_1,\dots, q_n) = (-1)^{n+1} \sum_{\sigma \in S_{n-1}} \sigma \cdot I_n(q_1,\dots, q_n)
\end{equation}
where the sum is now over \emph{all}~$(n-1)!$ permutations of the~$n-1$ numbers~$\{2,\dots, n\}$, and acts by permuting the arguments of~$I_n$ (and it leaves~$q_1$ fixed).

In fact, we claim that~$\Omega$ as in~\eqref{eqn:Omega_In} is given, for~$n \geq 3$, by
\begin{equation}\label{eqn:log_der}
\sum_{i_1,\dots,i_n=1}^g (\partial_{z_{i_1}} \cdots \partial_{z_{i_n}} \log \theta) (e) \omega_{i_1} \cdots \omega_{i_n}
\end{equation}
This statement suffices to prove the theorem. We will now prove this statement by induction for~$n \geq 3$, using the~$n=2$ formula~\eqref{eqn:var_integrand} as the base case (though the~$n=2$ integrand is not of the form~\eqref{eqn:log_der}, the induction step still works to obtain~\eqref{eqn:log_der} for~$n=3$ from~$n=2$). We proceed by decomposing~$I_n$ as a one form in~$q_1$ into canonical pieces, and using this decomposition, and induction, to compute it.

Notice that, if~$n \geq 3$, 
\begin{equation}\label{eqn:In_rec}
I_n = -I_{n-1}(q_2,\dots,q_n) 
\times \left( \frac{\theta(q_1 - q_2 -e) \theta(q_n-q_1-e)}{\theta(e) \theta(q_n-q_2-e)} \frac{E(q_n,q_2)}{E(q_1,q_n) E(q_1,q_2)} \right).
\end{equation}

Remarkably, Fay's Proposition 2.10 has an identity for the second factor on the right hand side of~\eqref{eqn:In_rec}. This will allow us to complete the proof. We use (38) in that proposition, with $a = q_{n}, b = q_2, x = q_1$, to get 
\begin{equation}\label{eqn:goodeq}
    I_n = -I_{n-1} \left( \omega_{q_2-q_n}(q_1) + \sum_{i=1}^g \left( \frac{\partial \log \theta}{\partial z_i }(e + q_2 - q_n) - \frac{\partial \log \theta}{\partial z_i }(e) \right)\omega_i(q_1) \right).
\end{equation}
Above~$\omega_{b-a}$ denotes the unique Abelian differential form of the third kind with simple of residue~$+1$, resp.~$-1$ at~$b$, resp.~$a$, and with vanishing~$A$ periods.

Summing the first term in parentheses in~\eqref{eqn:goodeq} over only cyclic shifts of~$(q_2,\dots,q_n)$ (this is a subset of the sum in~\eqref{eqn:Omega_In} over all permutations of~$\{2,\dots, n\}$) gives us (note these leave~$I_{n-1}$ invariant)
\begin{equation}
\sum_{j=2}^n \omega_{q_{j}-q_{j-1}}(q_1) = 0
\end{equation}
(the index~$j$ moves in a cyclic way through~$\{2,\dots, n\}$). Therefore, using the cycle notation~$(2 \cdots n)$ for the cyclic shift~$j \mapsto j+1$ permuting~$\{2,\dots, n\}$, and denoting~$(2 \cdots n)^i$ its~$i$th power, the sum of~\eqref{eqn:goodeq} 
 over cyclic shifts of~$(q_2,\dots, q_n)$ gives us 
\begin{equation}\label{eqn:der0}
\sum_{(2\cdots n)^i} \sigma \cdot I_n =  -I_{n-1} \left(\sum_{j=2}^{n} \sum_{i=1}^g \frac{\partial \log \theta}{\partial z_i }(e + q_{j}- q_{j-1}) \omega_i(q_1) - (n-1)\sum_{i=1}^g\frac{\partial \log \theta}{\partial z_i }(e)\omega_i(q_1) \right) .
\end{equation}
Again in the first sum on the right hand side~$j-1$ is interpreted as~$n$ if~$j=2$. If we write~$I_{n-1}(q_2,\dots, q_n) = I_{n-1}(q_2,\dots, q_n; e)$ to make the~$e$ dependence explicit, then a computation shows that
\begin{equation}\label{eqn:der1}
  -\left(d_{q_{1}} I_{n-1}(q_2,\dots, q_n;  \int_{y}^{q_1} \vec{\omega} + e) \right) |_{y=q_1} =  \eqref{eqn:der0}  .
\end{equation}
To see this we have used the the fact that~$\theta$ is an even function.

We claim that~\eqref{eqn:der1} is enough to complete the induction. We will write~$\Omega = \Omega_n$ to make the~$n$-dependence explicit in the notation. Indeed, summing both sides of~\eqref{eqn:der0} over permutations~$\sigma'$ of~$\{3,\dots,n\}$ and using~\eqref{eqn:der1}, we get \begin{multline}\label{eqn:der2}
\Omega_n(q_1,\dots, q_n) = (-1)^{n+1} \sum_{\sigma' \in S_{n-2}} \sum_{(2\cdots n)^i} \sigma' \cdot \sigma \cdot I_n \\
= (-1)^n \sum_{\sigma' \in S_{n-2}} \sigma' \cdot \left(d_{q_{1}} I_{n-1}(q_2,\dots, q_{n}; \int_{y}^{q_1} \vec{\omega} + e) \right) |_{y=q_1} .
\end{multline}
For the first equality we used the fact that each coset of~$S_{n-2}$ (permutations of~$\{2,\dots, n\}$ fixing~$2$) in~$S_{n-1}$ (all permutations of~$\{2,\dots, n\}$) can be represented by an element of the form~$(2\cdots n)^i$. The display above means that to obtain~$\Omega_n(q_1,\dots, q_n)$, one should take~$\Omega_{n-1}(q_2,\dots, q_n)$ with~$e $ replaced by~$e + \int_{y}^{q_1}\vec{\omega}$, take the differential in~$q_1$, and set~$y = q_1$. (The sign is correct because the signs in Lemma~\ref{lem:momcum} alternate.) The sequence of expressions~\eqref{eqn:log_der} has the same property. Thus, we obtain~$\Omega_n = \eqref{eqn:log_der}$ as desired. Therefore, substituting back~$e=e_{\mathrm w_{0,0}}^{(k N)}$ and using~\eqref{eqn:integral_cums}, the proof is complete.
\end{proof}

We conclude this section by explicitly matching the cumulant generating function in Remark~\ref{rmk:cgf} to the cumulant generating function of a multivariate discrete Gaussian. This requires the~\emph{modular transformations} appearing as Equation (5.1) in~\cite{Mum07a}.




The theta function identity which is relevant for us is 
\begin{equation}\label{eqn:pois_sum}
    \theta(z; -B^{-1}) = \sqrt{\det(-\i B)} e^{\i \pi z \cdot B z} \theta(B z; B).
\end{equation}
Equation~\eqref{eqn:pois_sum} leads the following corollary of Theorem~\ref{thm:discrete_comp_thm}.
\begin{cor}\label{eqn:discrete_gauss_corg}
    The cumulants of the discrete component~$Z = (Z_1,\dots,Z_g)$ approximate those of a~\emph{discrete Gaussian} distribution. The cumulant generating function of the asymptotic expressions for cumulants, obtained by throwing away~$o(1)$ errors in Theorem~\ref{thm:discrete_comp_thm}, is given by
\begin{equation}\label{eqn:cgfcor}
   \sum_{n = (n_1,\dots,n_g)} \frac{(2\pi \i)^n \kappa_{n_1\dots n_g}}{n_1!\cdots n_g!} z_1^{n_1}\cdots z_g^{n_g} = \log \frac{\theta(z + B^{-1} e_{\mathrm w_{0,0}}^{(k N)};-B^{-1})}{\theta(B^{-1} e_{\mathrm w_{0,0}}^{(k N)};-B^{-1})} + c_1 \cdot z  .
\end{equation}
   Above,~$c_1$ is chosen to make the linear term equal to~$0$.
\end{cor}

\subsection{Comparing to heuristics}
\label{subsec:heur}

We now present a conjecture for height fluctuations of dimer models with gaseous facets in the limit shape on general domains, which generalizes our Theorems~\ref{thm:main_intro} and~\ref{thm:discrete_gauss_intro}. The conjecture involves a discrete Gaussian: In the conjecture, we specify the scale matrix (one of the parameters of the discrete Gaussian distribution), and then we match this to predictions of~\cite[Section 24.2]{Gor21} given for lozenge tilings of domains with holes. The shift parameter again retains an~$N$ dependence, and for a dimer graph of a given size~$N$ it is not fully specified by the conjecture; however, the dependence of the shift on~$N$ is specified. We choose to include this partial description of the shift because it (in particular the connection between the shift and the limit shape) helped the authors in the derivation of the result in Theorem~\ref{thm:discrete_gauss_intro}. 


Consider the height function on a large subgraph of the square lattice equipped with~$k \times \ell$ doubly periodic edge weights; we consider general but ``nice enough'' boundary conditions. First, endow the (multiply connected) liquid region with the complex structure induced by the mapping from the liquid region to the half of the spectral curve~$\mathcal{R}_0$ via the slopes of the limit shape. This mapping and the resulting complex structure are described in detail for the case of uniform lozenge tilings in~\cite[Section 4]{Ken08} and~\cite[Section 2.3]{KO07}, and for general dimer models in~\cite[Section 1.9]{ADPZ20}. Then, the double of the liquid region can be viewed as a compact Riemann surface~$\tilde{\mathcal{R}}$ of genus~$\tilde{g}$, where~$\tilde{g}$ is the number of gaseous facets in the limit shape. Denote the liquid region itself by~$\tilde{\mathcal{R}}_0$. Let~$\tilde{B}$ denote the period matrix of~$\tilde{\mathcal{R}}$. This matrix is pure imaginary (since $\tilde{\mathcal{R}}$ is an M curve~\cite{BCT22}), symmetric, and has positive definite imaginary part. Let~$\tilde{f}_i : \tilde{\mathcal{R}}_0 \rightarrow \mathbb{R}$,~$i=1,\dots,\tilde{g}$ be defined as in~\eqref{eqn:intro_fidef2} and~\eqref{eqn:intro_fidef3} but with~$\tilde{\mathcal{R}}_0$
 replacing~$\mathcal{R}_0$. Finally, define for each~$i=1,\dots, \tilde{g}$ a point $(u_i, v_i)$ in the~$i$th gaseous facet, and define $\tilde{H}_i$,~$i=1,\dots,\tilde{g}$ in an analogous way to the definition of~$H_i$ appearing before the statement of Theorem~\ref{thm:discrete_gauss_intro} in the introduction.
 \begin{conj}\label{conj:general}
    The fluctuations of the dimer model height function in the liquid region of a graph of scale~$N$ are approximated in distribution by
    \begin{equation}
        \mathfrak{g}_{\tilde{\mathcal{R}}_0}(q) + \sum_{i=1}^{\tilde{g}} \tilde{f}_i(q) \tilde{Z}_i ,
    \end{equation}
    where $\mathfrak{g}_{\tilde{\mathcal{R}}_0}$ is a Gaussian free field on the liquid region~$\tilde{\mathcal{R}}_0$ with Dirichlet boundary conditions, independent from~$\tilde{Z} = (\tilde{Z}_1,\dots, \tilde{Z}_{\tilde{g}})$, and~$\tilde{Z}$ is distributed as a mean-subtracted discrete Gaussian with scale matrix~$ -\tilde{B}^{-1}$ and with an~$N$-dependent shift parameter given by~$ N (\tilde{H}_1(u_1,v_1), \dots, \tilde{H}_{\tilde{g}}(u_{\tilde{g}}, v_{\tilde{g}})) + e_0$, for some~$e_0 \in \mathbb{R}^{\tilde{g}}$.
\end{conj}
The conjecture above may be justified by heuristics very similar to those leading up to Conjecture 24.2 in~\cite{Gor21}; in particular, keeping track of slightly more data in the argument there leads to the partial description of the shift parameter.  Indeed, the second order expansion of the surface tension functional around the limit shape, which is the basis of the argument there, takes the same form even in the setting of doubly periodic weights (this requires a computation, and Theorem 5.5 in~\cite{KOS06}). Therefore, provided that the gaseous facets effectively behave as holes (which should be justified in general, though it is supported by our results in this work), the arguments in~\cite{Gor21} can be applied to the setting of a domain with gaseous facets, as well. We comment further on the shift at the end of this subsection. First, we match the scale matrix in the conjecture to the one predicted by Conjecture 24.2 in~\cite{Gor21}.

Before addressing the general case we match the scale parameter in the setting where $\tilde{R}$ has genus~$1$. The predicted scale parameter is associated to a constant~$C > 0$ specified in~\cite[Conjecture 24.1]{Gor21}, which with our definition of the discrete Gaussian, is related to the scale matrix (which is just a single complex number in genus~$1$) via
\begin{equation}\label{eqn:taustar}
\tau^* = \i \frac{C}{ \pi}.
\end{equation}
Note that~$\tau^*$ is pure imaginary and has positive imaginary part. 

If~$\tau \coloneqq \tilde B_{1 1}$ is the single entry of the period matrix of the double of the liquid region~$\tilde{ \mathcal{R}}$ (which we assume is genus~$1$), then~$\tilde{ \mathcal{R}}$ is isomorphic (via the Abel map) to~$ \mathbb{T}_{\tau} = \mathbb{C}/(\mathbb{Z} + \tau \mathbb{Z})$, the torus with fundamental domain~$[0,1] \times [0, \Im(\tau)] \subset \mathbb{C}$. If we use the complex coordinate~$z = x + \i y$ from~$\mathbb{C}$ on this torus, then the function which is~$1$ on the compact oval, which corresponds to the curve~$\{y = \frac{\Im(\tau)}{2}\} \subset \mathbb{T}_{\tau}$, and~$0$ on the outer oval can be written as~$g_1(z) = \frac{2}{|\tau|} y$. Thus using~\eqref{eqn:taustar} and the formula for~$C$ from Conjecture 24.1 in~\cite{Gor21}, we have
\begin{align*}
   -  \pi \i \tau^* = \frac{\pi}{2} \int_{0}^1 \int_{0}^{|\tau|/2} \frac{4}{|\tau|^2} dx dy = \frac{\pi}{2} \frac{2}{|\tau|} 
\end{align*}
so that 
\begin{equation}\label{eqn:othertau}
\tau^* =  -\frac{1}{\tau} .
\end{equation}
This matches the scale matrix in the conjecture above.

In the higher genus setting, a similar match of the scale matrix in Conjecture~\ref{conj:general} and Theorem~\ref{thm:discrete_gauss_intro} (note that in this case~$\tilde{B} = B$) can be made to the one appearing in~\cite[Conjecture 24.2]{Gor21}. For the basis of the first homology group and other objects associated to~$\tilde{\mathcal{R}}$, we will use the same notation we used for the corresponding objects from~$\mathcal{R}$, but with a tilde. The key calculation analogous to the computation of~$\tau^*$ above is the following: With~$\tilde f_m$,~$m=1,\dots,\tilde g$ the harmonic functions on~$\tilde{\mathcal{R}}_0$ defined prior to the conjecture above (which we then extend to~$\tilde{\mathcal{R}}$ by reflection about the outer oval~$\tilde{A}_0$),
\begin{equation}\label{eqn:fmfl}
-\i \int_{\tilde{\mathcal{R}}_0} \nabla f_m \cdot \nabla f_l  dx \wedge d y = \int_{\mathcal{R}} \partial f_m \wedge \overline{\partial f_l} = 2 \tilde{B}_{m l}^{-1} .
\end{equation}
 To obtain~\eqref{eqn:fmfl}, one should write~$f_m(q) = \int_{\bar q}^q \sum_{j=1}^g \tilde{B}_{m, j}^{-1} \tilde \omega_j$, and note that periods around~$\tilde{A}_i$ and~$\tilde{B}_i$ of~$\partial f_m$ are~$\tilde{B}_{m i}^{-1}$ and~$\delta_{i m}$, respectively. The Riemann bilinear identity then gives~\eqref{eqn:fmfl}. We conclude that the scale matrix in Conjecture~\ref{conj:general} is given by
 \begin{equation}\label{eqn:scale_match}
     -\tilde{B}_{m l}^{-1} = \frac{\i}{2} \int_{\tilde{\mathcal{R}}_0} \nabla f_m \cdot \nabla f_l  dx \wedge d y.
 \end{equation}
and the right hand side of~\eqref{eqn:scale_match} is the prediction from~\cite[Conjecture 24.2]{Gor21}.

By Theorem~\ref{thm:discrete_gauss_intro} and Equation~\eqref{eq:flow_plus_spectral_data} together with the discussion leading up to it in Section~\ref{subsec:theta_prime}, we know the constant~$e_0$ in the Aztec diamond setup; however, in general it remains unspecified by the conjecture above. We very briefly explain below one possible route to computing it, following arguments of~\cite{BDE00} and~\cite{BG24}. Following exactly analogous heuristic computations regarding eigenvalues of random matrices in the~\emph{multi-cut} setting presented in~\cite{BDE00}, as well as rigorous results of~\cite{BG24}, one may attempt to more precisely compute the discrete fluctuations of the heights of gaseous facets via the computation of asymptotic expansions of a family \emph{refined parition functions}. In more detail, in the genus~$1$ setting, for each value of~$\delta' \coloneqq \delta/N$, where~$\delta$ is a lattice scale ``height fluctuation'' of the facet (assuming for the moment this is well defined), one might expect that there is an asymptotic expansion of the form 
\begin{equation}\label{eqn:part_exp}
\log \mathcal{Z}_{N,\delta'} = N^2 \mathcal{F}_2(\delta') + N \mathcal{F}_1(\delta') + o(N)
\end{equation}
for~$\mathcal{Z}_{N,\delta'}$ the partition function restricted to dimer configurations which have approximate normalized facet height~$\overline{h}(\text{facet}) + \delta'$ (assuming we have chosen a convention where the facet is flat). Although~$\delta = N \delta'$ lives in a lattice, it is expected that~$\mathcal F_1$ and~$\mathcal F_2$ depend smoothly on~$\delta'$, and expanding~\eqref{eqn:part_exp} around the critical point~$\delta' = 0$ of~$\mathcal{F}_2$ should lead to an expression for the asymptotic distribution of the discrete component. The predictions for values of the scale matrix and the~$N$-dependent part of the shift depend on~$\mathcal{F}_2$ and the limit shape only. It is exactly~$\mathcal{F}_1$ which might predict the missing constant~$e_0$ in the conjecture above. The same observation about the necessity of computing the linear term~$\mathcal{F}_1$ appears in Remark 24.11 of~\cite{Gor21} in a discussion about tilings of regions with holes.

\subsection{Example: Symmetric 2 $\times$ 2 periodic weights}
\label{subsec:2by2_ex}
Here we specialize to weights which are gauge equivalent to the one parameter family of two periodic models studied in~\cite{CY14},~\cite{CJ16},~\cite{DK21}. The weights analyzed in~\cite{CJ16}, specialized to~$b=1$ (which can be done without loss of generality), are obtained by setting~$k = \ell = 2$ in our setup and setting~$\alpha_{2,1} = \beta_{2,1} = \alpha_{1,2} = \beta_{1,2} = a$,~$\alpha_{1,1} = \beta_{1,1} = \alpha_{2,2} = \beta_{2,2} = a^{-1} $, and~$\gamma_{i,j} = 1$ for all~$i,j$. Compare Figure~\ref{fig:Aztec_example} with these specializations with Figure 1 of~\cite{CJ16}. 

We remark that in our setup, the fundamental domain has~$4$ vertices of each type, while (as illustrated in their work) for this special case of~$2\times 2$ weights it is possible to choose a different~$\mathbb{Z}^2$ action, such that the fundamental domain contains~$2$ vertices of each type. The two spectral curves thus obtained will, of course, be isomorphic. Here we continue using our convention for the~$\mathbb{Z}^2$ action. 

In order to specialize our result to this setup, we must compute~$B = B_{1,1}$ (the single entry of the period matrix), and the real number~$e_{\mathrm w_{0,0}}^{(k N)} \in \mathbb{R}$. We postpone the computation of~$B$, as it has no closed form expression in terms of elementary functions, though it can be ``explicitly'' computed in terms certain elliptic integrals, which can be evaluated numerically.

  For the characteristic polynomial, we get 
\begin{equation}\label{eqn:pzw_a}
\frac{z}{w} P(z,w) =  -4 - 4/a^2 - 4 a^2 - 2/w - 2 w - 2/z + 1/(w z) + w/z - 2 z + z/w + w z.
\end{equation}
To get an idea for the ``sheeted cover'' picture of the surface, we compute the values of~$z$ which are branch points. To do this, we first compute the roots~$w$ of~$P(z,w) = 0$, which is a quadratic equation in~$w$ with coefficients depending on~$z$. We find that the branch points, or values of~$z$ where the equation for~$w$ has a double root, occur at~$z =0, -a^{\pm 2}$. It is not difficult to see that~$z = \infty$ is a branch point, too. Therefore, in the sheeted cover picture of~$\mathcal{R}$, we have cuts in~$\mathbb{R}$ along~$[-a^2, 0]$ and~$[-\infty,-\frac{1}{a^2}]$ if~$a < 1$, and along~$[-a^{-2}, 0]$ and~$[-\infty,-a^2]$ if~$a > 1$. 

Another valuable picture is the amoeba, depicted in Figure~\ref{fig:apt7amoeba}. Generically, for~$k=\ell=2$, there are two angles intersecting each line at infinity. In this special case, each pair of angles has merged into a single one. This merging is apparent from the amoeba illustrated in Figure~\ref{fig:apt7amoeba}. The merging of angles at~$z=0$ and~$z=\infty$ (horizontal tentacles in both directions of the amoeba) is consistent with the fact that~$z = 0$ and~$z = \infty$ are branch points of~$(z,w) \mapsto z$.

\begin{figure}
    \centering
    \includegraphics[width=0.5\linewidth]{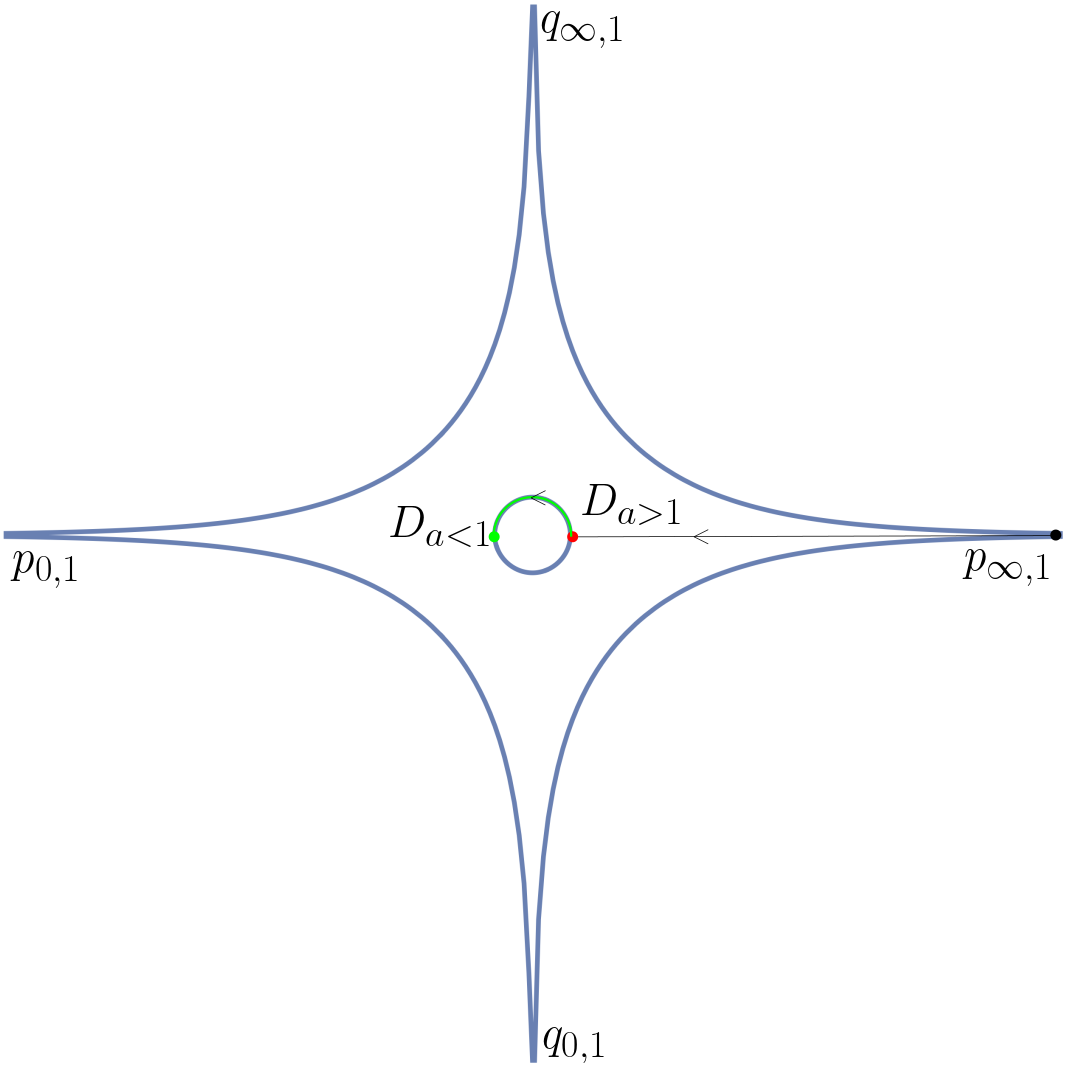}
    \caption{The amoeba for~$P(z,w)$ as in~\eqref{eqn:pzw_a} for~$a^{\pm} = 0.7$. The base point for the Abel map is~$ p_{\infty,1}$, and we show in each case an integration path from~$p_{\infty,1}$ to~$D$ (the green point if~$a<1$, the red one if~$a>1$) used to compute~$u(D)$. If~$a>1$, then this path consists of only the black segment, and if~$a <1$ the green segment is appended to this.}
    \label{fig:apt7amoeba}
\end{figure}

Next, we explain a computation of the real number~$e_{\mathrm w_{0,0}}^{(k N)}$. We follow the procedure described in detail in~\cite[Section 5.4]{BB23}, and we also exactly use the notation defined there. First, we must compute the divisor~$D$, the set of common zeros of~$\adj K_{G_1}(z,w)_{\mathrm b, \mathrm 
 w_{0,0}}$ as~$\mathrm b$ ranges over black vertices in~$G_1$. Using this, we must then compute the quantities
\begin{align}
     e_{\mathrm w_{0,0}} &= -u(D) + \Delta  \label{eqn:quant1}\\
     e_{\mathrm w_{0,0}}^{(0)} &= e_{\mathrm w_{0,0}}  + u(q_{0,1})-u(p_{0,1})  \label{eqn:quant2}\\
     e_{\mathrm w_{0,0}}^{(k N)} &= e_{\mathrm w_{0,0}}^{(0)} + 4 N  (u(q_{0,1})-  u(p_{0,1})) \label{eqn:quant3}
\end{align}
where~$u$ is the Abel map and~$\Delta$ is the vector of Riemann constants.  Equations~\eqref{eqn:quant1},~\eqref{eqn:quant2}, and\eqref{eqn:quant3} are Equations~\eqref{eqn:ew00first},~\eqref{eqn:ew000}, and \eqref{eqn:ew00} specialized to this particular setting. We have used the fact that each of the four pair of angles have merged; i.e.~$p_{0,1} = p_{0,2}$ and similarly for the other pairs of angles. In the definition of the Abel map, we choose basepoint~$q_0 = p_{\infty,1}$, as in Section~\ref{subsec:DCdist}. In our setup~$p_{\infty,1} = (\infty, 1)$. Moreover, since we are in genus~$1$,~$\Delta = 1/2 + B/2$. (In genus~$1$,~$\Delta$ loses its dependence on the base point.)
 
 We compute the divisor~$D$; it consists of the single point~$(z, w) = (-a^2, -1) \in \mathcal{R}$, which is on the compact oval, as the general theory~\cite{KO06} predicts. Next, to compute~\eqref{eqn:quant1}-\eqref{eqn:quant3} we first record some values of the Abel map. To do this, we note that evidently~$\tilde{P} = \frac{z}{w} P$ as in~\eqref{eqn:pzw_a} admits the symmetries
 \begin{equation}\label{eqn:Psym}
     \tilde{P}(z,w) = \tilde{P}(z^{-1}, w) = \tilde{P}(w,z).
 \end{equation}
 These symmetries of the spectral curve leave the~$A$ cycle and thus the holomorphic one form invariant (up to a sign), and therefore (using them to permute the angles) they imply that~$u(q_{\infty,1}) = 1/4$,~$u(p_{0,1}) = 1/2$,~$u(q_{0,1}) = 3/4$. As a result, we immediately see that~$e_{\mathrm w_{0,0}}^{(k N)} = e_{\mathrm w_{0,0}}^{(0)}$. This is a restatement of the known fact that these special~$2 \times 2$ weights are what is known as a \emph{torsion point}~\cite{BD22}.

 We must also compute~$u( (-a^2,-1) )$. This depends on whether~$a > 1$ or~$a < 1$. We split the computation into cases; the projection to the amoeba of the integration contour we use to compute this in each case is shown in Figure~\ref{fig:apt7amoeba}:
 \begin{enumerate}
     \item $ a > 1$: In this case, by the symmetry under~$w \mapsto 1/w$, the result must be pure imaginary, which means that~$u((-a^2,-1)) = B/2 $.
     \item $ a < 1$: Now by using the~$a > 1$ case and adding half of the~$A$ cycle to the previous the contour, we have~$u((-a^2,-1)) = 1/2 + B/2 $.
 \end{enumerate}

 Using the computations above and~\eqref{eqn:quant1}-\eqref{eqn:quant3}, we know  
 \begin{align*}
     e_{\mathrm w_{0,0}}^{(k N)} &=e_{\mathrm w_{0,0}} + 1/4 \\
     &= - u( (-a^2,-1) )+ 3/4 + B/2 .
\end{align*}
Therefore, as an element of the circle~$\mathbb{R}/\mathbb{Z}$, we have
\begin{equation}\label{eqn:aekn}
    e_{\mathrm w_{0,0}}^{(k N)} =
    \begin{cases}
       - \frac{1}{4}, & a > 1 \\
        \frac{1}{4}, & a < 1 .
    \end{cases}
\end{equation}

So, specializing Corollary~\ref{eqn:discrete_gauss_corg} to this setup, and recalling again that~$e_{\mathrm w_{0,0}}^{(k N)}$ is independent of~$N$ gives the following.

\begin{cor}\label{cor:discrete_gauss_a}
    The discrete component~$Z = Z_1$ \emph{converges in distribution} to~$X - \mathbb{E}[X]$, where~$X$ is a discrete Gaussian random variable with scale parameter~$\tau = -1/B \in \i \mathbb{R}_{>0}$, and shift parameter given by
    \begin{equation}\label{eqn:ae}
    e =
    \begin{cases}
        -\frac{1}{4}, & a > 1 \\
        \frac{1}{4}, & a < 1 .
    \end{cases}
\end{equation}
\end{cor}

\begin{proof}
    The display~\eqref{eqn:aekn} together with Corollary~\ref{eqn:discrete_gauss_corg} specialized to~$k = \ell = 2$ and the above weights leads to the statement above, modulo the part about convergence in distribution. However, the Corollary implies convergence of cumulants, and hence of moments; therefore, since a centered discrete Gaussian distribution is uniquely determined by its moments (its moment generating function is complex analytic with nonzero radius of convergence at~$z=0$) the convergence of moments implies convergence in distribution~\cite[Theorem 30.2]{B95}. 
\end{proof}

Next, we explain how to compute~$B$. Going back to the sheeted cover picture, the~$A$ cycle is the loop obtained by traversing~$[-\frac{1}{a^2},-a^2]$ (or~$[-a^2,-\frac{1}{a^2}]$) once along the top sheet and then once in the other direction on the bottom sheet. The~$B$ cycle can be chosen to be any loop in the top sheet containing the cut starting at~$0$, i.e.~$[-a^2,0]$ if~$a < 1$ or~$[-a^{-2},0]$ if~$a>1$, and not the other cut.

Next, we compute the holomorphic one form. It is a classical fact (and not difficult to check) that in this situation the holomorphic one form on~$\mathcal{R}$ can be computed in terms of the coordinate~$z$, for some constant~$c$, as
\begin{equation}\label{eqn:holform}
    \omega = \frac{1}{c} \frac{d z}{\sqrt{z (z+a^2)(z+a^{-2})}}
\end{equation}
where the two choices of branch of the square root correspond to the two sheets of~$\mathcal{R}$ (we make the choice that the ``top'' sheet corresponds to the principle brach). The constant~$c$ in~\eqref{eqn:holform} is chosen to normalize the integral around the~$A$ cycle to~$1$, so it is given by 
\begin{equation}\label{eqn:Acyc}
   c = \pm 2 \int_{-a^{-2}}^{-a^2} \frac{d z}{\sqrt{z (z+a^2)(z+a^{-2})}}.
\end{equation}
(The sign depends on a correct choice of orientation for the~$A$ cycle, and on whether~$a > 1$ or~$a < 1$, but it can ultimately be fixed by the apriori knowledge that~$\Im(B) > 0$.) The expression above is known as an \emph{elliptic integral}, and it can efficiently be computed numerically. Plugging back in the value of~$c$ to~\eqref{eqn:holform}, we may numerically integrate to compute the value of~$B$. For example, if~$a = .7$, then~$B \approx  0.521828 \i$.

\begin{remark}
The parameter~$t$ in the definition of shift~$e_{\mathrm w_{0,0}}^{(kN)}$ (as in~\eqref{eq:flow_plus_spectral_data}) depends on how the periodic edge weights align with the Aztec diamond. In other words, by shifting the edge weights inside the fundamental domain, we can obtain different measures. In fact, there are four distinct measures that can be obtained through such shifts. Namely, we can choose
\begin{equation}
    \begin{cases}
        \text{weights } 1: & \alpha_{2,1} = \beta_{2,1} = \alpha_{1,2} = \beta_{1,2} = a, \quad \alpha_{1,1} = \beta_{1,1} = \alpha_{2,2} = \beta_{2,2} = a^{-1}, \quad \gamma_{i,j} = 1, \\
        \text{weights } 2: & \alpha_{2,1} = \beta_{2,1} = \alpha_{1,2} = \beta_{1,2} = a^{-1}, \quad \alpha_{1,1} = \beta_{1,1} = \alpha_{2,2} = \beta_{2,2} = a, \quad \gamma_{i,j} = 1, \\
        \text{weights } 3: & \alpha_{2,1} = \beta_{1,1} = \alpha_{1,2} = \beta_{2,2} = a, \quad \alpha_{1,1} = \beta_{2,1} = \alpha_{2,2} = \beta_{1,2} = a^{-1}, \quad \gamma_{i,j} = 1, \\
        \text{weights } 4: & \alpha_{2,1} = \beta_{1,1} = \alpha_{1,2} = \beta_{2,2} = a^{-1}, \quad \alpha_{1,1} = \beta_{2,1} = \alpha_{2,2} = \beta_{1,2} = a, \quad \gamma_{i,j} = 1,
    \end{cases}
\end{equation}
The characteristic polynomial as well as the limit shape are the same for these four sets of weights. The divisor~$D$ and hence~$e_{\mathrm w_{0,0}}^{(k N)} $, however, depend on the choice. Indeed, if~$a<1$, computations similar to the above show the divisor consists of the point
\begin{equation}
    (z,w)=
    \begin{cases}
        (-a^2,-1), & \text{weights }1, \\
        (-a^{-2},-1), & \text{weights }2, \\
        (-1,-a^2), & \text{weights }3, \\
        (-1,-a^{-2}), & \text{weights }4,
    \end{cases}
\quad \text{and} \quad 
    e_{\mathrm w_{0,0}}^{(k N)}=e_{\mathrm w_{0,0}}^{(0)}=
    \begin{cases}
        \frac{1}{4}, & \text{weights }1, \\
        \frac{3}{4}, & \text{weights }2, \\
        0, & \text{weights }3, \\
        \frac{2}{4}, & \text{weights }4.
    \end{cases}
\end{equation}
The first two cases correspond to the two cases in Corollary~\ref{cor:discrete_gauss_a}.
\end{remark}

\section{Steepest Descent Arguments}
\label{sec:steepest_arguments}

In this section we will perform the asymptotic analyses necessary to compute asymptotics of the inverse of the Kasteleyn matrix. We require an asymptotic expansion of~$K^{-1}(\mathrm b, \mathrm w)$ for each of~$\mathrm b$ and~$\mathrm w$ in four regimes: in the bulk (i.e. region~\eqref{item:first}, liquid region), near the edge (i.e. region~\eqref{item:second}, near the liquid-gas or liquid-frozen boundary), at the edge (i.e. region~\eqref{item:third}, near the liquid-gas or liquid-frozen boundary), and in a facet (i.e. region~\eqref{item:fourth}, in the frozen or gas region). We will derive estimates for~$K^{-1}$ for each of~$4^2 = 16$ cases when either point is in each of four regimes. The lemmas presented in Section~\ref{subsec:lemmastatements} are proven in the same order as they are presented.

\subsection{Edge behavior of the critical point}
In this section, we prove Lemma~\ref{lem:sqrt_singularity}.
\begin{proof}[Proof of Lemma~\ref{lem:sqrt_singularity}]
For concreteness, we assume~$\eta=\eta_{fb}+\epsilon$, the case~$\eta=\eta_{fb}-\epsilon$ follows in a similar way. 

Using the explicit formula for the action function we get
    \begin{equation}\label{eqn:Fperturbation}
        \partial_z F(z; \xi, \eta_{fb} + \epsilon) = \partial_z F(z; \xi, \eta_{fb}) + \epsilon \frac{\ell}{2} \frac{1}{z}.
    \end{equation}
Suppose that~$z(\xi, \eta_{fb}) = z_0 \in \mathbb{R}$. We know~$z_0 \neq 0$ since we are away from the tangency points. We make the replacement~$\tilde z(\xi,\eta) =z(\xi,\eta)-z_0$, and set~$f_0(z) = \partial_z F(z+z_0; \xi, \eta_{fb})$ and~$f_\epsilon(z) = \partial_z F(z+z_0; \xi, \eta_{fb} +\epsilon)$. Then~$\tilde{z}(\xi, \eta_{fb}+\epsilon)$ is a zero of~$f_\epsilon$ and~$f_0$ has a zero of order~$2$ at~$0$ (recall we are away from cusps). 

We look for two zeros of~$f_\epsilon$ near~$0$. Let
\begin{equation}\label{eqn:cauchyeps}
\tilde z_\pm(\epsilon) =  \frac{1}{2\pi \i}  \int_{C_{\delta,\pm}} z \frac{f_\epsilon'(z)}{f_\epsilon(z)} dz,
\end{equation}
with~$C_{\delta,+}$~($C_{\delta,+}$) a semicircular arc around~$0$ in the upper (lower) half plane together with the straight segment along~$[-\delta, \delta] \subset \mathbb{R}$ for an appropriate choice of~$\delta = \delta(\epsilon) > 0$ if~$(\xi,\eta)$ is in the liquid region, and with~$C_{\delta,+}$~($C_{\delta,+}$) a semicircular arc around~$0$ in the right (left) half plane together with the straight segment along~$[-\delta \i, \delta \i] \subset \mathbb{R}$, both oriented in positive direction. By the Cauchy integral formula (or argument principle),~$\tilde z_\pm(\epsilon)$ is a zero of~$f_\epsilon$ if it has a zero in the interior of~$C_\delta$ and vanishes identically otherwise. 

We take~$\delta = A \sqrt{\epsilon}$ for large enough but fixed~$A > 0$. For~$z\in C_{\delta,\pm}$, we Taylor expand the numerator and denominator of the integrand of~\eqref{eqn:cauchyeps} around~$z=0$ using~\eqref{eqn:Fperturbation}:
\begin{equation*}
z \frac{f_\epsilon'(z)}{f_\epsilon(z)}=z \frac{f_0'(z)-\epsilon\frac{\ell}{2(z+z_0)^2}}{f_0(z)+\epsilon\frac{\ell}{2(z+z_0)}}
=z \frac{f_0''(0)z+O(\epsilon)}{\frac{1}{2}f_0''(0)z^2+\epsilon\frac{\ell}{2z_0}+O(\epsilon^{3/2})}.
\end{equation*}
If~$A$ is large enough so that one zero of~$\frac{1}{2}f_0''(0)z^2+\epsilon\frac{\ell}{2z_0}$ lies in the interior of~$C_\delta$ for all small enough~$\epsilon>0$, then
\begin{equation*}
z \frac{f_0''(0)z+O(\epsilon)}{\frac{1}{2}f_0''(0)z^2+\epsilon\frac{\ell}{2z_0}+O(\epsilon^{3/2})}=\frac{f_0''(0)z^2}{\frac{1}{2}f_0''(0)z^2+\epsilon\frac{\ell}{2z_0}}\left(1+O\left(\epsilon^{1/2}\right)\right).
\end{equation*}
It follows that
\begin{equation}\label{eqn:zeros_close_edge}
\tilde z(\epsilon)=\frac{1}{2\pi \i}  \int_{C_{\delta}}\frac{f_0''(0)z^2}{\frac{1}{2}f_0''(0)z^2+\epsilon\frac{\ell}{2z_0}}\d z+O(\epsilon)=\pm\sqrt{-\frac{\epsilon\ell}{f_0''(0)z_0}}+O(\epsilon),
\end{equation}
where the sign and the square root are taken so that~$\tilde z(\epsilon)$ is in the interior of~$C_\delta$. This proves~\eqref{eqn:equivimz} and~\eqref{eqn:equivimz2} with~$a=\sqrt{\frac{\ell}{|f''_0(0)z_0|}}$.

The equivalence~\eqref{eqn:Fpp} can then be obtained from~\eqref{eqn:equivimz} and and~\eqref{eqn:equivimz2} together with the fact that
\begin{align*}
    F''(z(\xi, \eta_{fb} \pm \epsilon);\xi, \eta_{fb} \pm \epsilon ) &= F''(z(\xi, \eta_{fb});\xi, \eta_{fb} \pm \epsilon ) \\
    &+\left( z(\xi, \eta_{fb} \pm \epsilon) - z(\xi, \eta_{fb}) \right) F'''(z ;\xi, \eta_{fb} \pm \epsilon )  \\
    &= \left( z(\xi, \eta_{fb} \pm \epsilon) - z(\xi, \eta_{fb}) \right) F'''(z ;\xi, \eta_{fb} \pm \epsilon ) + O(\epsilon) .
\end{align*}

The uniformity statement in the lemma is implied by our assumption on the point~$(\xi, \eta_{fb})$ that it stays away from tangency points and cusps.
\end{proof}

\subsection{Steepest descent in the bulk}

\begin{proof}[Proof of Lemma~\ref{lem:steepest}]
To prove the lemma, we perform a steepest descent analysis on~$K^{-1}(\mathrm b,\mathrm w)$ given in Lemma~\ref{lem:double_int_simple}. We complete the proof in four steps to attempt to clarify the argument. In the first step, we describe the paths of steepest descent and ascent for the double contour integral. In the second step, we obtain the residues we pick up in the deformation of the original curves to the curves of steepest descent and ascent. In the third step, we derive the leading order term and bound the error terms in the double contour integral. Finally, in the fourth step, we bound the single integral obtained in step two.

Another steepest descent analysis of~$K^{-1}(\mathrm b,\mathrm w)$ was obtained in~\cite{BB23} for microscopically close~$\mathrm b$ and~$\mathrm w$. The argument here is different, however, we can, and will reuse the curves defined there as well as certain properties of the integrand stated in Lemma~\ref{lem:double_int_simple}.

\begin{figure}
\vspace{-40pt}
\centering
\includegraphics[scale=.18]{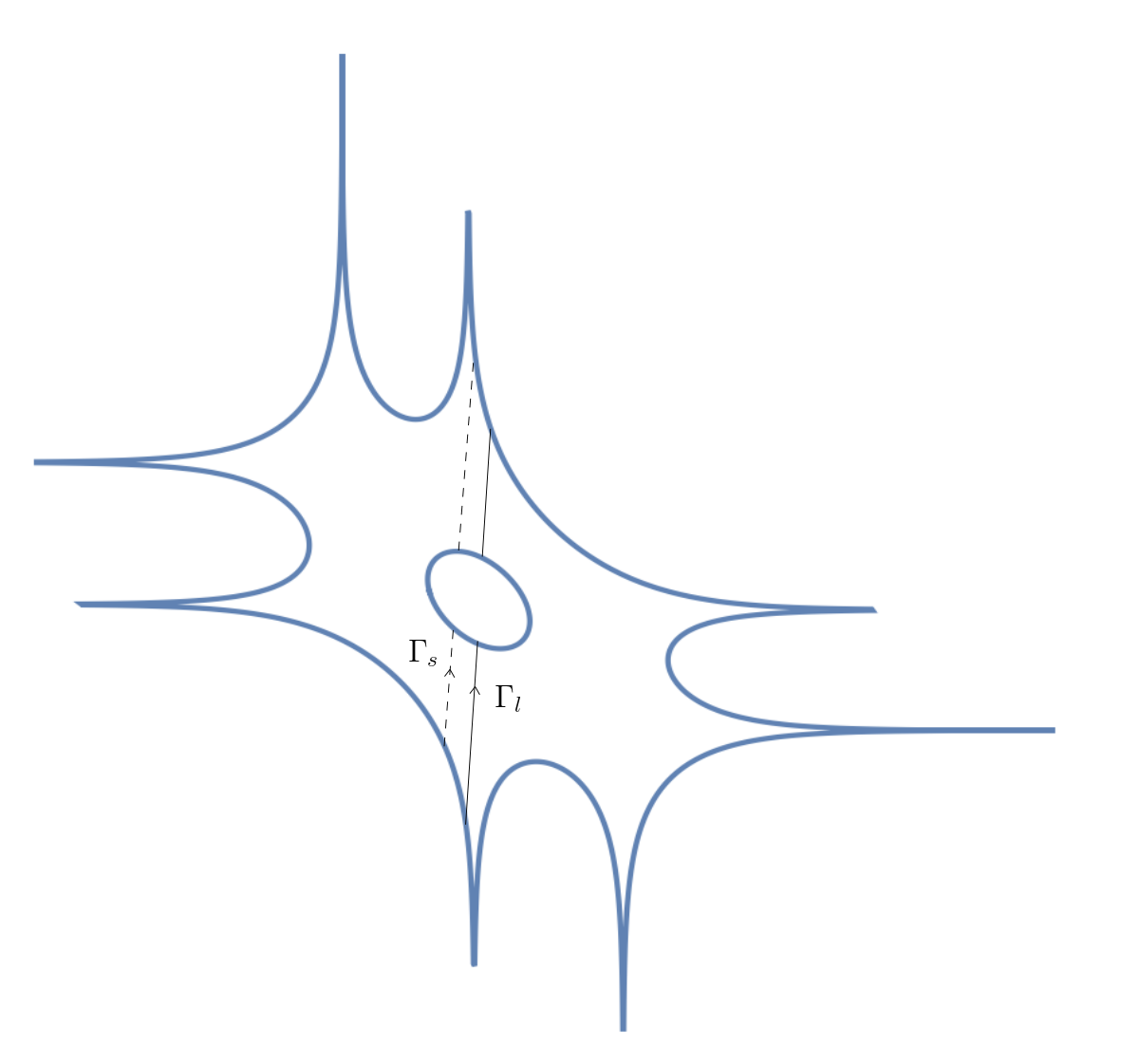}
\includegraphics[scale=.125]{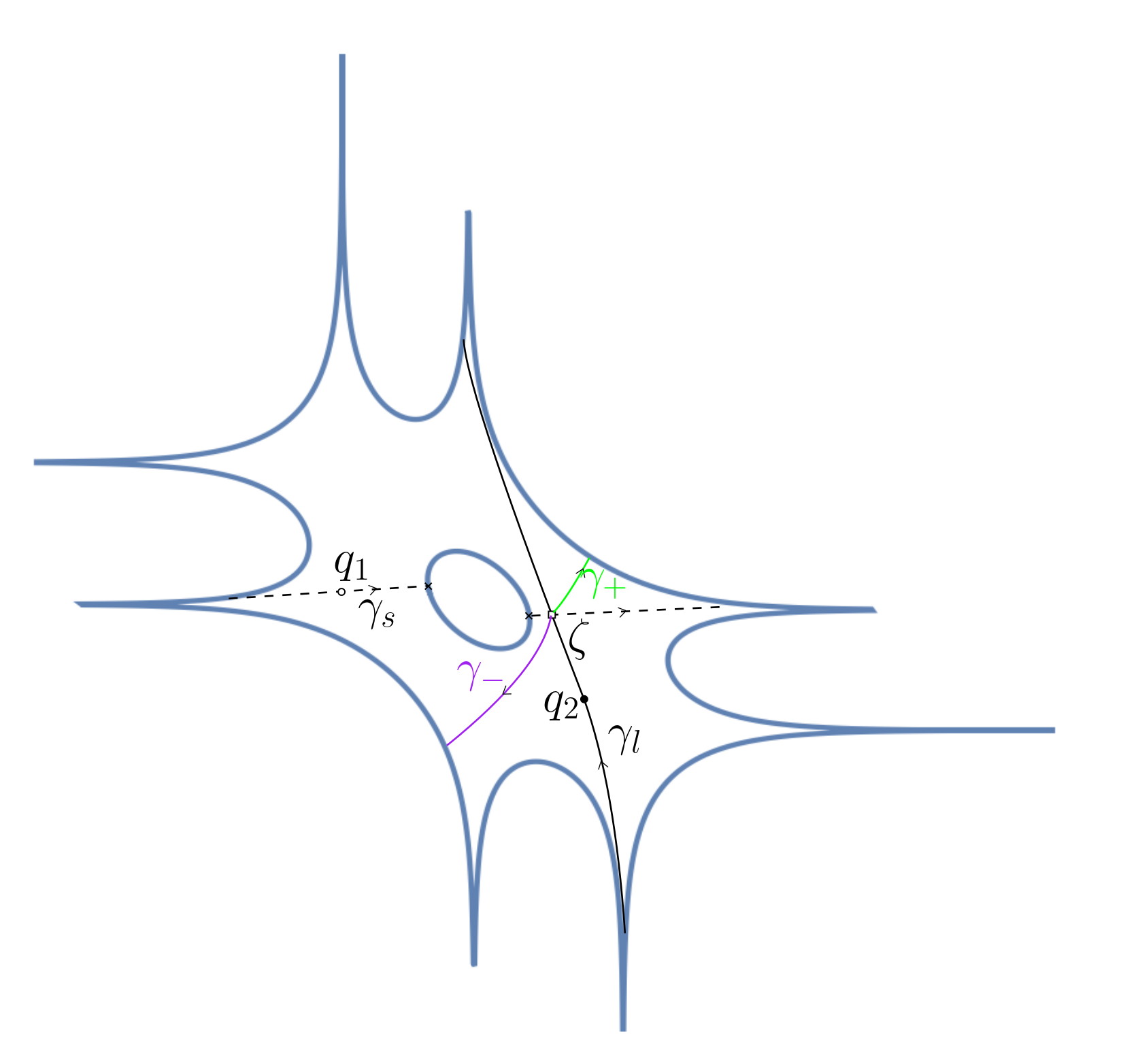}
\caption{A schematic picture of the new steepest descent contours (right)~$\gamma_s$ (dashed) and~$\gamma_l$ (solid) compared to the starting contours~$\Gamma_s$ and~$\Gamma_l$ (left). The path~$\gamma_s$ passes through the critical point~$q_1$ of~$F_1(q')$, and the path~$\gamma_l$ through the critical point~$q_2$ of~$F_2(q)$. The green contour is from the~$q=q'$ residue. The resulting single integral (which combines with the term~$I_1$ in~$K^{-1}$) could be over the green or the purple contour, see the second step in the proof of Lemma~\ref{lem:steepest}.}
\label{fig:descent_contours}
\end{figure}

\paragraph{Deforming the contours:}
We first invoke the argument of~\cite[Lemma 6.5]{BB23} to deform the contours~$\Gamma_s$ and~$\Gamma_l$ (called~$\widetilde{\Gamma}_s$ and~$\widetilde{\Gamma}_l$ there) to new ones~$\gamma_s$ and~$\gamma_l$, respectively, which are preserved (up to orientation) by complex conjugation and which are \emph{steep descent contours} in~$\mathcal{R}_0$, meaning that they have the following properties:
\begin{enumerate}
\item Along~$\gamma_s \cap \mathcal{R}_0$, for~$q' \neq q_1$,~$\Re[F_1(q') - F_1(q_1)] < 0$ and~$\Re[F_1]$ is strictly decreasing moving along the contour~$\gamma_s$ away from~$q_1$. \label{eq:steepest_descent_analysis_1} 
\item Along~$\gamma_l \cap \mathcal{R}_0$, for~$q \neq q_2$,~$\Re[F_2(q) - F_2(q_2)] > 0$ and~$\Re[F_2]$ is strictly increasing moving along the contour~$\gamma_l$ away from~$q_2$.\label{eq:steepest_descent_analysis_2}
\end{enumerate}

By the arguments in~\cite[Lemma 6.5]{BB23}, the contours~$\gamma_s$ and~$\gamma_l$ can be chosen so to exactly match the contours of steepest descent for~$F_1(q')$ and~$-F_2(q)$ through points~$q_1$ and~$q_2$, respectively. This will only require deforming~$\Gamma_s$ through some angles~$p_{0,j}$ and~$p_{\infty, j}$ and will only require deforming~$\Gamma_l$ through some angles~$q_{0,j}$ and~$q_{\infty, j}$; neither of these deformations crosses poles of the integrand, which, in each variable, is a meromorphic one form on the compact surface~$\mathcal{R}$. The new contour~$\gamma_s$ will begin at one of the angles~$p_{0, j}$ and end at one of the angles~$p_{\infty, j}$ (recall that these are the angles corresponding to horizontally stretched tentacles of the amoeba), and the new contour~$\gamma_l$ will begin at one of the angles~$q_{0, j}$ and end at one of~$q_{\infty, j}$. We cannot exclude the possibility that one of the steepest descent contours intersects the real part of~$\mathcal{R}$ at a different critical point of the corresponding action. In this case there is always a way to continue the contour so that it starts and ends at an angle, while preserving steep descent properties~(a) and~(b) above; in fact we can simply continue them along steepest descent contours moving away from the new critical point. An example of this situation is shown in Figure~\ref{fig:descent_contours} for~$\gamma_s$; there~$\gamma_s$ intersects a critical point along the compact oval (the two critical points along the compact oval are indicated with crosses), and it is continued by moving away from the compact oval starting from the other critical point.

One consequence of this choice of contours is that the double contour integral in~\eqref{eqn:double_int_simple} over the parts of the contours where either variable is in the complement of an~$O(\delta)$ neighborhood  of its critical point (either~$q_1$ or~$q_2$), is negligible, and in fact decays exponentially in a power of~$N$. Thus, one may restrict attention to where both integration variables are in their corresponding~$O(\delta)$ neighborhood of the critical point. We will ultimately choose~$\delta$ going to~$0$ with~$N$, see the third step below.

\paragraph{Keeping track of the residues picked up during the deformation:} During the deformation of~$\Gamma_s$ and~$\Gamma_l$ to~$\gamma_s$ and~$\gamma_l$, respectively, in the double contour integral in~\eqref{eqn:double_int_simple}, the only residue picked up is the one at~$z = z'$. Denote by~$\zeta$ the intersection point of~$\gamma_s$ and~$\gamma_l$ in~$\mathcal R_0$. If~$\ell x_2 + i_2 \leq \ell x_1 + i_1$, the inverse Kasteleyn matrix~\eqref{eqn:kast_inv_2} has no single integral term, and the residue gives the single integral
\begin{equation}\label{eq:rough_single_integral}
\frac{1}{2\pi\i}\int_{\gamma_{+}} G_{i_2,i_1}(q,q)_{j_1+1,j_2+1} e^{N (F_1(q)- F_2(q))}  \frac{d z}{z},
\end{equation}
where the contour~$\gamma_{+}$ starts at the intersection point~$\zeta$ and moves, viewing the part in~$\mathcal{R}_0$ as a subset of the amoeba, to~$A_{0,1}$, which is the upper-right boundary component of the amoeba. Otherwise if~$\ell x_2 + i_2 > \ell x_1 + i_1$, this residue will cancel out part of~$I_1$ (the single integral in the formula for the inverse Kasteleyn~\eqref{eqn:kast_inv_2}). In this latter case, we again end up with the integral~\eqref{eq:rough_single_integral} but with some contour~$\gamma_{-}$ instead of~$\gamma_+$. The image of the curve~$\gamma_-$ in the amoeba consists of a contour from the image of~$\zeta$ to a point~$p_0 \in A_{0,k+\ell+1}$, which is the lower-left boundary component of the amoeba. Both cases are shown in Figure~\ref{fig:descent_contours}.

We emphasize that the new contours~$\gamma_s$ and~$\gamma_l$ necessarily intersect, so that~$\zeta$ is always well defined. They might intersect more than once, in which case~$\gamma_{\pm}$ consists of several connected components, each of which starts and ends at an intersection point, or at~$A_{0,1}$ or~$A_{0,k+\ell+1}$, respectively.

\paragraph{Bounding errors, and computation of the leading order behavior from neighborhoods of critical points:}
Now we compute asymptotics of the double contour integral~\eqref{eqn:double_int_simple}, with~$\gamma_s$ and~$\gamma_l$ replacing~$\Gamma_s$ and~$\Gamma_l$. Let~$q'$ and~$q$ be the variables in the double contour integral where we integrate~$q'$ over~$\gamma_s$ and~$q$ over~$\gamma_l$. 

Following standard arguments, we expect that the double integral is dominated by parts of the integration contours in which both~$z' = z(q')$ and~$z = z(q)$ are in an~$O(\delta)$ neighborhood (in the local~$z$ plane) of critical points~$z_1 = z(q_1)$,~$z_2 = z(q_2)$. Following the arguments of~\cite{BF14}, we will take~$\delta = N^{-\frac{1}{4}}$. By the properties~(a) and~(b) of the curves~$\gamma_s$ and~$\gamma_l$, restricting to these neighborhoods gives an error
\begin{equation}
e^{N (\Re[F_1(q_1)] - \Re[F_2(q_2)])} O(e^{-\mu N \delta^{2}}),
\end{equation}
with~$\mu \sim \min_{i=1,2}(|F_i''(q_i)|)$. In the bulk~\eqref{item:first}, the critical point is simple, so it follows from Lemma~\ref{lem:sqrt_singularity}, that we have the lower bound~$\mu \geq c N^{-\frac{1}{6}}$ for some~$c > 0$, so that~$O(e^{-\mu N \delta^{2}}) = O(e^{-c N^{\frac{1}{3}}} )$, which is negligible in comparison to the error term in the statement. Thus, we have~$4$ leading order contributions from~$\delta$-neighborhoods of critical points, since~$q'$ can be near~$q_1$ or~$\overline{q_1}$ and~$q$ can be near~$q_2$ or~$\overline{q_2}$. 

Consider the contribution~$I_2^{++}$ where~$q'$ is near~$q_1$ and~$q$ is near~$q_2$. Now, using~$z',z$ as local coordinates for~$q', q$, we may replace~$F_1(q')$ and~$F_2(q)$ with their quadratic Taylor expansions about~$q_1 = (z_1, w_1)$ and~$q_2=(z_2,w_2)$, and we may replace the local contour by a straight line which matches the direction of steepest descent (for~$F_1$) or ascent (for~$F_2$). Finally we may change coordinates as~$z' = z'(x') = z_1 +  \frac{1}{\sqrt{-F_1''(z_1)}} x'$ and~$z = z(x) =  z_2 +  \frac{1}{\sqrt{F_2''(z_2)}} x$, where in each square root the branch of the square root is chosen so that the positive~$x$ direction agrees with the orientation of the contour, and we get 
\begin{multline}
  I_2^{++}
    = e^{N (F_1(q_1)- F_2(q_2))} \frac{1}{\sqrt{-F_1''(z_1)}} \frac{1}{\sqrt{F_2''(z_2)}} \\
    \times \frac{1}{ (2\pi\i)^2}\int_{-\delta \sqrt{|F_1''(z_1)|}}^{\delta \sqrt{|F_1''(z_1)|}} \int_{-\delta\sqrt{|F_2''(z_2)|}}^{\delta \sqrt{|F_2''(z_2)|}}
    G_{i_2, i_1}(q', q)_{j_1+1,j_2+1}  
     \frac{1}{z(z-z')} \\
     \exp \left(N \left( -(x')^2/2 -  x^2/2 + O\left(\frac{(x')^3}{\sqrt{|F_1''(z_1)|^3 }}\right) + O\left(\frac{x^3}{\sqrt{|F_2''(z_2)|^3 }}\right) \right)\right)  \d x \d x' . \label{eqn:double_int_local1} 
\end{multline}
At this stage it is important to point out that because~$|(\xi_{1,N},\eta_{1,N}) - (\xi_{2,N},\eta_{1,N})| > N^{-\frac{1}{16}}$ and~$q \mapsto (\xi(q), \eta(q))$ is smooth (and in particular Lipschitz) up to the boundary~$\partial \mathcal{R}_0$ away from tangency points, the function 
$$ G_{i_2, i_1}(q', q)_{j_1+1,j_2+1}  
     \frac{1}{z(z-z')}$$
is bounded above in absolute value by~$C N^{\frac{1}{16}}$ along the local contours, where~$C$ is uniform in all parameters. Indeed, if~$z = z'$ but~$w \neq w'$, then by Lemma~\ref{lem:double_int_simple},~$G_{i_2, i_1}(q',q) = 0$, and this zero cancels out the pole from~$\frac{1}{z-z'}$. (In particular, the function in the display as a function of~$q$ always has a \emph{simple} pole at~$q = q'$, even at a branch point of~$z$, which implies that the bound is valid as long as both points are away from cusps and tangency points.)

As described in Lemma 6.6 and Proposition 6.9 in~\cite{BF14}, up to an error of size
\begin{equation}\label{eqn:biggest_error}
 e^{N (\Re [F_1(q_1)]- \Re[F_2(q_2)])} \frac{1}{2 \pi N \sqrt{|F_1''(z_1)|} \sqrt{|F_2''(z_2)|}}O(N^{-\frac{1}{8}}),
 \end{equation}
 we may replace all functions in the integrand with their values at~$(z',w') = (z_1,w_1), (z,w) = (z_2,w_2)$ and omit the cubic error terms in the exponential in~\eqref{eqn:double_int_local1}. Furthermore, we make the substitution~$x \rightarrow x N^{-\frac{1}{2}}$ and~$x' \rightarrow x' N^{-\frac{1}{2}}$, and we also may extend the integration contours to~$\pm \infty$ at the cost of an exponentially small error.

This gives
\begin{multline}
  I_2^{++}  
    \sim e^{N (F_1(q_1)- F_2(q_2))}G_{i_2, i_1}(q_1, q_2)_{j_1+1,j_2+1} \frac{1}{z_2(z_2-z_1)}  \frac{1}{\sqrt{-F_1''(z_1)}} \frac{1}{\sqrt{F_2''(z_2)}} \\
    \times \frac{1}{N (2\pi\i)^2}\int_{-\infty}^{\infty} \int_{-\infty}^{\infty}  
    e^{ -(x')^2/2 -  x^2/2}  \d x \d x'\\
    = -\frac{1}{2 \pi N}e^{N (F_1(q_1)- F_2(q_2))}G_{i_2, i_1}(q_1, q_2)_{j_1+1,j_2+1} \frac{1}{z_2(z_2-z_1)}  \frac{1}{\sqrt{-F_1''(z_1)}\sqrt{F_2''(z_2)}}  \label{eqn:double_int_final} 
\end{multline}
where~$\sim$ denotes equality up to the error~\eqref{eqn:biggest_error}. This gives the first of the four terms in the statement of the lemma.

We emphasize that all bounds on error terms above hold uniformly in~$(\xi_{j, N}, \eta_{j, N})$, so long as both points are in region~\eqref{item:first}; we recall that this means the distance~$|\eta_{j, N} - \eta_{fb}(\xi_{j,N})|$ by which one has to perturb~$\eta_{j, N}$ to reach the neareast point on the arctic curve is~$\geq c_1 N^{-\frac{1}{3}}$, for some~$c_1 > 0$. 

The other three terms in parentheses in the theorem statement are obtained in a similar way to the above.

\paragraph{Bounding the single integral term:}
Finally, we must bound the contribution from the remaining single integral term described in the second step above. First, we note that in the case that~$\gamma_{\pm}$ has multiple components, then each component starts and ends at an intersection point say~$\zeta$ and~$\zeta'$. If we deform the contour so that in the amoeba it is a straight line between these two points, then the quantity~$\Re[F_1(q) - F_2(q)]$ is monotonic along the contour (this is a linear function in the coordinates of the amoeba, as we indicate below). Since we will ultimately bound~$e^{N (\Re[F_1(\zeta)-F_1(q_1)] - \Re[F_2(\zeta)-F_2(q_2)])}$ for an arbitrary intersection point~$\zeta$, we may extend each curve segment to the outer boundary, and our bound for the case of one intersection point then implies a similar bound for when there are multiple.

Now without loss of generality we assume there is just one intersection point. The integrand of the single integral~\eqref{eq:rough_single_integral} contains the factor~$e^{N (F_1(q)- F_2(q))}$. First we deform the contour~$\gamma_{\pm} $ so that the image of the part in~$\mathcal{R}_0$ in the amoeba is a straight line, chosen such that~$\Re[F_1(q)- F_2(q)] = (\frac{x_2}{N} - \frac{x_1}{N}) \log|w| - (\frac{y_2}{N} - \frac{y_1}{N}) \log|z|$ is strictly decreasing. Such a deformation is always possible: If~$x_2 >> x_1$, in which case we start with~$\gamma_{-}$, then we choose the contour starting at the intersection point~$\zeta$ and so that its image in the amoeba goes vertically along~$\log|z|=\log|z(\zeta)|$ and so that~$\log|w|\to -\infty$. If we have to cross a tentacle corresponding to~$w = 0$, it can be done without picking up a residue since~$x_2 >> x_1$ implies the integrand is regular there. If instead~$x_2 << x_1$, then the contour is~$\gamma_{+}$, and we may similarly pick a contour which moves so that~$\log|w|\to+\infty$. In the case~$x_1 \approx x_2$ a similar explicit construction shows that we can deform the curve to a steep descent contour, which decreases as we move away from~$\zeta$.

Since the size of the exponential along the new contour is maximized at~$\zeta$, it suffices to bound the norm of the integrand at~$\zeta$. Since~$|(\xi_{1,N},\eta_{1,N})-(\xi_{2, N},\eta_{2,N})| > N^{-\frac{1}{16}}$, the two critical points are separated, and in particular both critical points~$q_1$ and~$q_2$ cannot be too close to~$\zeta$. In fact, for some~$c > 0$ which can be chosen uniformly for both points varying over region~\eqref{item:first}, at least one of the steepest descent contours must be traversed for distance at least~$ c  N^{-\frac{1}{16}}$ (in the local coordinate~$z$ or~$w$) in order to reach~$\zeta$. As a result, we obtain a bound for the integrand evaluated at~$\zeta$ of the form 
\begin{align*}
    e^{N (\Re(F_1(\zeta) - F_2(\zeta))} \leq e^{N (\Re(F_1(q_1) - F_2(q_2))} O(e^{- N \mu_1-N \mu_2})
\end{align*}
with~$\mu_1 \sim  |F_1''(q_1)| N^{-\frac{1}{8}}$ and~$\mu_2 \sim  |F_2''(q_2)| N^{-\frac{1}{8}}$, and at least one of~$\mu_i$ satisfies~$\mu_i \geq c' N^{-1/6 - 1/8}$, for a constant~$c' > 0$. This is because~$|F_1''(q_1)|$,~$|F_2''(q_2)| \geq c_1' N^{-1/6}$ (where~$c_1'$ depends on~$c_1$ in the definition of regime~\eqref{item:first}), and because~$\zeta$ is on the steepest descent and ascent contour of~$F_1$ and~$F_2$, respectively, with local quadratic behavior near~$q_1$ and~$q_2$, respectively.
\end{proof}

\begin{figure}
\centering
    \includegraphics[scale=.3]{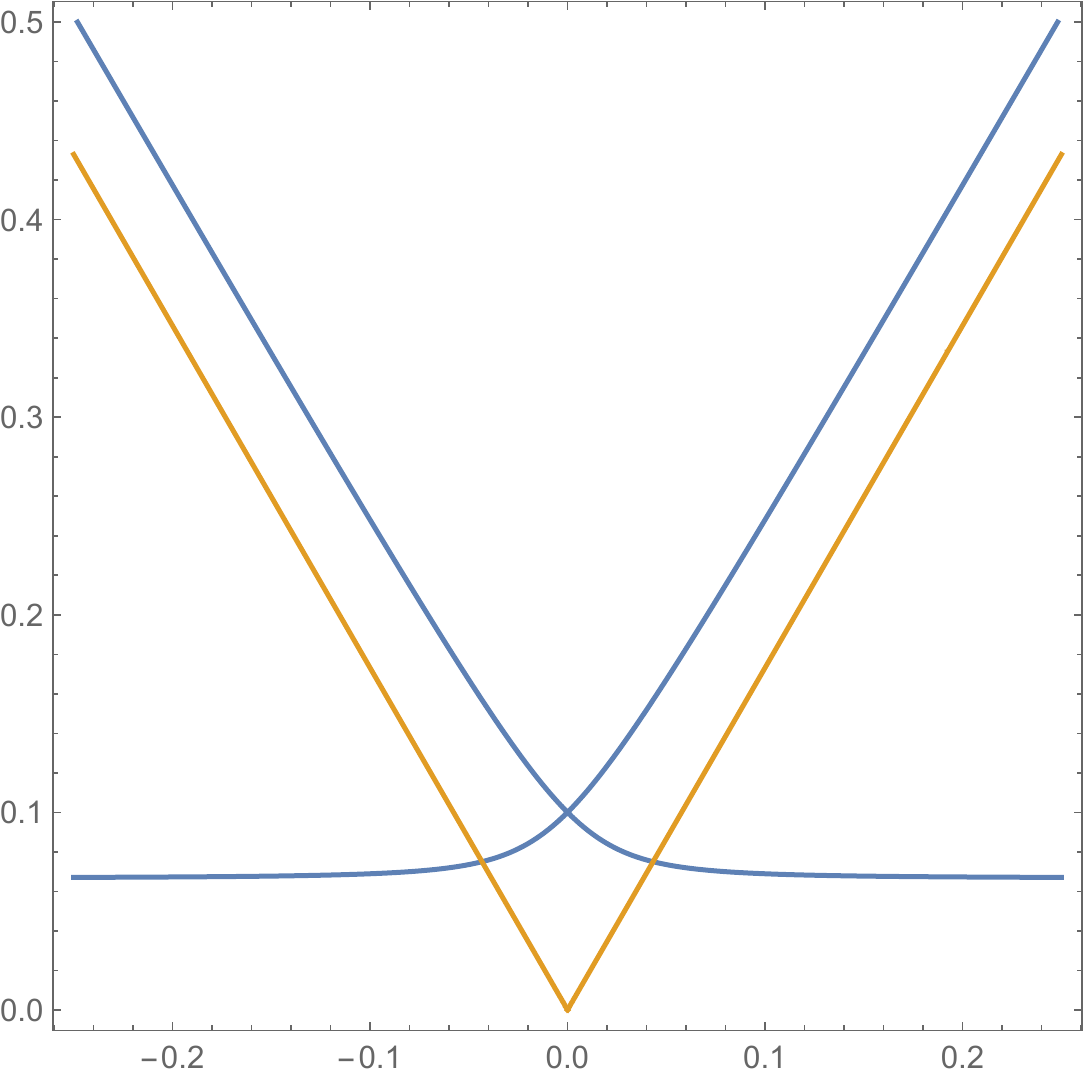}
    \caption{Here we show the steepest descent and ascent contours (level lines of the imaginary part) of the function~$\i \epsilon (z-\i \epsilon)^2 + (z-\i \epsilon)^3$, for~$\epsilon = 0.1$, and for comparison we have shown the directions~$e^{\i \pi/3}$ and~$e^{\i 2 \pi/3}$. This is similar to the local steepest descent and ascent contours of~$F_1$ in region~\eqref{item:second}.}
    \label{fig:barelyliquid}
\end{figure}

\subsection{Steepest descent near the edge}

\begin{proof}[Proof of Lemma~\ref{lem:nearedge_steepest}]
We complete the proof using similar steps as in the proof of Lemma~\ref{lem:edge_steepest}. We will assume that~$(\xi_{1, N}, \eta_{1, N})$, the rescaled coordinates of the black vertex, is the point which is near the edge. In the case that~$(\xi_{2, N}, \eta_{2, N})$, the rescaled coordinates of the white vertex, is at the edge and~$(\xi_{1, N}, \eta_{1, N})$ is in the bulk, or in the case that they are both at the edge, the argument is similar. 

\paragraph{Deforming the contours:} As before, we deform the contours~$\Gamma_s$ and~$\Gamma_l$ to new ones~$\gamma_s$ and~$\gamma_l$, respectively, which as before, are steep descent contours, and are preserved (up to orientation) by complex conjugation. We will therefore only focus on the parts of the contours in~$\mathcal{R}_0$, which we may view as subsets of the amoeba, and locally as subsets of~$\mathbb{C}$ in the upper or lower half plane. In this case, however, we will construct the steep descent contour for~$(\xi_{1, N}, \eta_{1, N})$ ``by hand'' instead of globally choosing the one of steepest descent. The reason for doing so is that the~$O(\delta)$ neighborhood of the critical point contains the boundary of~$\mathcal R_0$. After adjusting the curve close to the critical point~$q_1$ to control the leading order term, we end up on the boundary of~$\mathcal R_0$, which is a (possibly different) curve of steepest descent that we will follow.

 We now work in the local coordinate~$z$, and let, as before,~$z_1=z(q_1)$ where~$q_1$ is the critical point of~$F_1$. We assume without loss of generality that for points in~$\mathcal{R}_0$,~$z \in \mathbb{H}$, which we can achieve by replacing~$z$ by~$-z$ if necessary. If~$(\xi_{1, N}, \eta_{1, N})$ was exactly at the arctic curve, then we would have~$F_1''(z_1) = 0$. In our case,~$(\xi_{1,N}, \eta_{1,N})$ is near the arctic curve, so for~$z_0 = \Re[z_1] \in \mathbb{R}$, we have that~$z_1-z_0 \in \i \mathbb{R}$ is small (recall Lemma~\ref{lem:sqrt_singularity}) and purely imaginary. This implies
\begin{align} \label{eqn:second_deriv_equivalence}
    F_1''(z_1) = \i \Im(z_1) F_1'''(z_{0}) + O(|\Im(z_1)|^2).
\end{align}
Since~$F_1'''(z_{0}) \in \mathbb{R}$ (which follows from the fact that~$F_1$ is real valued on the real part of~$\mathcal{R}$), ~\eqref{eqn:second_deriv_equivalence} implies that locally the direction of the steepest descent contour, which is defined up to a sign by~$\arg(\sqrt{-F_1''(z_1)})$, is close to either~$\pm e^{\i \pi/4}$ or~$\pm e^{\i 3 \pi/4}$. It is the former if~$F_1'''(z_0) > 0$, and otherwise it is the latter. We call the unit length complex number in this direction~$\hat{\theta}$. We choose the overall sign in~$\hat{\theta}$ so that it points towards the real axis.

\begin{figure}
    \centering
    \includegraphics[width=0.55\linewidth]{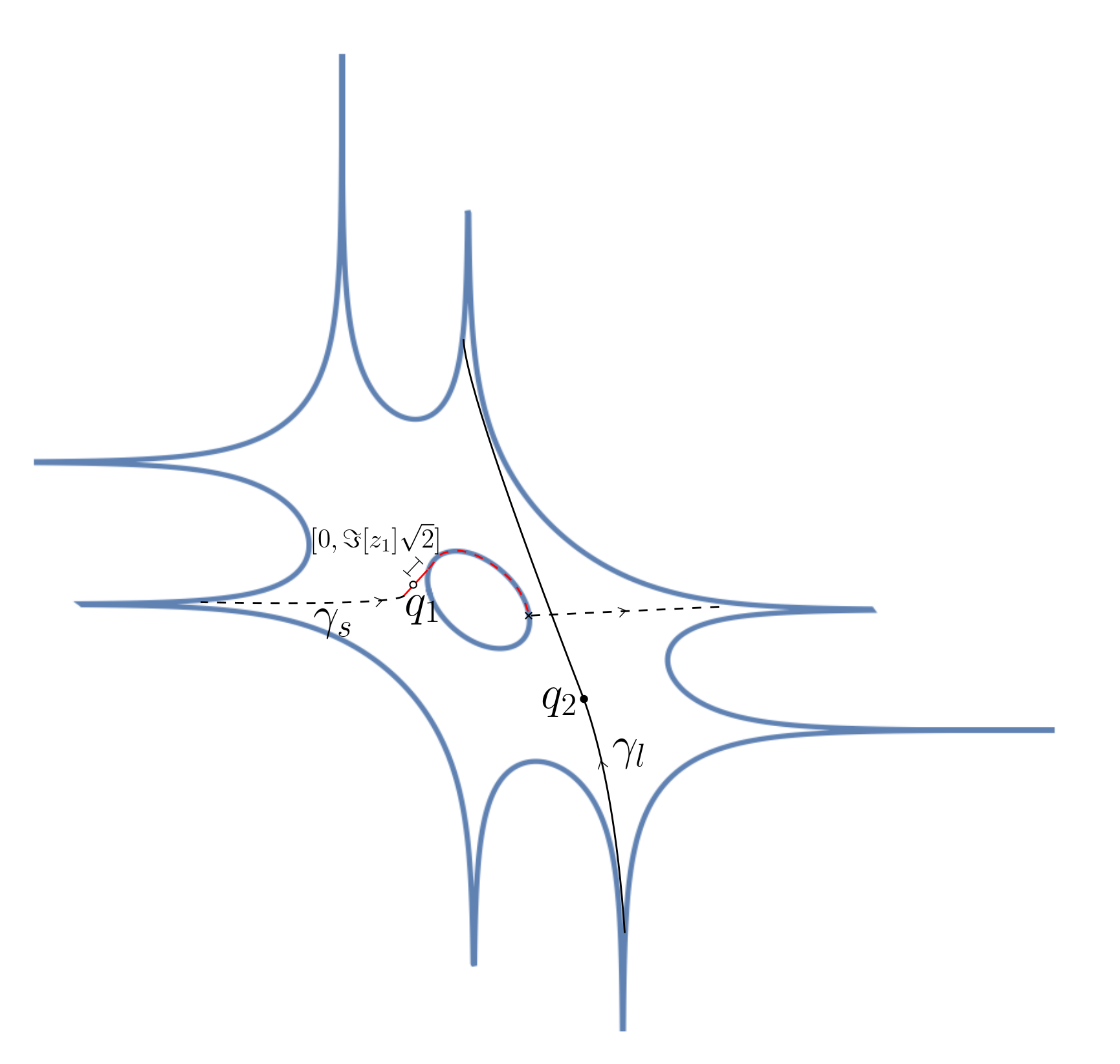}
    \caption{The steep descent contour that we choose near the edge may look like the above. We may need to continue it along the real part of~$\mathcal{R}$ until we hit a new critical point, and then follow a new steepest descent contour from there to an angle. This is Case B discussed in the proof.}
    \label{fig:nearedge}
\end{figure}

In this case the steepest descent contours will start off approximately in the direction~$\pm \hat{\theta}$, but if~$\Im(z_1)$ is small enough, then inside of a~$\delta$-neighborhood it will change course and start to follow the direction~$e^{ \pi \i/3} $ or~$e^{2 \pi \i /3}$ in one direction, and will start to move parallel to the real axis in the other direction; see Figure~\ref{fig:barelyliquid}. We ultimately will not use the steepest descent contours exactly, however; we will modify the part moving parallel along the real axis.

We choose a contour which locally matches the straight line~$\gamma_{\text{loc}} = z_1 \pm x \hat{\theta}$, where~$x \in [-\delta,  |\Im(z_1)| \sqrt{2}]$; recall~$\hat{\theta}$ defined above points towards the real axis, so that the positive~$x$ direction corresponds to moving towards the real axis. For the part of~$\gamma_s$ starting at~$q_1$ and moving away from the real axis, we initially take the steepest descent contour, and then deform it to locally match the part of the straight segment~$\gamma_{\text{loc}}$ corresponding to~$x \in [-\delta, 0]$. For the part of the contour leaving~$q_1$ moving towards the real axis, we take the straight segment and we cut off the contour after~$x = |\Im(z_1)| \sqrt{2}$. The contour~$\gamma_{\text{loc}}$ intersects the real axis exactly at the endpoint. See Figure~\ref{fig:nearedge}; the solid red segment represents~$\gamma_{\text{loc}}$, and the part of it corresponding to~$x \in [0, |\Im(z_1)| \sqrt{2}]$ is indicated there. At the point on the real axis,~$\Re F_1$ decreases in one direction and increases in the other. In fact, the real line is a curve of steepest descent of~$F_1$. We continue the curve~$\gamma_s$ by following this curve of steepest descent. 

In the proof of Lemma~\ref{lem:steepest}, we used that~$\gamma_s\cap \mathcal R_0$ begins close to an angle~$p_{0,j}$ and ends close to an angle~$p_{\infty,i}$, which means that~$\Gamma_s$ can be deformed to~$\gamma_s$. This fact was proven in~\cite{BB23}. Here we cannot rely directly on~\cite{BB23}, however, a similar argument still applies and we claim that~$\Gamma_s$ can still be deformed to~$\gamma_s$. Indeed, the level lines~$\{\Re F_1(q)=\Re F_1(q_1)\}$ defines two connected sets, denoted by~$C_1$ and~$C_2$ in~\cite{BB23}, in which~$\{\Re F_1(q)<\Re F_1(q_1)\}$. The set~$C_1$ only contains angles of the type~$p_{0,j}$ and~$C_2$ only angles of the type~$p_{\infty,j}$ and the intersection of their closure only contain the critical point;~$\overline{C_1}\cap \overline{C_2}=\{q_1\}$. Equations~\eqref{eqn:boundxneg} and~\eqref{eqn:boundxpos} below, show that~$\gamma_{\text{loc}}$ is contained in~$\overline{C_1}\cup \overline{C_2}$ and, hence, so is all of~$\gamma_s$. This proves the claim.

\paragraph{Keeping track of the residues picked up during the deformation:} During the deformation to obtain such contours, the only residue picked up, as in the proof of Lemma~\ref{lem:steepest}, is the one at~$z_1 = z_2$. Furthermore, this combines with the single integral, so that overall we get another single integral over a contour beginning at the intersection point~$\zeta$ and ending at a point in~$A_0$, as in that proof.

\paragraph{Bounding errors, and computation of the leading order behavior from neighborhoods of critical points:}
We consider the part of the new~$q'$ integration contour~$\gamma_s$ in which~$z'$ is in a small neighborhood of a critical point. Recall that in this case we have changed the local contour to the straight line~$\gamma_{\text{loc}}$, parameterized by~$ z' = z_1 + x \hat{\theta}$, where~$x \in [-\delta,  |\Im(z_1)| \sqrt{2}]$, and~$\hat{\theta}$ equals ~$\pm e^{\i \pi/4}$ or~$\pm e^{\i 3 \pi/4}$; the positive~$x$ direction points towards the real axis. Now the contour~$\gamma_{\text{loc}}$ intersects the real axis exactly at its endpoint, when~$x = \sqrt{2} |\Im(z_1)|$. Finally, we again choose~$\delta = N^{-\frac{1}{4}}$.

We know that
\begin{equation}\label{eqn:II_local_equivalence}
    F_1(z') - F_1(z_1) = \frac{1}{2}  F''(z_1) \hat{\theta}^2 x^2 + \frac{1}{6} (\hat{\theta} x)^3 F'''(z_1) +  O(x^4),
\end{equation}
and that~$\Re[\hat{\theta}^3 F_1'''(z_1)] > 0$ for~$N$ large enough, since by construction~$\Re[\hat{\theta}^3 F_1'''(z_0)] > 0$. Thus, if~$x \leq 0$, we have
\begin{equation}\label{eqn:boundxneg}
    \Re[\frac{1}{2}F_1''(z_1) \hat{\theta}^2 x^2 + \frac{1}{6} F_1'''(z_1) (\hat{\theta} x)^3 + O(x^4)]\leq -\frac{1}{2}|F_1''(z_1)| x^2  \leq -\frac{1}{6}|F_1''(z_1)| x^2
\end{equation}
because the cubic term has negative real part and the quartic term is subleading for the chosen value of~$\delta$. On the other hand, for~$x \geq 0$, we have
\begin{equation}\label{eqn:boundxpos}
    \Re \left[\frac{1}{2}F_1''(z_1) (\hat{\theta} x)^2 + \frac{1}{6} F_1'''(z_1) (\hat{\theta} x)^3 + O(x^4) \right]  \leq -\frac{1}{6}|F_1''(z_1)| x^2
\end{equation}
because for~$x \leq \sqrt{2} |\Im(z_1)|$, we may use~\eqref{eqn:second_deriv_equivalence} to see that the quadratic term dominates the cubic one. 

As in the previous proof, the leading contribution comes from the part of~$\gamma_s$ laying in a neighborhood of~$q_1$. Let us estimate the error terms coming from the parts of~$\gamma_s$ outside this neighborhood. The part connecting with~$\gamma_{\text{loc}}$ at~$x=-\delta$ may be estimated by~$F_1(q_\delta)-F_1(q_1)$, where~$q_\delta$ is the point corresponding to~$x=-\delta$. Using~\eqref{eqn:II_local_equivalence} and~\eqref{eqn:boundxneg}, we get the error term
\begin{equation}
e^{N(\Re[F_1(q_1)]- \Re[F_2(q_2)])} O(e^{-\mu N}), \quad \text{where} \quad \mu \gtrsim |F_1''(z_1)| x^2=O\left(N^{-\frac{5}{6}}\right),
\end{equation}
leading to an exponentially small part, compared with the bound in the statement. For the other part of~$\gamma_s$, going along a steepest descent curve starting on the boundary of~$\mathcal R_0$, there are two possibilities: The curve goes directly to an angle, which then is the end of the contour (as a curve in~$\mathcal R_0$), or it goes to a critical point on the boundary of~$\mathcal R_0$. We note that the part of the integral along the real part of the amoeba cancels out when we add the integral along the conjugate (with the other orientation). This means that if the curve goes directly to an angle, there is nothing to do. If the curve hits a critical point, call it~$q_c'$, it clearly suffices to get an upper bound of the form~$N \Re(F_1(q_c') - F_1(q_1)) \leq - \epsilon'' N$. In this case, we use the fact that we are bounded away from any cusp to observe that for some~$\epsilon' > 0$, we have in local coordinates~$|z_1 - z_c'| > \epsilon'$. Thus, recalling~$z_0 =\Re[z_1]$, for all~$N$ large enough we have~$F_1(q_c') - F_1(z_0) < -\epsilon''$ for some~$\epsilon'' > 0$, which implies the desired bound, giving the (after factoring out~$e^{N \Re[F_1(q_1)]}$) exponential decay.

%
%

What remains, is to compute a bound on the leading order contribution near the critical point. As in Lemma~\ref{lem:steepest}, we approximate the bounded factors in the integrand by their values at~$q_2$ and~$q_1$, respectively. We will focus only on the~$q'$ integral because the~$q$ integral can either be computed as in Lemma~\ref{lem:steepest} if~$(\xi_{2,N}, \eta_{2,N})$ is in the bulk, or bounded in a similar way to the~$q'$ integral if it is near the edge.

We use~\eqref{eqn:boundxneg} and~\eqref{eqn:boundxpos} to bound the relevant part of the~$q'$ integrand, and then doing the Gaussian integral leads to
\begin{equation}
 C_1 \left| \int_{-\delta }^{\sqrt{2} |\Im(z_1)|} e^{-N \frac{|F_1''(q_1)|}{6} x^2} d x \right| \leq C_2 \frac{1}{N^{\frac{1}{2}} \sqrt{|F_1''(q_1)|} } .
\end{equation}
This gives the main contribution, which leads to the expression in the statement of the lemma.

\paragraph{Bounding the single integral term:}
Finally, the remaining single integral can be bounded as before, in the bulk case Lemma~\ref{lem:steepest}. The only case not handled by previous arguments is when both points are nearby each other and near the same facet, and the intersection point happens exactly at the corresponding part of~$\partial \mathcal{R}_0$. However, in this case, assuming~$\zeta$ is at least~$\frac{1}{2} N^{-\frac{1}{16}}$ away from~$q_1$ we may upper bound the difference~$\Re[F_1(\zeta) - F_1(q_1)]$ using a Taylor expansion, and we see that the third order term will dominate and be negative, and have order~$N^{-\frac{3}{16}}$, which is sufficient. Thus, after factoring out~$e^{N(\Re[F_1(q_1)]- \Re[F_2(q_2)])}$, again the single intergral is (exponentially in~$N^{\alpha}$ for some~$\alpha > 0$) subleading, and thus and can be ignored since we only need a bound.
\end{proof}

\subsection{Steepest descent at the edge}
Now we will move on to the proof of Lemma~\ref{lem:edge_steepest}, so for that we recall some notation. We have assumed without loss of generality that~$(\xi_{1,N}, \eta_{1,N})$ (and possibly but not necessarily~$(\xi_{2,N}, \eta_{2,N})$ as well) is at the edge, that is, in region~\eqref{item:third}. We defined~$\eta_{1,fb}$ as the nearby value of~$\eta_{1,N}$ such that~$(\xi_{1,N}, \eta_{1,fb}) \in \partial \mathcal{F}_R $ is exactly on the arctic curve. With this assumption, if the frozen region is ``above''~$(\xi_{1,N},\eta_{1,fb})$, then being in region~\eqref{item:third} means that we have~$N^{-\frac{2}{3}+\frac{1}{100}} \geq  \eta_{1, N} - \eta_{1,f b} \geq -c_2 N^{-\frac{2}{3}}$. In addition, define~$\tilde{F}_1(q) \coloneqq F(q; \xi_{1, N}, \eta_{1, f b})$, and let~$u_1=\frac{\ell}{2}N^{\frac{2}{3}}(\eta_{1,N}-\eta_{1,fb})$ so that
\begin{equation}
   N F_1(q) = N \tilde{F}_1(q) +  N^{\frac{1}{3}} u_1 \log z .
\end{equation}

\begin{proof}[Proof of Lemma~\ref{lem:edge_steepest}]
We proceed in steps once again.

\begin{figure}
    \centering
    \includegraphics[scale=.22]{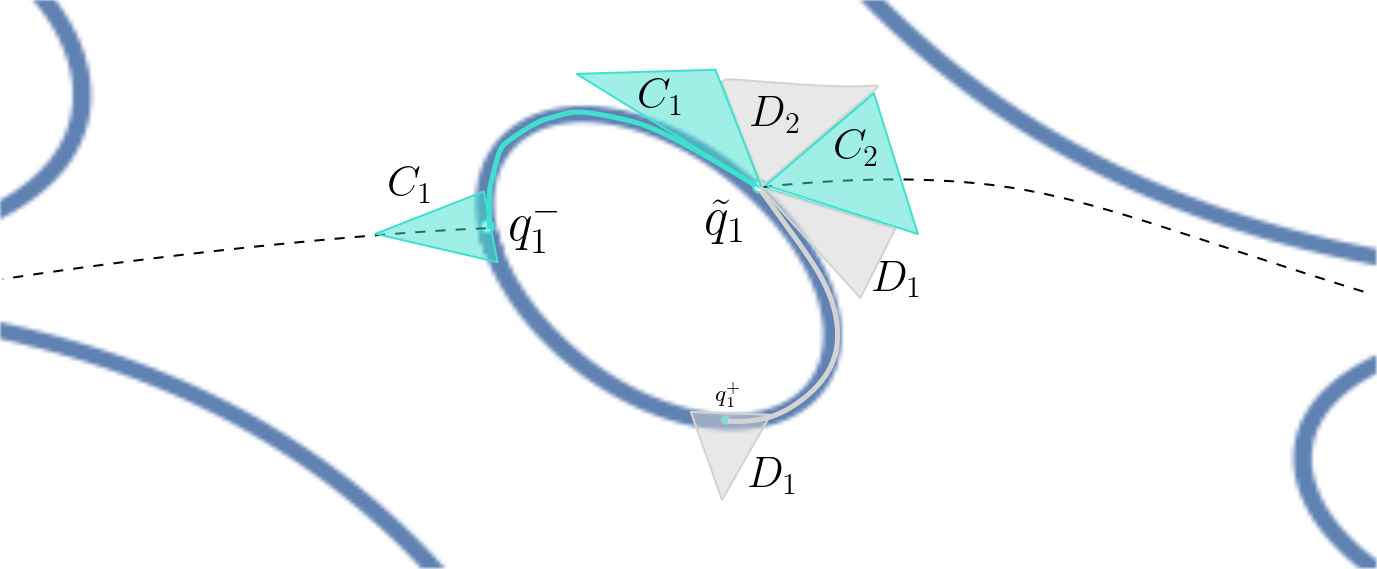}
    \caption{A schematic of the steepest descent contour~$\gamma_s$, drawn in the amoeba, in one special case when~$\tilde q_1$ is in a compact oval. We have also used the notation~$C_1$,~$C_2 $ (shaded turquoise) to denote connected components of~$\{\Re[\tilde{F}_1(q)] < \Re[\tilde{F}_1(\tilde{q}_1)]\} $ and~$D_1$,~$D_2$ (shaded gray) to denote connected components of~$\{\Re[\tilde{F}_1(q)] > \Re[\tilde{F}_1(\tilde{q}_1] \}$. In this example, in terms of a local coordinate in~$\mathbb{H}$ which parameterizes the compact oval clockwise, we have~$\tilde F_1'''(\tilde{q}_1) > 0$. In particular, the gray arc along the compact oval is in~$D_1$ and the turquoise arc is in~$D_2$. Here~$q_1^+$ and~$q_1^-$ denote the other two critical points on the compact oval, and are a local max and a local min of~$\tilde{F}_1$, respectively.}
    \label{fig:edge_deformation}
\end{figure}

\paragraph{Deforming the contours:} 
Now we deform the contour~$\Gamma_s$ (for~$q'$, which corresponds to~$F_1$) so that it passes through the critical point~$\tilde{q}_1$ of~$\tilde{F}_1$. Since~$d \tilde{F}_1$ has a double~$0$ at~$\tilde{q}_1$, in the local coordinate~$z$, near~$\tilde{q}_1$, the level lines of the real part of~$\tilde{F}_1$ split the upper half plane into four connected components: There are two connected components~$C_1$ and~$C_2$ of the region~$\{\Re[\tilde{F}_1(q)] < \Re[\tilde{F}_1(\tilde{q}_1)]\}$, and two connected components~$D_1$ and~$D_2$ of~$\{\Re[\tilde{F}_1(q)] > \Re[\tilde{F}_1(\tilde{q}_1)]\}$, as depicted in Figure~\ref{fig:edge_deformation}. A direct adaptation of the arguments in~\cite[Lemma 6.5]{BB23} allow us to deduce that one of~$C_1$ or~$C_2$ must contain an angle~$p_{0, j}$, and the other must contain an angle~$p_{\infty, j}$, and neither one can contain both types of angles. Say~$C_1$ contains a~$p_{0, j}$ angle and~$C_2$ contains a~$p_{\infty, j}$ angle. We can define~$\gamma_s$, similarly do the previous proofs, by taking the part~$\gamma_s\cap \mathcal R_0$ as a steepest descent curve in~$\overline{C_1}\cup \overline{C_2}$ going from an angel~$p_{0,j}$ and end at an angel~$p_{\infty,i}$ and that goes through~$\tilde q_1$. At~$\tilde q_1$, one part of~$\gamma_s\cap \mathcal R_0$ will leave the boundary of~$\mathcal R_0$, while the other part will go along the boundary of~$\mathcal R_0$ (similarly to the curve in the proof of Lemma~\ref{lem:nearedge_steepest}). See Figure~\ref{fig:edge_deformation} for an illustration in the case when~$\tilde q_1$ is on a compact oval.

Similarly to the previous two lemmas we then locally deform the contour~$\gamma_s$ to be straight segments in a~$\delta$-neighborhood of the critical points. For~$\gamma_s$ this means we choose~$\gamma_{\text{loc}} = \tilde z_1 + \tilde x \hat \theta $, where~$\tilde z_1=z(\tilde q_1)$ and~$\hat \theta$ is unit length and has direction~$\pm e^{\i \pi/3}$ or~$\pm e^{\i 2\pi/3}$.

\paragraph{Keeping track of the residues picked up during the deformation:}
 We again only pick up a residue from the pole at~$q' = q$. We end up with a single integral moving from an intersection point of the contours~$\gamma_s$,~$\gamma_l$ (or potentially one made up of several segments between such intersection points), and ending at either the component~$A_{0,1}$ or~$A_{0,k+\ell+1}$ of~$\partial \mathcal{R}_0$.

\paragraph{Bounding errors, and computation of the leading order behavior from neighborhoods of critical points:}
Recall that we assume, without loss of generality, that the nearby facet is ``above''~$(\xi_{1,N},\eta_{1,fb})$ and that near~$\tilde q_1$, the positive part~$\mathcal{R}_0$ corresponds to~$\Im[z] > 0$ in the local~$z$ coordinate, so that with this coordinate we may restrict our attention to the upper half plane. 

First, suppose that~$u_1 \in [-l, l]$ for some finite~$l > 0$. Now in this case, the~$q'$-dependent part of the integrand is of the form
\begin{align*}
   e^{N \tilde{F}_1(q')+ N^{\frac{1}{3}} u_1 \log z'} G_{i_2,i_1}(q',q)_{j_1+1,j_2+1} \frac{1}{z (z-z')}.
\end{align*}
The part of the~$q'$ integral away from an~$O(\delta)$ neighborhood contributes
\begin{equation}\label{eqn:awaydelta}
e^{N\Re[F_1(\tilde q_1)]}O(e^{- N \delta^3 |\tilde{F}_1'''(\tilde{q}_1)|}),
\end{equation}
and, as before, we choose~$\delta = N^{-\frac{1}{4}}$, so it decay, compared with the bound in the statement, exponentially in a power of~$N$. We similarly may throw away the parts of the integral where~$q$ is outside of a small neighborhood of~$q_2$, or of~$\tilde{q}_2$, depending on the regime of~$(\xi_{2,N},\eta_{2,N})$.

If we expand~$\frac{G_{i_2,i_1}(q',q)_{j_1+1,j_2+1}}{z (z-z')}$ around~$z' = \tilde{z}_1$, and around~$z = z_2$ (or~$z = \tilde{z}_2$, depending on the case), then we will get, after substituting~$z' =\tilde z_1+ N^{-\frac{1}{3}} v$ and recalling that locally in the upper half plane the contour is a straight line in the direction~$\hat{\theta}$ which is equal to~$e^{\i \frac{\pi}{3}}$ if~$\tilde{F}_1'''(\tilde{z}_1)>0$ and~$e^{2 \i \frac{\pi}{3}}$ if~$\tilde{F}_1'''(\tilde{z}_1)<0$:
\begin{multline}\label{eqn:edge_expr}
e^{N F_1(\tilde{q}_1)}  G_{i_2,i_1}(\tilde{q}_1,q_2)_{j_1+1,j_2+1} \frac{1}{z_2 (z_2-\tilde{z}_1) N^{\frac{1}{3}}}\int_{-\overline{\hat{\theta}} N^{\frac{1}{3}}\delta}^{\hat{\theta} N^{\frac{1}{3}} \delta}  e^{ \tilde{F}_1'''(\tilde{z}_1) v^3 + u_1 \frac{1}{\tilde{z}_1} v + O\left(N^{-\frac{1}{3}} (|v|^4 + |u_1| |v|^2)\right)} d v  \\
\times \int (\text{function of $z$}) dz.
\end{multline}
For the curve in the above integral, recall that the integral along the boundary of~$\mathcal R_0$ cancels out when we add the integral along the conjugate with the opposite orientation. The error accumulated on the first line from discarding the error in the Taylor expansion of~$\frac{1}{z-z'}$ decays as~$e^{N \Re [F_1(\tilde{q}_1)]} O(N^{-\frac{1}{3}-\frac{1}{8}})$. In addition, with an argument similar to the one in~\cite[Lemma 6.1]{BF14}, we can also throw out the~$O\left(N^{-\frac{1}{3}} (|v|^4 + |u_1| |v|^2)\right)$ error in the exponent, at the cost of an~$O(N^{-\frac{1}{3}})$ error for the value of the integral in the first line. Since the contour is locally a straight segment in the direction of steepest descent, the remaining integral converges to a smooth function of~$u_1$,~$\tilde{F}_1'''(\tilde{z}_1)$, and~$\tilde{z}_1$, which can be written in terms of the Airy function by the computations in~\cite[Lemma 6.1]{BF14}, see especially (6.5) there. Because this function is smooth, and we are free to choose the constant~$C_3$ in~\eqref{eqn:one_edge_bound} and~\eqref{eqn:both_edge_bound}, this is sufficient for the regime~$u_1 \in [-l,l]$. 

The local contribution from the~$q$ integral can be computed in a similar way or as in the proof of Lemma~\ref{lem:steepest} depending on which regime~$(\xi_{2, N}, \eta_{2, N})$ is in, and this will give the bound~\eqref{eqn:one_edge_bound} or~\eqref{eqn:both_edge_bound}.

Next, we assume~$u_1 \in [l, N^{\frac{1}{100}}]$. We may go through each bound in the argument above, and verify that each will remain valid for~$u_1$ up to a very slowly growing power of~$N$. In particular,~\eqref{eqn:awaydelta} continues to be the error from the parts of the contour away from~$q_1$ and will decay at the same rate if~$\delta = N^{-\frac{1}{4}}$, so long as~$u_1$ does not grow faster than a small power of~$N$. This follows from the computation 
\begin{equation*}
-N \delta^3 |\tilde{F}_1'''(\tilde{z}_1)| +  N^{\frac{1}{3}}u_1 \frac{1}{\tilde{z}_1} \delta \leq - N^{\frac{1}{4}} |\tilde{F}_1'''(\tilde{z}_1)| + C N^{\frac{1}{100}} N^{\frac{1}{12}} \lesssim - N^{\frac{1}{4}}.
\end{equation*}
Furthermore, we claim that the linear term in the exponent in the first line in~\eqref{eqn:edge_expr} actually has negative real part and thus helps the convergence, so that the error bounds required to ultimately arrive at the Airy function still remain small uniformly, even for~$u_1$ going all the way up to~$N^{\frac{1}{100}}$. In addition, the integral in the first line of~\eqref{eqn:edge_expr} is still finite uniformly for~$u_1 \in [-l,N^{\frac{1}{100}}]$; in fact it is given in terms of an Airy function as before. Finally, since the Airy function decays for large positive values of its argument, the bounds~\eqref{eqn:one_edge_bound},~\eqref{eqn:both_edge_bound} in the statement of the lemma remain valid.

It remains to show the claim that the linear term has negative real part. To show the claim, it suffices, due to the definition of the curve and since~$u_1\geq l>0$, to show if~$\tilde{F}_1'''(\tilde{z}_1)>0$, then~$\tilde{z}_1 < 0$, and if~$\tilde F_1'''(\tilde z_1) < 0$, then~$\tilde z_1 > 0$. By assumption of $u_1$,~$ (\xi_{1,N},\eta_{1,N})$ is in the gaseous or frozen facet, which means that the critical points of $F_1$ are real. The proof of Lemma~\ref{lem:sqrt_singularity}, see \eqref{eqn:zeros_close_edge}, shows that the critical point of $F_1$ approximates $\tilde z_1\pm \sqrt{-\frac{\epsilon \ell}{\tilde F_1'''(\tilde z_1)\tilde z_1}}$, which is real if and only if $\tilde{F}_1'''(\tilde{z}_1)\tilde{z}_1<0$.


\paragraph{Bounding the single integral term:} In this case, we may bound the single integral term by the same arguments as in the previous two lemmas.
\end{proof}

\subsection{Steepest descent and moment bounds in the facet}
The goal of this section is to prove Lemma~\ref{lem:facet_bound}. The main challenge is to obtain the bound if the point lies in a gas region. To deal with that case, we first state and prove two lemmas bounding the inverse Kasteleyn matrix as at least one point is in regime~\eqref{item:fourth} in a gaseous facet. 

The first lemma deals with the setting when at least one point is in regime~\eqref{item:fourth} in a gaseous facet, and the other point is in any regime except that it cannot be in the same gaseous facet. As before, both points are bounded away from the cusps, tangency points, and points in the arctic curve with a vertical tangent. Because the statement is similar if we swap the black and white vertices, it suffices to consider the case that only~$(\xi_{1,N}, \eta_{1,N})$ is in region~\eqref{item:fourth}, and~$(\xi_{2,N}, \eta_{2,N})$ may or may not also be. We denote the vertices with the rescaled coordinates~$(\xi_{j,N}, \eta_{j,N})$ by~$(\mathrm b_j,\mathrm w_j)$,~$j=1,2$. 

\begin{lem}[Steepest descent with at least one point in a gaseous facet; a crude bound]\label{lem:fourth_dif_facets}
    We have the following asymptotics when~$(\xi_{1,N}, \eta_{1,N})$ is in regime~\eqref{item:fourth} inside of a gaseous facet, and the other point is in any regime, except if it is in regime~\eqref{item:fourth} we require it is not in the same facet as~$(\xi_{1,N}, \eta_{1,N})$:
    \begin{equation}\label{eqn:fourth_different}
   |K^{-1}(\mathrm b_2,\mathrm w_1)|\leq  C_4  e^{N (\Re[F_1(q_1)-\Re[F_2(q_2)])]}.
   \end{equation}
    
\end{lem}
\begin{remark}
If one wanted, one could split this lemma up further into cases and compute precise leading order asymptotics in each case. However, we only need a bound, and~\eqref{eqn:fourth_different} is sufficient for our purposes.
\end{remark}

\begin{proof}[Proof of Lemma~\ref{lem:fourth_dif_facets}]
Now we repeat the four steps. This is very similar to the previous cases. One thing we note in this case is  the following: If one point is in regime~\eqref{item:fourth}, then~$q_1$ and~$q_2$, the critical points giving the main contributions, may be close to each other even if~$(\xi_{1,N}, \eta_{1,N})$ and~$(\xi_{2,N}, \eta_{2,N})$ are macroscopically far apart. 







\paragraph{Deforming the contours:} Now, by the results of~\cite{BB23}, there will be four critical points~$q_{1}^{1,+}$,~$q_{1}^{2,+}$, and~$q_{1}^{1,-}$,~$q_{1}^{2,-}$ of~$F_1$ which are in the relevant compact oval of~$\partial \mathcal{R}_0$. We have used the notation~$q_{1}^{1,+}$,~$q_{1}^{2,+}$ to denote the two local maxima of~$F_1$ along the compact oval, and~$q_{1}^{1,-}$,~$q_{1}^{2,-}$ denote the two local minima. Each pair of these critical points is~$\geq c_4  \sqrt{\epsilon_0}$ apart (for some uniform constant~$c_4 > 0$) locally in the~$z'$ plane by Lemma~\ref{lem:sqrt_singularity}, where~$\epsilon_0 = N^{-\frac{2}{3} + \frac{1}{100}}$. 


Call~$q_1$ the critical point of~$F_1$ which will give the leading order behavior for the~$q'$ integral; this is the maximizer of the two local minima of~$F_1$ along the compact oval. Similarly, in case there is any ambiguity, call~$q_2$ the relevant critical point of~$F_2$.

We deform the contours to steepest descent contours, as described in~\cite[Section 6.2.2]{BB23}: The curve~$\Gamma_s$ is deformed to a steep descent contour~$\gamma_s$ moving through both locally minimizing points~$q_{1}^{1,-}$ and~$q_{1}^{2,-}$ on the compact oval. We may also arrange that~$\gamma_s$ is vertical in local~$z$ coordinates inside of a~$\delta$-neighborhood of~$z_1$, where we, as before take~$\delta=N^{-\frac{1}{4}}$. Properties of the region~$\{\Re[F_1(q')] \leq \Re[F_1(q_1)]\}$ described in~\cite{BB23} imply that such a contour can be continued globally to a contour of steep descent.

We may deform the~$q$ contour~$\Gamma_l$ in a similar way if~$(\xi_{2, N}, \eta_{2, N})$ is in a facet in region~\eqref{item:fourth}, and otherwise if it is in another region we may deform it as in the previous cases, explained in the proofs of Lemmas~\ref{lem:steepest},~\ref{lem:nearedge_steepest}, and~\ref{lem:edge_steepest}.

Next, we demand that the following condition holds for the contours, in case if the critical points are near each other (say, if they are within~$O(N^{-\frac{1}{4}})$ of each other in some local coordinate). We demand that they intersect at most once inside of~$\mathcal{R}_0$ in the~$\delta = N^{-\frac{1}{4}}$-neighborhood of~$q_1$ (in local~$z$ coordinates), and they are never parallel in this neighborhood. This property actually holds automatically with our chosen contours because the~$q'$ contour~$\gamma_s$ is locally vertical (in the~$z'$ coordinate), and if~$q_2$ is near~$q_1$, then the contour of the~$q$ integral is chosen to be locally at an angle of~$\pi/3$ or~$\pi/4$, as we have seen in the construction of the curves in the previous proofs.

\paragraph{Keeping track of the residues picked up during the deformation:}
As usual, we only pick up a residue from the pole at~$q'=q$. 

 In this case, we note that~$\zeta$ could potentially be on the same compact oval as~$q_1$. This means that to control the single integral,~$\zeta$ can be chosen in this compact oval arbitrarily. This is because the single integral is an integral of a meromorphic (with poles only at angles) one form along a (invariant under conjugation) loop connecting the point~$\zeta$ on the compact oval to the outer oval~$A_0$, and thus it is independent of the choice of~$\zeta$ on the compact oval.

\paragraph{Bounding errors, and computation of the leading order behavior from neighborhoods of critical points:}  
We first explain the local contribution from the region of integration near each critical point. Now if~$(\xi_{2,N},\eta_{2,N})$ is in regime~\eqref{item:first},~\eqref{item:second}, or~\eqref{item:third}, the justification for restricting the region of integration to when~$q$ is a~$\delta$-neighborhood of~$q_2$ is as before, and if it is in a gaseous facet in region~\eqref{item:fourth} then the argument is the same as what we outline below for the~$q'$ integral.

 For the~$q'$ integral, we may, similarly to regions~\eqref{item:first} and~\eqref{item:second}, expand to second order at the critical point~$q_1$ and perform the Gaussian integral. The parts more than~$O(\delta)$ away from the critical point~$q_1$ decay at least as~$O(e^{-c N \delta^2 |F_1''(q_1)|})$, and we have~$|F_1''(q_1)| \gtrsim N^{-\frac{1}{3}+\frac{1}{200}}$. Indeed, at the border of a size~$\delta$ neighborhood in the~$z'$ plane, we have
\begin{align*}
  \Re[  F_1(q') - F_1(q_1) ]&= \Re[-\frac{1}{2} \delta^2 |F_1''(z_1)| \pm \i \delta^3 F_1'''(q_1)+ O(\delta^4)  ] \\
  &= -\frac{1}{2} \delta^2 |F_1''(z_1)| + O(\delta^4)
\end{align*}
and above we have used the fact that~$F_1'''$ is real and our steep descent contour is locally a straight segment in the imaginary direction. Taking~$\delta = N^{-\frac{1}{4}}$ suffices for the quadratic term to dominate the quartic one, and thus get a good bound. 

 Thus, denoting by~$\theta_2$ the local direction of the contour~$\gamma_l$ for~$q$, we may bound the remaining local contribution (the parts of the contours where each variable is in a~$\delta$-neighborhood of its critical point) by 
\begin{equation}\label{eqn:local_cont}
e^{N (\Re[F_1(q_1)] - \Re[F_2(q_2)])} C \int_{-\delta }^{\delta } \int_{-\delta }^{\delta }\frac{e^{-\frac{1}{10} N |F_1''(q_1)| x_1^2}}{|z_1 + \i x_1 - (z_2 +\theta_2 x_2)|} dx_1 dx_2 \leq C' e^{N (\Re[F_1(q_1)] - \Re[F_2(q_2)])} .
\end{equation}

Indeed, if~$q_1$ and~$q_2$ are nearby and~$|z_1 - z_2| \leq \delta = N^{-\frac{1}{4}}$, then due to the fact that~$\frac{1}{z- z'}dz dz'$ is integrable as long as the directions of the contours are not parallel, we obtain~\eqref{eqn:local_cont}.

On the other hand, if~$q_1$ and~$q_2$ are not nearby, then we can upper bound the integrand by~$N^{\frac{1}{4}}$, and then performing the local integrals leads to~\eqref{eqn:local_cont} (with room to spare).

Thus, in all cases we have obtained~\eqref{eqn:fourth_different}.

\paragraph{Bounding the single integral term:} As we said in the secodn step above, if~$\zeta$ is on a compact oval, we have the freedom to choose it. We claim that, regardless of whether we have such a choice, we can upper bound the single integral's contribution by the following quantity:

\begin{equation}\label{eqn:single_int_bd_4}
   C e^{N (\Re[F_1(q_1)]-\Re[F_2(q_2)])}.
\end{equation}
Indeed, this is straightforward if~$\zeta$ is not on the same compact oval as~$q_1$. This is because arguments from the fourth step in the proof of Lemma~\ref{lem:steepest} imply that the single integral contour can be deformed such that the point~$\zeta$ gives the leading contribution for the single integral, and by the choice of contours as steep descent contours we automatically have the bound
\begin{equation}\label{eqn:zeta_triv}
e^{N(F_1(\zeta) - F_2(\zeta))} \leq e^{N (\Re[F_1(q_1)]-\Re[F_2(q_2)])}
\end{equation}
at the intersection point~$\zeta$ so that we may bound the single integral by~\eqref{eqn:single_int_bd_4} and thus also by~\eqref{eqn:fourth_different}. 

If~$\zeta$ is somewhere in the same compact oval as~$q_1$, then, by the definition of the steepest descent curves, one of the local minimum of~$F_1$ on the compact oval is on the steepest descent curve~$\gamma_l$ and both of them are on the steepest descent curve~$\gamma_s$. To satisfy~\eqref{eqn:zeta_triv} we may therefore choose~$\zeta$ as the local minimum of~$F_1$ that lies on booth~$\gamma_s$ and~$\gamma_l$. 
\end{proof}

Finally, we consider the case when both points are in regime~\eqref{item:fourth} and are inside of the same facet. Here we must show a different type of bound, since the single integral might dominate rather than the double integral, as in all of the previous situations.

\begin{figure}
    \centering
\includegraphics[width=0.59\linewidth]{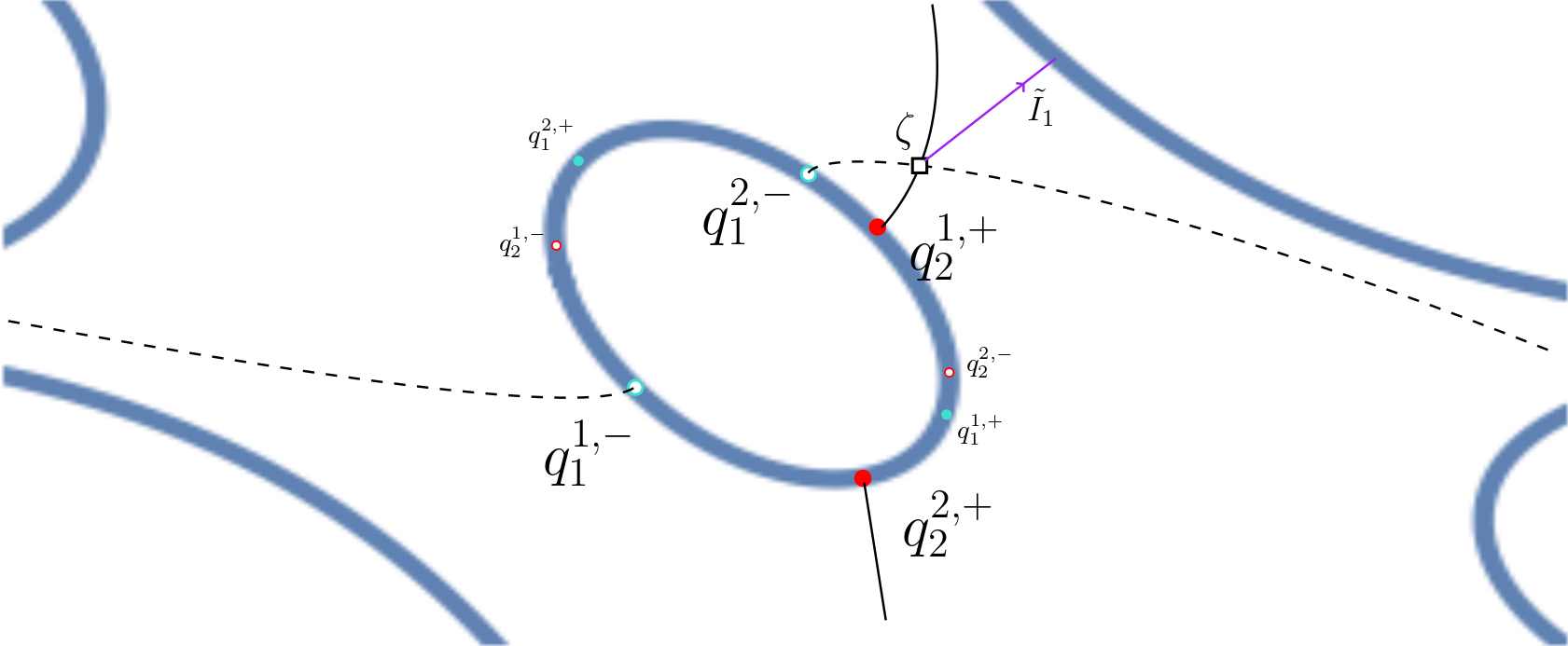}
\\\vspace{20pt}
\includegraphics[width=0.59\linewidth]{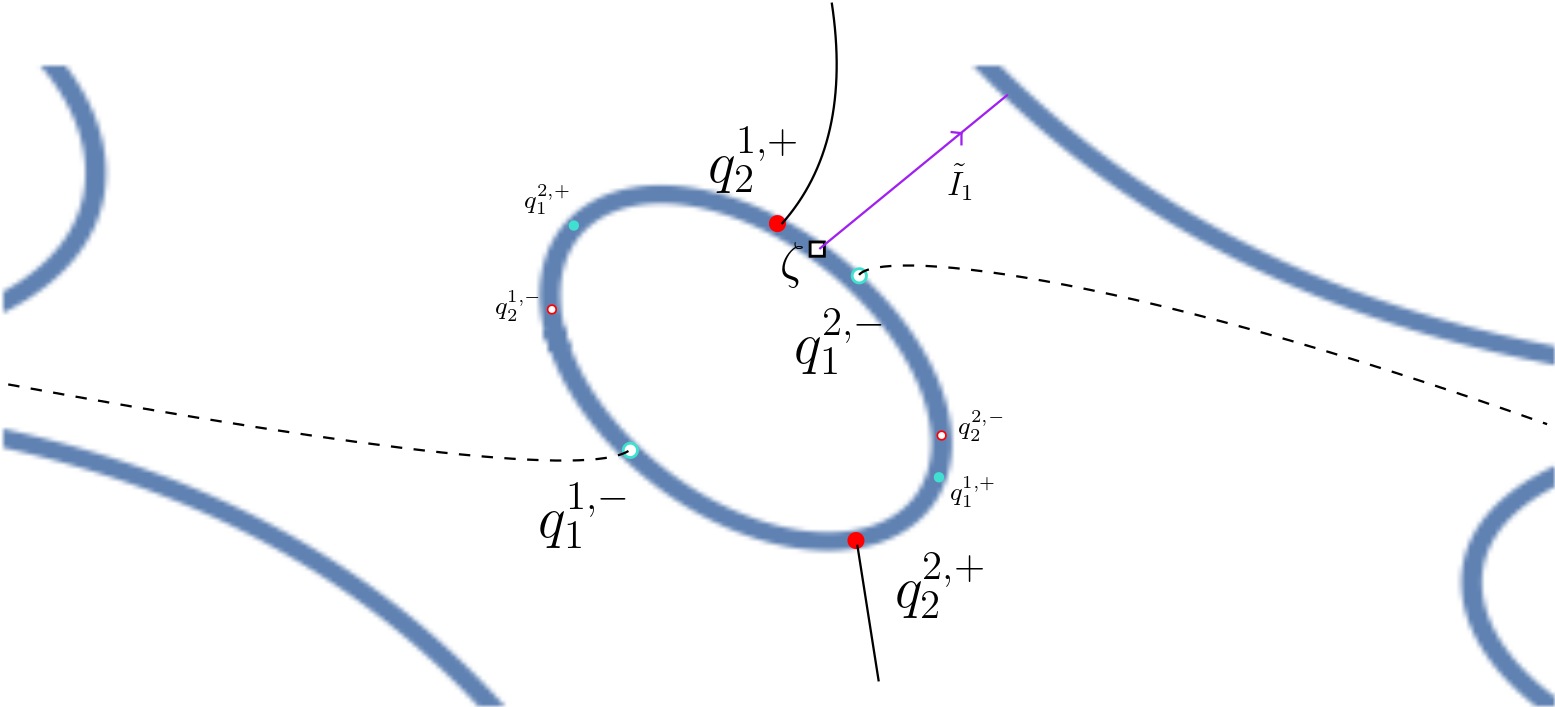}
    \caption{We illustrate the non-interlacing (top), or interlacing (bottom), of local maxima of~$F_2$,~$q_2^{j,+}$, with the local minima of~$F_1$,~$q_1^{j,-}$. In the non-interlacing case the intersection point is away from the compact oval. We also illustrate what the remaining (after contour deformation) single integral~$\tilde{I}_1$ could look like in each case. Note that in the interlacing case, by holomorphicity of the integrand of~$\tilde{I}_1$ and because it is a closed contour, we have the freedom to vary the point~$\zeta$ along the compact oval.}
\label{fig:gas_saddle}
\end{figure}

In the lemma below, both points~$(\xi_{1,N},\eta_{1,N})$ and~$(\xi_{2,N},\eta_{2,N})$ (the rescaled coordinates of the vertices~$(\mathrm b_1,\mathrm w_1)$ and~$(\mathrm b_2,\mathrm w_2)$, respectively) in the Aztec diamond are in a gaseous facet. Thus, on the corresponding compact oval~$A_i$,~$i=1,\dots,g$, we have four critical points for~$F_1$, with local maxima (as the compact oval is traversed) denoted~$q_1^{j,+}$,~$j = 1,2$, and local minima denoted~$q_1^{j,-}$,~$j = 1,2$. Similarly, there are four critical points~$q_2^{j, \pm}$,~$j=1,2$ of~$F_2$. We will distinguish between two different cases: Either the local maxima of~$F_2$,~$q_2^{1,+}$ and~$q_2^{2,+}$, and the local minima of~$F_1$,~$q_1^{1,-}$ and~$q_1^{2,-}$, interlace on the compact oval~$A_i$, as shown in Figure~\ref{fig:gas_saddle}, on the bottom, or they do not, as shown on the top of the same figure.  We also denote~$q_2^+$ as the minimizer of local maxima of~$F_2$, i.e.~$F_2(q_2^+) = \min_{j=1,2}F_2(q_2^{j,+})$, and we denote by~$q_2^-$ the maximizer of local minima of~$F_2$,~$F_2(q_2^-) = \max_{j=1,2}F_2(q_2^{j,-})$. We similarly denote~$q_1^+$ as the minimizer of local maxima of~$F_1$.
\begin{lem}[Steepest descent, both points inside the same gaseous facet]\label{lem:same_facet_steepest}
Suppose both points are in regime~\eqref{item:fourth} and inside of the same gaseous facet. Then, for all~$N$ large enough we have
 \begin{equation}\label{eqn:steepest_samefacet}
K^{-1}(\mathrm b_2,\mathrm w_1)
   = \tilde{I}_1(\mathrm b_2,\mathrm w_1) +  O(e^{N(\Re\left[F_1(q_1^-)- F_1(q_2^+)\right])})
\end{equation}
where~$\tilde{I}_1$ is the single integral remaining after the deformation of contours, and has the following properties: Either the points~$q_1^{1,-}, q_1^{2,-}$ and~$q_2^{1,+},q_2^{2,+}$ \emph{interlace} on the corresponding compact oval, in which case
\begin{equation}\label{eqn:single_int_bound}
|\tilde{I}_1(\mathrm b_2,\mathrm w_1)| \leq C \inf_{\zeta \in A_i} e^{N(\Re\left[F_1(\zeta)- F_2(\zeta)\right])},
\end{equation}
where~$C>0$ is a uniform (in the position of both points) fixed constant; or, they do not interlace, and the single integral is 
\begin{equation}\label{eqn:single_int_bound_non_interlace}
|\tilde{I}_1(\mathrm b_2,\mathrm w_1)|=O(e^{N(\Re\left[F_1(q_1^-)- F_2(q_2^+)\right])}),
\end{equation}
i.e. its size is at most as large as the bound for the double integral term. 

If the critical points interlace, there exists~$\zeta\in A_i$ on the compact oval such that 
\begin{align}\label{eqn:WTS0}
    F_1(\zeta) - F_1(q_1^{+}) \leq 0 \\ 
    -F_2(\zeta) + F_2(q_2^{-}) \leq -N^{-1 + \frac{1}{200}} .
    \label{eqn:WTS1}
\end{align}
Moreover, if there are three points, all in regime~\eqref{item:fourth}, such that the two pairs~$(\mathrm b_2,\mathrm w_1)$ and~$(\mathrm b_3,\mathrm w_2)$ both satisfying the interlacing condition, then, we can choose~$\zeta_1$ and~$\zeta_2$ in the bound~\eqref{eqn:single_int_bound} satisfying~\eqref{eqn:WTS0} and~\eqref{eqn:WTS1}, so that
\begin{equation}\label{eqn:WTS2}
-F_2(\zeta_1) + F_2(\zeta_2) \leq 0.
\end{equation}
\end{lem}
\begin{proof}
We will prove the lemma by analyzing the kernel via a similar sequence of steps as above. We omit the explicit description of each step, and will instead point out the crucial differences in this case.

First, we bound the double integral term. To start, we deform both contours to steep descent contours as described for~$\Gamma_s$ in Lemma~\ref{lem:fourth_dif_facets}. Then, by omitting parts of the contours away from critical points and performing the Gaussian integration to bound the local contribution to the double integral, we will generically be able to derive a bound of the form
\begin{equation*}
|I_2(\mathrm b_2,\mathrm w_1)|\leq  C_5 \frac{1}{N \sqrt{|F_1''(z_1)||F_2''(z_2)|}} e^{N \Re[F_1(q_1)-F_2(q_2)]}.
\end{equation*}
However, to safely account for the case when the two critical points are nearby each other (which can occur if both points are in the facet, even if their rescaled coordinates are far apart in the Aztec diamond), we claim that one can always derive a bound of the form 
\begin{equation}\label{eqn:both_fourth_double_general}
   |I_2(\mathrm b_2,\mathrm w_1)|\leq  C_5 e^{N \Re[F_1(q_1)-F_2(q_2)]}.
\end{equation}
Indeed, to achieve this in the case that~$q_1 \approx q_2$, one may deform the steep descent contours slightly so that their intersections with~$A_i$ occur at a distance of at least~$\sim \frac{1}{\sqrt{N}}$, so that~$\frac{1}{|z_1 - z_2|} \lesssim N^{\frac{1}{2}}$. Then, the local contribution can be bounded by~\eqref{eqn:both_fourth_double_general}.

In contrast to the previous proofs, the single integral is not necessarily small. The intersection point of the contours (which we called~$\zeta$ in the proofs above), which is the starting point for the single integral which is left after the deformation of contours, will now potentially be on a compact oval of~$\partial \mathcal{R}_0$. However, if the points~$q_1^{1,-}, q_1^{2,-}$ and~$q_2^{1,+},q_2^{2,+}$ do not interlace, then~$\zeta$ will not be on a compact oval and we can obtain the bound~\eqref{eqn:single_int_bound_non_interlace} for the single integral term by using a steepest descent contour starting from~$\zeta$ as in the previous lemmas, see for example the fourth step in the proof of Lemma~\ref{lem:fourth_dif_facets}. This situation is depicted in the top image of Figure~\ref{fig:gas_saddle}. 

If the critical points do interlace, as depicted in the bottom image of Figure~\ref{fig:gas_saddle}, then~$\zeta$ is on a compact oval. We remark that there could be several segments of the single integral since, as before, there could be several intersection points; however, we may assume without loss of generality that there is only one intersection point as we have already explained in the proof of Lemma~\ref{lem:steepest}. In this case, the position of~$\zeta$ on this compact oval can be chosen without changing the value of the remaining single integral, since the single integral contour is a closed loop in~$\mathcal{R}$, the homology class of which does not depend on the position of~$\zeta$ on the oval or boundary arc. See Figure~\ref{fig:gas_saddle}, bottom display, where the contour of the single integral~$\tilde{I}_1$ is illustrated in purple.

Now we simply describe how to choose~$\zeta$ such that~\eqref{eqn:WTS0} and~\eqref{eqn:WTS1} can be satisfied in case the critical points interlace in the right way. Indeed, finding such a~$\zeta$ is sufficient, because as before (see step 4 in the proof of Lemma~\ref{lem:steepest}) we may deform the single integral contour so that the largest value of the function in the exponent is at the point~$\zeta$, meaning that this point gives the leading contribution.


Now, recall that we denote~$q_2^+$ as the minimizer of local maxima of~$F_2$, and~$q_1^+$ as the minimizer of local maxima of~$F_1$. Relabel the two maxima for~$F_1$ and~$F_2$ so that~$q_1^+ = q_{1}^{1, +}$ and~$q_2^+ = q_2^{1,+}$ (so the smaller of the two maxima in both cases is the one with upper index~$1$, and the larger with upper index~$2$). 

If~$q_2^+$ and~$q_1^+$ are not in the same connected component of the oval when one removes the points~$q_1^{1,-}$ and~$q_1^{2,-}$, then by the interlacing condition,~$q_2^{2,+}$ and~$q_1^+$ are in the same component (this situation is illustrated in Figure~\ref{fig:gas_saddle}, bottom), in which case we choose~$\zeta = q_2^{2,+}$. The first inequality~\eqref{eqn:WTS0} is satisfied because of~$\zeta$ and~$q_1^+$ being in the same component of the oval after removing~$q_1^{j,-}$,~$j=1,2$; indeed,~$F_1$ is monotonic on each arc~$[q_1^{1,-},q_1^+]$ and~$[q_1^{2,-},q_1^+]$ (these intervals denote arcs which are ``minimal'' in the sense that they do not contain any other critical points of~$F_1$). We postpone showing~\eqref{eqn:WTS1} until after the next paragraph. We remark that if we choose~$\zeta = q_2^{2,+}$, then in the case that we have a second pair of points~$(\mathrm b_3,\mathrm w_2)$ and another intersection point~$\zeta_2 \in A_i$ corresponding to this pair, the  inequality~\eqref{eqn:WTS2} is clearly satisfied (with~$\zeta_1 = \zeta$ and regardless of the choice of~$\zeta_2$), since~$q_2^{2,+}$ is the global maximum of~$F_2$ along the compact oval.
 
In the other case, when~$q_2^{+}$ and~$q_1^+$ are in the same component after we remove~$q_1^{1,-}$ and~$q_1^{2,-}$, we choose~$\zeta = q_2^+$. Again the interlacing implies~\eqref{eqn:WTS0}. 

Now we explain how to get~\eqref{eqn:WTS1} in either of the situations described in the last two paragraphs, that is,~$\zeta$ is one of the two local maxima. If~$(\xi_{2,N}, \eta_{2, N})$ is in a compact subset of the gas region not containing the boundary (independently of~$N$), then, since~$F_2$ has no double critical point,~\eqref{eqn:WTS1} holds for some ($N$ independent) constant. In the case where~$(\xi_{2,N}, \eta_{2, N})$ is near the boundary of the facet, there is a possible problem if~$\zeta\to q_2^-$ as~$N\to \infty$. In that case, we use the bound
\begin{equation}\label{eqn:f2bound}
- F_2(\zeta) + F_2(q_2^{-})\leq -\frac{1}{2} \delta^2  |F_2''(q_2^-)| + \frac{1}{6} \delta^3|F_2'''(q_2^-)| + O(\delta^4)
\end{equation}
for some small enough choice of~$\delta$. By Lemma~\ref{lem:sqrt_singularity} and the definition of regime~\eqref{item:fourth}, we have, in the worst case possible,~$|F_2''(q_2^-)| \sim N^{-\frac{1}{3}+\frac{1}{200}}$ and~$\delta = N^{-\frac{1}{3}}$. Thus, the first term in~\eqref{eqn:f2bound} dominates, and we still obtain~\eqref{eqn:WTS1}.

%

 Now, finally, assume we also have a third point in the same facet whose coordinates give an action function~$F_3$, which in turn has critical points~$q_3^{j,+}$ and~$q_3^{j,-}$ for~$j=1,2$. We have assumed in the lemma that~$q_2^{j,-}$ and~$q_3^{j, +}$,~$j=1,2$, are interlacing (as are the first pair of points). We can assume that for the first pair of points we have chosen~$\zeta = \zeta_1 = q_2^+$, since in the case that~$ \zeta_1 = q_2^{2,+}$, the bound is trivial (as we remarked already earlier in the proof). We must choose~$\zeta_2$, the point on the compact oval for the second pair of points, according to the same rule that we used for~$\zeta_1$, and show that~\eqref{eqn:WTS2} holds. We again relabel the two local maximizers~$q_3^{j,+}$,~$j=1,2$ so that~$q_3^{2,+}$ is the larger one for~$F_3$. For the second pair of points, according to the discussion above we set~$\zeta_2 = q_3^{1,+}$ if~$q_3^{1,+}$ and~$q_2^+$ are in the same component when~$q_2^{1,-}$ and~$q_2^{2,-}$ are removed, and we set~$\zeta_2 = q_3^{2,+}$ in the case that~$q_3^{2,+}$ and~$q_2^+$ are in the same component; i.e. we always choose~$\zeta_2$ as one of~$q_3^{1,+},q_3^{2,+}$, and we pick the one in the same connected component as~$q_2^+ = \zeta_1$. In either case,~\eqref{eqn:WTS2} follows from the fact that~$\zeta_1$ and~$\zeta_2$ are in the same component when the local minima of~$F_2$ are removed, and~$\zeta_1$ is equal to the local maximum~$q_2^+$. 
\end{proof}

We are ready to prove Lemma~\ref{lem:facet_bound}. Below, we say that an edge belongs to a region if the rescaled coordinates of either of its endpoints belongs to that region.
\begin{proof}[Proof of Lemma~\ref{lem:facet_bound}]
We first prove the bound if there is at least one edge in a frozen facet, and then if there is at least one edge in the gaseous facet.

Suppose~$\mathrm e_1$ is in a frozen facet. First, suppose~$\mathrm e_1$ is in a compact subset of the frozen facet, it follows from the analysis in~\cite[Section 6.2.3]{BB23} that~$\mathbb{E}\mathbf{1}\{\mathrm e_1\} = O(e^{-c N})$ or~$1-O(e^{- c N})$ for some~$c > 0$, which can be taken uniformly bounded away from~$0$, as~$\mathrm e_1$ varies over the compact subset of the facet. It follows that
\begin{equation}\label{eqn:varbound}
\var(\mathbf{1}\{\mathrm e_1\}) = O(e^{-c N})
\end{equation}
in either case. From this we can deduce the result with Cauchy-Schwarz.

However, being in region~\eqref{item:fourth} does not guarantee that~$\mathrm e_1$ will remain inside of a fixed compact subset of the frozen facet, so we must improve the precision of our estimate in the case that it approaches the boundary of the facet. When we are near the boundary, we may invoke Lemma~\ref{lem:sqrt_singularity} and perform some computations similar to those in the proof of the previous lemma (see~\eqref{eqn:f2bound} and the discussion immediately after it) to lower bound~$c$ in~\eqref{eqn:varbound} as~$c \gtrsim N^{-1+\frac{1}{200}}$, which is exactly what we need for the result.

Suppose now that at least one edge is in a gaseous facet. The left hand side of~\eqref{eqn:facet_moment_bound} can be expressed by~\eqref{eqn:sigma_term}, and we argue that each permutation in the expansion can be bounded above by~$e^{-N^{\frac{1}{200}}}$. We have to bound terms of the form
\begin{equation}\label{eqn:termtobound}
\left|\prod_{j=1}^{r} K^{-1}(\mathrm b_{\sigma(j)},\mathrm w_j) \right|
\end{equation}
where~$\sigma$ is a permutation of~$\{1,\dots, r\}$ having no fixed points. We may assume that~$\sigma$ consists of a single cycle equal to~$(1 \; 2\; \cdots \; r)$. The general case consists of considering terms with several cycles, which can all be bounded in the same way as we describe for the single cycle, leading to an overall bound for the product of several cycles which is at least as good as the case of a single cycle.

The expression~\eqref{eqn:termtobound} can be expanded as follows after invoking Lemma~\ref{lem:same_facet_steepest} to obtain a bound for each matrix entry: Two terms will contribute to each entry; one is the single integral, and the other is the bound on the double integral term. 

Now we examine a term in the expansion of~\eqref{eqn:termtobound}, where~$\sigma$ is a single cycle. Suppose, for now, that we have chosen the single integral term for all of the first, say,~$k-1$ factors in~\eqref{eqn:termtobound} and that all of these pairs of points are in region~\eqref{item:fourth} and satisfy the interlacing condition as in Lemma~\ref{lem:same_facet_steepest}, and that we have chosen the other error term for the remaining terms. The single integrals are upper bounded by~$C e^{N(F_j(\zeta_j) - F_{j+1}(\zeta_j))}$, where~$\zeta_j$ can be chosen arbitrarily on the corresponding compact oval. By Lemma~\ref{lem:same_facet_steepest} we can choose~$\{\zeta_j\}_{j=1}^{k-1}$ in such a way that we have
\begin{equation}\label{eqn:zetabounds}
- F_{j}(\zeta_{j-1}) + F_j(\zeta_j) \leq 0, \qquad j = 2,\dots, k-1,
\end{equation}
and such that
\begin{align}
N(-F_{k}(\zeta_{k-1}) +F_k(q_k^-)) &\leq -N^{\frac{1}{200}}, \label{eqn:Fkkbd}\\
N(F_1(\zeta_1) - F_1(q_1^+)) \leq 0 .\label{eqn:F11bd}
\end{align}
Now, by the steepest descent lemmas, see especially Lemmas~\ref{lem:same_facet_steepest} and~\ref{lem:fourth_dif_facets}, for large enough~$C > 0$, we have
\begin{equation}\label{eqn:termtobound2}
\left|\prod_{j=1}^{r} K^{-1}(\mathrm b_{\sigma(j)},\mathrm w_j) \right|
\leq C \prod_{j=1}^{k-1} e^{N (F_j(\zeta_j) - F_{j+1}(\zeta_{j}))} \prod_{j=k}^{r}e^{N(F_j(q_j^-) - F_{j+1}(q_{j+1}^+))}.
\end{equation}
Above we interpret the index~$j$ to move cyclically in~$1,\dots, r$. We have also denoted by~$q_j^+$ and~$q_j^-$ the critical point of~$F_j$ which is relevant for the corresponding entry of~$K^{-1}$; if~$(\xi_{j,N},\eta_{j,N})$, the rescaled coordinates of the black vertex, is in one of the regions~\eqref{item:first},~\eqref{item:second} or~\eqref{item:third}, these two points are the same point,~$q_j$, and in case the point is in region~\eqref{item:fourth} in a gaseous facet, they are different, and correspond to a local maximum and minimum along the corresponding compact oval, respectively. In particular, we always have~$F_j(q_j^-) - F_j(q_j^+) \leq 0$. Therefore, by that observation, and by~\eqref{eqn:zetabounds},~\eqref{eqn:Fkkbd},  and~\eqref{eqn:F11bd}, we get from~\eqref{eqn:termtobound2} that
\begin{multline}\label{eqn:single_cyclefinalbound}
\left|\prod_{j=1}^{r} K^{-1}(\mathrm b_{\sigma(j)},\mathrm w_j) \right|
\leq C e^{N (F_1(\zeta_1) - F_1(q_1^+))}e^{N (F_k(q_k^-) - F_k(\zeta_{k-1}))} \\
\times \prod_{j=2}^{k-1} e^{N (F_j(\zeta_j) - F_j(\zeta_{j-1}))} \prod_{j=k+1}^{r}e^{N(F_j(q_j^-) - F_j(q_j^+))}\leq C' e^{-N^{\frac{1}{200}}}.
\end{multline}

Recall that for the first~$k-1$ factors of~\eqref{eqn:termtobound} we chose the single integral term, and we assumed that all of these pairs of points satisfied the interlacing condition. If, instead, for one of those we either choose the double integral, it is in another region, or the pair of points does not satisfy the interlacing from Lemma~\ref{lem:same_facet_steepest}, then we can consider maximal strings of indices along which the single integral has been chosen and the interlacing is satisfied. For each such string we can apply~\eqref{eqn:zetabounds}, and then apply~\eqref{eqn:Fkkbd} and~\eqref{eqn:F11bd} at the two endpoints of the string to arrive at a bound of the form~\eqref{eqn:single_cyclefinalbound} (coming just from one such maximal string).

Then, for an arbitrary permutation one bounds each cycle as above, and thus we obtain the desired bound for all terms in the expansion of~\eqref{eqn:termtobound} with at least one double integral term or at least one single integral term either corresponding to a pair of points which are not both in the same facet, or corresponding to a pair of critical points in the same facet which are not interlaced correctly.

The only remaining case to consider is when~$k-1=r$ and there is a term consisting entirely of single integrals, for pairs of points all of which are all inside of the same facet and interlaced correctly. Recall that in this case we have assumed a lower bound on distances between vertices, so that they are far apart, in particular at a distance~$>> N^{\frac{1}{200}}$, from each other at lattice scale. In this remaining case, we need to bound the term consisting of all single integrals by something of order~$O(e^{-N^{\frac{1}{200}}})$. One way to do this is to note that these single integral terms exactly correspond to the entries of the inverse Kasteleyn of the corresponding gaseous translation invariant Gibbs measure. Thus, the desired bound can be obtained from the fact that after an appropriate gauge, this inverse Kasteleyn decays exponentially in the distance between vertices, see~\cite[Section 4.5]{KOS06}. In fact, one way to show this is to argue that after an appropriate gauge, one can choose a~$\zeta$ in each expression of the form~\eqref{eqn:single_int_bound} such that it is decaying exponentially.
\end{proof}

\appendix
\section{Removing the technical assumptions on edge weights}
\label{app:A}
\begin{proof}[Proof of Lemma~\ref{lem:double_int_simple} without technical restrictions on edge weights]

    We may remove the condition 4.1 (b) of~\cite{BB23} on the edge weights as follows. The entries of~$K^{-1}$ are rational in the edge weights. Furthermore, so is the formula in the lemma, when both the double integral and single integral are interpreted as sums of residues at angles; in fact, even when the spectral curve becomes singular, such a formula in terms of residues still makes sense, and remains analytic.
    

    The set of~$\{\alpha_{i, j}, \beta_{i, j}, \gamma_{i, j}\}$ parameterizing the space of~$k \times \ell$ periodic edge weights for which the statement of the lemma is true includes the open set (in~$\mathbb{R}^{3 k \ell}$)  satisfying 4.1 (b) in~\cite{BB23}. By the previous paragraph, the left and right hand sides in the lemma are rational functions of~$\{\alpha_{i, j}, \beta_{i, j}, \gamma_{i, j}\}$; therefore, by analytic continuation, the set of weights for which the equality holds includes all weights satisfying Assumption~\ref{ass:1}.
    
\end{proof}

\section{Review of cumulants}
\label{app:B}

Here we briefly review the relationship between joint moments and joint cumulants and state several facts relevant to Section~\ref{sec:GFF_DC} of the text; we refer the reader to the appendix on classical cumulants in~\cite{NS06} for a more formal discussion. The relationship between moments and cumulants can be understood through the analytic framework of generating functions or the combinatorial framework of \emph{set partitions}. A set partition of a set $[n] = \{1, 2, \dots, n\}$ is a partition of the set into non-overlapping, non-empty subsets. 

\subsection{Joint moments and joint cumulants}\label{subapp:jmjc}
For a collection of random variables~$X = (X_1, X_2, \dots, X_m)$, the joint moment generating function is defined by
\begin{equation}\label{eqn:MGFdef}
    M_{X}(t) = \mathbb{E}\left[e^{t \cdot X}\right],
\end{equation}
where $t = (t_1, t_2, \dots, t_m)$.

The joint cumulant generating function is defined by
\begin{equation}
    K_{X}(t) = \log M_{X}(t),
\end{equation}
and derivatives of $K_{X}(t)$ yield joint cumulants:
\begin{equation}\label{eqn:cdefs}
\kappa_{n_1,\dots,n_m} = \frac{\partial^{n_1 + n_2 + \cdots + n_m}}{\partial t_1^{n_1} \partial t_2^{n_2} \cdots \partial t_m^{n_m}} K_{X}(t) \bigg|_{t = 0}.
\end{equation}
In particular, we define

\begin{equation}\label{eqn:cdefs2}
\kappa[X_1,\dots, X_m] \coloneqq  \frac{\partial^{m}}{\partial t_1 \partial t_2 \cdots \partial t_m} K_{X}(t) \bigg|_{t = 0} = \kappa_{1,\dots,1}.
\end{equation}

It is not difficult to see that
\begin{equation}
    \label{eqn:cdefs3}
\kappa[\underbrace{X_1, X_1, \dots, X_1}_{n_1 \text{ times}}, \underbrace{X_2, X_2, \dots, X_2}_{n_2 \text{ times}}, \dots, \underbrace{X_m, \dots, X_m}_{n_m \text{ times}}] = \kappa_{n_1,\dots,n_m},
\end{equation}
so in general it suffices to work with~$\kappa[X_1,\dots, X_m]$ as defined in~\eqref{eqn:cdefs2}.

We note a few simple facts:
\begin{itemize}
\item The set of 
 cumulants (assuming they are sufficiently nice) uniquely determine a probability distrubtion.
 \item Cumulants of order ~$1$ are expectations, $\kappa[X] = \mathbb{E}[X]$.
    \item The cumulants of order ~$2$ agree with centered moments, $\kappa[X, Y] = \Cov(X, Y)$.
    
    \item Higher order cumulants~$\kappa[X_1,\dots, X_m]$,~$m\geq 2$, do not depend on the mean of~$(X_1,\dots, X_m)$.
\end{itemize}

A uniquely characterizing property of joint cumulants is the following expression for moments in terms of cumulants; for any mean zero~$(X_1,\dots, X_m)$,
\begin{equation}\label{eqn:momcumap}
    \mathbb{E}[X_1 X_2 \cdots X_m] = \sum_{\pi \in \mathcal{P}([m])} \prod_{B \in \pi} \kappa\left[X_i :\; i \in B\right] .
\end{equation}

\begin{remark}\label{rmk:welldef}
    Consider the collection of equations given by~\eqref{eqn:momcumap} with the variables replaced by subsets of~$(X_1,\dots, X_m)$. Inverting these equations to compute joint cumulants in terms of moments, one is never required to compute a joint moment with a the same random variable repeated more than once. This fact is important when we deal with cumulants corresponding to joint moments of the Gaussian free field: Even when the joint moment generating function is not well-defined, we may use~\eqref{eqn:momcumap} to define cumulants (without repeated variables) in terms of moments.
\end{remark}

\subsection{Joint cumulants of Gaussians}

Consider a mean zero Gaussian vector~$(X_1,\dots,X_n)$. This implies that for any distinct~$i_1,\dots,i_k$, letting~$Y_1=X_{i_1},\dots,Y_{r}=X_{i_{r}}$ for~$r =2 k$ or~$r=2 k+1$, we have
\begin{equation}\label{eqn:wickform1}
\begin{cases}
 \mathbb{E}[Y_{1} \cdots Y_{2k}] = \sum_{p} \prod_{j=1}^{k}  \mathbb{E}[Y_{p(2 j-1)} Y_{p(2 j)} ]  & r=2k \text{ even}\\
  \mathbb{E}[Y_{1} \cdots Y_{2k+1}] = 0 & \text{else}
 \end{cases}
\end{equation}
where the sum in the first case is over all distinct pairings~$p$ of~$\{1,\dots,2 k\}$. Furthermore, moments of the form~\eqref{eqn:wickform1} but with an odd number of~$Y_j$'s are zero.

It is also possible to see, by computing the generating function of a multivariate Gaussian, that for any distinct~$i_1,\dots,i_k$, letting~$Y_1=X_{i_1},\dots,Y_k=X_{i_k}$
\begin{equation}\label{eqn:wickform2}
\begin{cases}
    \kappa[Y_{1}, \cdots, Y_{k}]  =0 & k > 2\\
   \kappa[Y_{1}, Y_{2}]  = \mathbb{E}[Y_1 Y_2] & k =2
    \end{cases}
\end{equation}

The above, together with the defining relation~\eqref{eqn:momcumap} and Remark~\ref{rmk:welldef}, imply the following.
\begin{lem}\label{lem:wickformlem}
    A tuple of random variables~$(X_1,\dots, X_n)$ has joint cumulants (with no repeated variables) satisfying~\eqref{eqn:wickform2} if and only if its joint moments (with no repeated variables)  satisfy~\eqref{eqn:wickform1}. This equivalence holds even in the case that joint moments of~$(X_1,\dots,X_n)$ with repeated variables are infinite or not well-defined.
\end{lem}

\subsection{Independence and joint cumulants}

A key property of joint cumulants, which we use in Proposition~\ref{prop:ind_moments}, is their behavior under independence. 
\begin{lem}\label{lem:indvars}
    The random variables~$(X_1, X_2, \dots, X_{m_1})$ and~$(Y_1, Y_2, \dots, Y_{m_2})$ satisfy 
\begin{equation}\label{eqn:indvars}
        \mathbb{E}\left[ \prod_{i=1}^{m_1}X_i^{n_i} \prod_{i=1}^{m_2}Y_i^{n_i} \right] = \mathbb{E}\left[ \prod_{i=1}^{m_1}X_i^{n_i} \right] \mathbb{E}\left[ \prod_{i=1}^{m_2}Y_i^{n_i'} \right]
    \end{equation}
   for any nonnegative~$n_i , n_i' \geq 0$ if and only if any joint cumulant involving variables from both~$(X_1,\dots,X_{m_1})$ and~$(Y_1,\dots,Y_{m_2})$ is zero.
\end{lem}

\bibliographystyle{alpha}
\bibliography{bibliotek}

\end{document}